\documentclass[11pt,reqno]{amsart}

\usepackage[T1]{fontenc}
\usepackage[utf8]{inputenc}
\usepackage[a4paper,margin=1in]{geometry}
\usepackage{amsmath,amssymb,amsthm,mathtools}
\usepackage{enumitem}
\usepackage{mathrsfs}
\usepackage{tikz}
\usetikzlibrary{arrows.meta,calc,positioning,decorations.pathreplacing}
\usepackage[colorlinks=true,linkcolor=blue,citecolor=blue,urlcolor=blue]{hyperref}

\numberwithin{equation}{section}

\theoremstyle{plain}
\newtheorem{theorem}{Theorem}[section]
\newtheorem{proposition}[theorem]{Proposition}
\newtheorem{lemma}[theorem]{Lemma}
\newtheorem{corollary}[theorem]{Corollary}

\theoremstyle{definition}
\newtheorem{definition}[theorem]{Definition}
\newtheorem{example}[theorem]{Example}

\theoremstyle{remark}
\newtheorem{remark}[theorem]{Remark}

\newcommand{\R}{\mathbb R}
\newcommand{\N}{\mathbb N}
\newcommand{\Nzero}{\mathbb N_0}
\newcommand{\Z}{\mathbb Z}
\newcommand{\Zmod}[1]{\mathbb Z/#1\mathbb Z}
\newcommand{\DeltaM}{\Delta^{m-1}}
\newcommand{\inte}{\operatorname{int}}
\newcommand{\bd}{\partial}
\newcommand{\co}{\operatorname{co}}

\newcommand{\Fix}{\operatorname{Fix}}
\newcommand{\dist}{\operatorname{dist}}

\newcommand{\id}{\operatorname{id}}

\newcommand{\x}{\mathbf x}
\newcommand{\y}{\mathbf y}

\newcommand{\p}{\mathbf p}
\newcommand{\e}{\mathbf e}
\newcommand{\f}{\mathfrak f}

\newcommand{\Leb}{\operatorname{Leb}}

\newcommand{\htop}{h_{\mathrm{top}}}
\newcommand{\Prob}{\mathcal P}
\newcommand{\Lip}{\operatorname{Lip}}
\newcommand{\BL}{\mathrm{BL}}

\title{Full-measure historic behavior and emergence in simplex dynamics}

\author[F. Mukhamedov]{Farrukh Mukhamedov}
\address{Department of Mathematical Sciences, College of Science, United Arab Emirates University, P.O. Box 15551, Al Ain, Abu Dhabi, UAE}
\email{far75m@yandex.ru; farrukh.m@uaeu.ac.ae }
\subjclass[2020]{37A05, 37B20, 37C70, 37C45, 37B40, 92D25}
\keywords{Historic behavior; non-statistical dynamics; pointwise emergence; Takens' last problem; Stein--Ulam map; simplex dynamical system; Lotka--Volterra operator; empirical measures; weak cyclic cover}

\begin{document}

\begin{abstract}
We give a finite-dimensional boundary-residence mechanism for Takens-type
historic behavior.  For a class of continuous self-maps of a simplex, a
multiplicative Lyapunov drift toward the boundary, combined with cyclic
residence near separated boundary regions, forces every non-fixed interior
orbit to have non-convergent Ces\`aro averages and non-convergent empirical
measures.  Thus the historic set has full relative Lebesgue measure.  The
mechanism is independent of hyperbolic specification, wandering domains and
positive entropy: in fact, for the maps considered here the topological entropy
is carried entirely by the boundary restriction.  The non-convergence persists
under all positive powers and, more generally, under multiplicatively thick
sampling times.  The abstract criterion applies to Stein--Ulam and
Lotka--Volterra type stochastic operators once the model-specific cyclic coding
and logarithmic residence estimates are verified.  In explicit cyclic-collar
models we identify the full weak-* accumulation set of empirical measures; this
accumulation set is a non-trivial circle of probability measures and yields
polynomial pointwise emergence of order \(\varepsilon^{-1}\).  We also obtain
examples with positive-dimensional omega-limit sets by combining the mechanism
with a Bara\'nski--Misiurewicz type boundary-tracing construction.
\end{abstract}
\maketitle

\section{Introduction}

Takens' last problem asks for mechanisms which produce orbits with
non-existent time averages on sets of positive Lebesgue measure
\cite{Takens2008}.  In modern terminology these are orbits with
\emph{historic behavior}: for some continuous observable, the corresponding
Birkhoff averages fail to converge.  The problem is not merely to construct one
exceptional orbit, but to explain how deterministic finite-dimensional dynamics
can display statistical irregularity on a large set of initial conditions.
Known positive answers are often tied to homoclinic tangencies, wandering
domains, heteroclinic networks or related mechanisms
\cite{KirikiSoma2017,LabouriauRodrigues2017}.  A quantitative form of the same
question is the theory of emergence, where one measures the complexity of the
accumulation set of empirical statistics; see, in particular, the viewpoint of
Kiriki, Nakano and Soma~\cite{KNS22}.

The aim of this paper is to give a different finite-dimensional mechanism.  It
is a boundary-residence mechanism in simplex dynamics.  Let
\[
        \Delta^{m-1}
        =
        \left\{
        \x=(x_1,\ldots,x_m)\in\R^m:
        x_i\ge0,\ \sum_{i=1}^m x_i=1
        \right\}
\]
be the standard simplex.  We consider continuous maps
\(K:\Delta^{m-1}\to\Delta^{m-1}\) with a distinguished interior fixed point
\(\p\).  The structural assumptions are deliberately geometric.  A
multiplicative Lyapunov function drives every non-stationary interior orbit
toward the boundary; a weak cyclic cover organizes the boundary motion into a
monotone lifted itinerary; logarithmic residence estimates force the orbit to
spend increasingly long epochs near prescribed boundary regions; and a
convex-hull separation condition prevents a single Ces\`aro limit from being
compatible with all residence regions.

The main consequence is a full-measure historic-behavior theorem.  Under the
weak admissibility hypotheses introduced below, every point of
\(\operatorname{int}\Delta^{m-1}\setminus\{\p\}\) has non-convergent vector
Ces\`aro averages,
\[
        \frac1n\sum_{k=0}^{n-1}K^k(x),
\]
and non-convergent empirical measures,
\[
        \frac1n\sum_{k=0}^{n-1}\delta_{K^k(x)}.
\]
Thus the historic set has full relative Lebesgue measure in the simplex.  The
phenomenon is stable under all positive powers of the map and under
multiplicatively thick sampling of the orbit.  In this sense the paper gives a
Takens-type mechanism inside the boundary-preserving simplex category: the
large historic set is produced by slow cyclic residence near separated boundary
regions, not by a hyperbolic horseshoe, a specification argument or a wandering
domain.

The mechanism is also designed to connect with the classical Stein--Ulam
problem.  Nonlinear stochastic maps on simplices, especially quadratic and
Lotka--Volterra type operators, arise in population genetics, evolutionary
dynamics and nonlinear Markov processes.  The continuous-time Lotka--Volterra
mechanism goes back to Lotka and Volterra
\cite{Lotka1920,Volterra1926}; for broader nonlinear population models see
\cite{GoelMaitraMontroll1971}, and for evolutionary and population-dynamical
background see Hofbauer and Sigmund~\cite{HofbauerSigmund1998}.  For
finite-dimensional stochastic operators we refer to Lyubich's monograph
\cite{Lyubich1992} and to the survey literature on quadratic stochastic
operators, for example~\cite{Ganikhodjaev2009}.  The Stein--Ulam map and its
Lotka--Volterra variants are discrete simplex models in which an interior
coexistence state coexists with a cyclic drift toward the boundary.

The original Stein--Ulam problem is connected with Ulam's questions on
nonlinear transformations of the simplex \cite{Ulam1960,SteinUlam1964}.
Zakharevich proved a fundamental failure of the ergodic hypothesis for
quadratic mappings of a simplex \cite{Zakharevich1978}, and the
\(\omega\)-limit structure of the Stein--Ulam spiral map was studied by
Bara\'nski and Misiurewicz~\cite{BaranskiMisiurewicz2010}.  More recent work
introduced Stein--Ulam spiral type stochastic operators and proved
non-ergodicity and historic behavior for several classes of Lotka--Volterra
maps; see \cite{GGJ,JL,JamilovMukhamedov2022,JamilovMukhamedov2024,S23}.  The
closest predecessor is the SUS-t framework of Jamilov and
Mukhamedov~\cite{JamilovMukhamedov2024}, where historic behavior and its
persistence under positive powers were proved for a concrete Stein--Ulam spiral
class.  The present paper takes a different point of view: it extracts the
abstract Lyapunov-cyclic mechanism behind such examples and proves consequences
which do not depend on a particular stochastic-operator formula.

The main theorem is proved by a block argument.  First, a multiplicative
Lyapunov identity
\[
        \varphi(K(x))=\varphi(x)\psi(x),
        \qquad 0\le \psi\le1,
\]
with equality only at the interior fixed point \(\p\), forces every
non-fixed interior orbit to escape compact subsets of the interior.  Second,
after the orbit leaves a fixed neighborhood of \(\p\), the weak cyclic cover
allows one to choose a lifted itinerary \((r_n)\) satisfying
\[
        K^n(x)\in G_{r_n\,({\rm mod}\,N)},
        \qquad r_{n+1}-r_n\in\{0,1\},
        \qquad r_n\to+\infty .
\]
This weak rule allows overlaps of adjacent sectors but prevents trapping in a
proper part of the cyclic cover.  Third, the logarithmic residence estimate
turns boundary approach into residence blocks whose lengths are comparable with
the previous history of the orbit.  Finally, if a vector Ces\`aro average
converged to a limit \(L\), the long blocks near each residence region \(R_j\)
would force
\[
        L\in \operatorname{co}(R_j)
        \qquad\text{for every }j.
\]
The separation condition
\[
        \bigcap_j \operatorname{co}(R_j)=\varnothing
\]
then gives a contradiction.

This proof has several consequences which are useful for positioning the
simplex examples within the general theory of non-statistical dynamics.  First,
the historic behavior persists for every positive power \(K^s\).  A long
ordinary residence block always contains a long arithmetic subblock of times
congruent to \(0\pmod s\), so the same convex-hull obstruction applies to the
sampled orbit.  Second, the same idea works for multiplicatively thick sampling
times: every linearly long block contains sampled times with positive relative
frequency.  Third, all invariant probability measures are supported on the
boundary together with the distinguished fixed point.  Consequently the
topological entropy of a weak admissible map is the entropy of its boundary
restriction.  Thus the full-measure historic behavior produced here is not
explained by entropy generated in the interior.

The explicit cyclic-collar models go beyond non-convergence.  For those models
we compute the entire weak-* accumulation set of empirical measures.  The limit
set is a one-parameter family
\[
        \{\eta_\tau:\tau\in\mathbb T\}
\]
of probability measures supported on a boundary circle.  The parameter is not a
formal artifact: the map \(\tau\mapsto\eta_\tau\) is bi-Lipschitz for the
bounded-Lipschitz metric on probability measures.  Hence the pointwise
emergence of each collar orbit is polynomial of order
\(\varepsilon^{-1}\).  This gives a concrete quantitative form of the
non-existence of averages, complementary to the emergence viewpoint of
\cite{KNS22}.

The simplex framework also contains natural examples from population dynamics.
In many cyclic-dominance models the boundary supports a heteroclinic cycle
connecting pure states, while an interior fixed point represents coexistence.
The classical May--Leonard three-species model~\cite{ML75} is a continuous-time
prototype: trajectories may successively approach neighborhoods of single
species and spend longer times near the boundary.  The present paper treats a
discrete-time simplex analogue.  The long residence epochs are used not merely
to describe the geometry of the orbit, but to force non-convergence of Ces\`aro
averages and empirical measures by a convex-geometric obstruction.

We also formulate a general \(\p\)-Stein--Ulam Lotka--Volterra verification
scheme.  In this scheme the weighted arithmetic--geometric mean inequality
supplies the Lyapunov structure automatically, while the cyclic coding and
residence estimates remain explicit model-dependent hypotheses to be checked.
The paper further shows that the mechanism is stable under sufficiently small
boundary-preserving conjugacies, producing weak admissible examples outside the
Lotka--Volterra class, and combines the construction with a
Bara\'nski--Misiurewicz type tracing condition to obtain
positive-dimensional omega-limit sets.

The scope of the result should be understood precisely.  We do not claim to
solve the unrestricted smooth \(C^r\) form of Takens' last problem for arbitrary
maps, diffeomorphisms or flows.  The weak admissibility hypotheses include
boundary invariance, a cyclic cover and a residence estimate, and a later
Baire-category argument shows that weak admissibility is not residual in the
full \(C^r\) space of all simplex self-maps.  The contribution is instead an
explicit and verifiable structural class of finite-dimensional
boundary-preserving maps for which historic behavior has full relative Lebesgue
measure and persists under natural operations.

The paper is organized as follows.  Section~2 fixes notation, historic behavior
and pointwise emergence.  Sections~3--4 introduce weak admissibility and prove
basic Lyapunov and block properties.  Section~5 proves the main full-measure
non-statistical theorem.  Sections~6--7 prove persistence under positive powers
and multiplicatively thick subsequences.  Sections~8--10 discuss stability and
concrete cyclic-collar families, including the explicit empirical accumulation
and emergence computation.  Section~11 gives the \(\p\)-Stein--Ulam
Lotka--Volterra verification scheme.  The remaining sections treat entropy,
non-Lotka--Volterra conjugate examples, positive-dimensional omega-limit sets
and Baire-category issues.

\section{Notation, historic behavior and emergence}

Throughout, \(m\ge2\) is fixed and
\[
        \DeltaM
        =
        \left\{
        \x=(x_1,\ldots,x_m)\in\R^m:
        x_i\ge0,\ \sum_{i=1}^m x_i=1
        \right\}.
\]
The relative interior and boundary of \(\DeltaM\) are denoted by
\(\inte\DeltaM\) and \(\bd\DeltaM\).  The relative \((m-1)\)-dimensional
Lebesgue measure on \(\DeltaM\) is denoted by \(\Leb_{m-1}\).  For a map
\(T:\DeltaM\to\DeltaM\), we write
\[
        A_n(T)(\x)=\frac{1}{n}\sum_{k=0}^{n-1}T^k(\x).
\]
For \(E\subset\R^m\), \(\co(E)\) denotes its convex hull.  Since all ambient
spaces are finite-dimensional, the convex hull of a compact set is compact.

\begin{definition}
Let \(X\) be a compact metric space, let \(T:X\to X\), and let
\(g:X\to\R^d\) be continuous.  The orbit of \(x\in X\) has  \textit{historical
behavior} or \textit{non-statistical behavior} for the observable \(g\) if the averages
\[
        \frac{1}{n}\sum_{k=0}^{n-1}g(T^k x)
\]
do not converge in \(\R^d\).  When \(g\) is the identity map on a simplex, this
is called vector-valued historical behavior.
\end{definition}

\begin{definition}
Let \(X\) be a compact metric space and let \(T:X\to X\) be continuous.  The
empirical measures of the orbit of \(x\in X\) are
\[
        \mu_n^x
        =
        \frac1n\sum_{k=0}^{n-1}\delta_{T^k x}.
\]
We say that the orbit of \(x\) is non-statistical if \((\mu_n^x)_{n\ge1}\)
does not converge in the weak-* topology.  If \(X\) carries a reference measure
\(\Leb\), we say that \(T\) is non-statistical if the set of non-statistical
points has full \(\Leb\)-measure.  A probability measure \(\mu\) is a physical
measure if its basin
\[
        B(\mu)=\{x\in X:\ \mu_n^x\to\mu\}
\]
has positive reference measure.
\end{definition}

\begin{definition}
\label{def:pointwise-emergence}
Let \((X,d_X)\) be a compact metric space.  Denote by \(\Prob(X)\) the space of
Borel probability measures on \(X\).  We use the bounded-Lipschitz metric
\[
        d_{\BL}(\mu,\nu)
        =
        \sup
        \left\{
        \left|\int f\,d\mu-\int f\,d\nu\right|:
        \|f\|_\infty\le1,\ \Lip(f)\le1
        \right\}.
\]
It induces the weak-* topology on \(\Prob(X)\).  For the empirical measures
\(\mu_n^x\), let
\[
        \mathcal V_T(x)
        =
        \bigcap_{N\ge1}
        \overline{\{\mu_n^x:n\ge N\}}^{\,w^*}
\]
be their weak-* accumulation set.  The pointwise emergence of the orbit of
\(x\) at scale \(\varepsilon>0\) is
\[
        \mathscr E_x(\varepsilon)
        =
        \min\left\{q:\ \exists\nu_1,\ldots,\nu_q\in\Prob(X)
        \text{ such that }
        \limsup_{n\to\infty}
        \min_{1\le i\le q}d_{\BL}(\mu_n^x,\nu_i)
        \le\varepsilon
        \right\}.
\]
Equivalently, \(\mathscr E_x(\varepsilon)\) is the minimum cardinality of a
closed \(\varepsilon\)-net for \(\mathcal V_T(x)\) in the metric \(d_{\BL}\).
Indeed, every subsequential limit of \((\mu_n^x)\) belongs to \(\mathcal V_T(x)\),
and the distance from \(\mu_n^x\) to \(\mathcal V_T(x)\) tends to zero.
\end{definition}

\begin{remark}
Vector-valued historical behavior for the identity observable implies scalar
historical behavior for at least one coordinate observable.  However, in
general the coordinate may depend on the initial point.  Therefore a
full-measure vector statement gives a positive-measure scalar statement for at
least one coordinate by the pigeonhole principle, but it does not by itself
identify a single coordinate that is irregular on a full-measure set.
\end{remark}

For \(N\in\N\), write \(\Zmod{N}\) for the cyclic group.  A subset
\(I\subset\Zmod{N}\) is called a cyclic interval if, after choosing a cyclic
order, it is represented by consecutive residues.  More explicitly, after
choosing integers \(\alpha\le\beta\), one may write
\[
        I=\{\alpha,\alpha+1,\ldots,\beta\}\pmod N.
\]
A lift of \(I\) to the \(q\)-th turn around the cycle is then
\[
        \widetilde I_q=\{\alpha+qN,\alpha+qN+1,\ldots,\beta+qN\}\subset\Z.
\]

\section{Weak admissible Stein--Ulam maps}

\begin{definition}\label{def:residence-block}
Let \(R\subset\DeltaM\), let \(K:\DeltaM\to\DeltaM\), and let
\(\x\in\DeltaM\).  A pair \((b,\ell)\in\Nzero\times\N\) is called a
residence block through \(R\) along the orbit of \(\x\) if
\[
        K^{b+r}(\x)\in R,
        \qquad 0\le r\le \ell-1.
\]
If, in addition, \(b\ge1\), \(K^{b-1}(\x)\notin R\), and
\(K^{b+\ell}(\x)\notin R\), then the block is a genuine entrance--exit block.
The arguments below only require residence blocks; this convention avoids any
ambiguity caused by overlaps of adjacent compact sectors.
\end{definition}

\begin{definition}\label{def:lifted-block}
Let \(G_0,\ldots,G_{N-1}\) be a cyclic compact cover, let
\(I\subset\mathbb Z/N\mathbb Z\) be a cyclic interval, and let
\((r_n)_{n\ge n_0}\) be an integer lift of an orbit with
\(K^n(\x)\in G_{r_n\,({\rm mod}\,N)}\).  A pair
\((b,\ell)\in\Nzero\times\N\), with \(b\ge n_0\), is called a lifted residence
block through \(I\) if, for some lifted copy \(\widetilde I_q\subset\mathbb Z\),
\[
        r_{b},r_{b+1},\ldots,r_{b+\ell-1}\in \widetilde I_q.
\]
When \(R=\bigcup_{i\in I}G_i\), every lifted residence block through \(I\) is a
residence block through \(R\).  Maximal lifted blocks are obtained by taking all
consecutive times for which the lift remains in a fixed copy
\(\widetilde I_q\).
\end{definition}

\begin{definition}\label{def:weak-admissible}
A continuous map \(K:\DeltaM\to\DeltaM\) is called a \textit{weak admissible
Stein--Ulam type map} if the following data and conditions are given.

\begin{enumerate}[label=\textup{(W\arabic*)},leftmargin=3.2em]
\item \textbf{Interior fixed point and face invariance.}
There is a point \(\p\in\inte\DeltaM\) such that
\[
        K(\inte\DeltaM)\subset\inte\DeltaM,
        \qquad
        K(\bd\DeltaM)\subset\bd\DeltaM,
\]
and
\[
        \Fix(K)\cap\inte\DeltaM=\{\p\}.
\]

\item \textbf{Multiplicative Lyapunov structure.}
There are continuous functions
\[
        \varphi:\DeltaM\to[0,\infty),
        \qquad
        \psi:\DeltaM\to[0,1]
\]
such that
\[
        \varphi^{-1}(0)=\bd\DeltaM,
        \qquad
        \varphi(\x)>0\quad(\x\in\inte\DeltaM),
\]
\[
        \varphi(K(\x))=\varphi(\x)\psi(\x),
        \qquad \x\in\DeltaM,
\]
\[
        \psi^{-1}(1)=\{\p\},
\]
and \(\varphi\) has a strict unique maximum at \(\p\):
\[
        \varphi(\x)<\varphi(\p),
        \qquad \x\in\DeltaM\setminus\{\p\}.
\]

\item \textbf{Weak cyclic compact cover outside the center.}
There are an open neighborhood \(U_0\) of \(\p\), an integer \(N\ge m\), and
compact sets
\[
        G_i\subset\DeltaM\setminus U_0,
        \qquad i\in\mathbb Z/N\mathbb Z,
\]
such that
\[
        \overline{U_0}\subset\inte\DeltaM,
        \qquad
        \DeltaM\setminus U_0=\bigcup_{i\in\mathbb Z/N\mathbb Z}G_i .
\]
The sets \(G_i\) are allowed to overlap.  This is essential in applications,
where neighboring sectors usually share boundary arcs or faces.

\item \textbf{Residence regions and convex-hull separation.}
Let
\[
        I_1,\ldots,I_m\subset\mathbb Z/N\mathbb Z
\]
be non-empty cyclic intervals whose union is \(\mathbb Z/N\mathbb Z\).  Define
\[
        R_j=\bigcup_{i\in I_j}G_i,
        \qquad j=1,\ldots,m.
\]
Each \(R_j\) is a non-empty compact set.  We require
\[
        \bigcap_{j=1}^m\co(R_j)=\varnothing.
        \tag{W4}\label{eq:convhull-separation}
\]
The regions \(R_j\) are not required to be convex or pairwise disjoint.

\item \textbf{Orbitwise lifted cyclic progress.}
For every
\[
        \x\in\inte\DeltaM\setminus\{\p\},
\]
once the orbit is outside \(U_0\) forever, there is a monotone lifted itinerary
which makes infinitely many turns around the cycle.  Precisely, if \(n_0\) is
such that
\[
        K^n(\x)\notin U_0,
        \qquad n\ge n_0,
\]
then one can choose integers \(r_n\), \(n\ge n_0\), such that
\[
        K^n(\x)\in G_{r_n\,({\rm mod}\,N)},
        \qquad
        r_{n+1}-r_n\in\{0,1\},
        \qquad n\ge n_0,
\]
and
\[
        r_n\longrightarrow+\infty.
\]
The lifted itinerary is part of the weak admissible data along the orbit.  It
replaces a single-valued sector code and therefore remains meaningful when
adjacent compact sectors overlap.

\item \textbf{Logarithmic residence estimate.}
The lifted itineraries in \textup{(W5)} can be chosen so that, for each
\(j\in\{1,\ldots,m\}\), there are constants
\[
        A_j>0,
        \qquad B_j>0,
        \qquad C_j>1,
\]
with the following property.  Every maximal lifted residence block
\((b,\ell)\) through \(I_j\) along every non-fixed interior orbit satisfies
\[
        \ell
        \ge
        A_j\log_{C_j}^{+}
        \frac{B_j}{\varphi(K^{b}(\x))},
        \tag{W6}\label{eq:log-residence}
\]
where
\[
        \log_C^+(t)=\max\{0,\log_C t\}.
\]
Since such a block is contained in \(R_j\), this is a residence-time estimate at
entrance time \(b\).
\end{enumerate}
\end{definition}

\begin{remark}
Condition \textup{(W3)} supplies the finite cyclic compact cover outside the
central neighborhood \(U_0\).  Condition \textup{(W5)} is the orbitwise
non-trapping rule: after leaving \(U_0\), the orbit admits a lifted itinerary
which can only stay in the present label or advance by one label, and this lift
tends to \(+\infty\).  Thus the orbit crosses every lifted copy of every
cyclic interval \(I_j\).  These crossings are the lifted residence blocks used
in Lemma~\ref{lem:blocks}.
\end{remark}

\begin{remark}
The definition deliberately does not require a single-valued sector map.  A
single-valued finite code with compact disjoint fibers would not model the
closed sector covers that occur in Stein--Ulam examples.  The compact-cover
formulation keeps the sector sets closed while preserving the monotone lifted
itinerary needed for the block argument.
\end{remark}

\subsection*{Schematic figures for the weak admissible construction}

The following figures are purely schematic.  They show the construction in the
triangle \(\Delta^2\), where the geometry is visible.  In higher-dimensional
simplices the same information is encoded by the cyclic compact cover
\(G_0,\ldots,G_{N-1}\), the orbitwise lifted itineraries, the residence regions
\(R_1,\ldots,R_m\), and the convex-hull separation condition
\eqref{eq:convhull-separation}.

\begin{figure}[htbp]
\centering
\begin{tikzpicture}[scale=0.86,>=Stealth,every node/.style={font=\small}]
    \coordinate (e1) at (0,0);
    \coordinate (e2) at (6,0);
    \coordinate (e3) at (3,5.15);
    \coordinate (p)  at (3,1.72);

    \fill[gray!3] (e1)--(e2)--(e3)--cycle;
    \draw[thick] (e1)--(e2)--(e3)--cycle;

    \fill[blue!18] (0.18,0.18)--(1.45,0.18)--(1.85,0.75)--(1.12,1.28)--(0.43,0.78)--cycle;
    \draw[blue!65!black,thick] (0.18,0.18)--(1.45,0.18)--(1.85,0.75)--(1.12,1.28)--(0.43,0.78)--cycle;
    \node[blue!70!black] at (0.90,0.58) {$R_1$};

    \fill[green!18] (4.55,0.18)--(5.82,0.18)--(5.58,0.82)--(4.88,1.28)--(4.15,0.75)--cycle;
    \draw[green!50!black,thick] (4.55,0.18)--(5.82,0.18)--(5.58,0.82)--(4.88,1.28)--(4.15,0.75)--cycle;
    \node[green!45!black] at (5.10,0.58) {$R_2$};

    \fill[red!18] (2.45,4.28)--(3,5.00)--(3.55,4.28)--(3.55,3.58)--(2.45,3.58)--cycle;
    \draw[red!65!black,thick] (2.45,4.28)--(3,5.00)--(3.55,4.28)--(3.55,3.58)--(2.45,3.58)--cycle;
    \node[red!70!black] at (3,4.10) {$R_3$};

    \fill[yellow!30] (p) circle (0.42);
    \draw[orange!80!black,thick] (p) circle (0.42);
    \node[orange!70!black] at (3,1.72) {$U_0$};
    \fill (p) circle (1.5pt);
    \node[below right=1pt of p] {$\p$};

    \draw[->,very thick,blue!60!black]
        (1.45,1.05) .. controls (1.85,2.20) and (2.35,3.10) .. (2.67,3.62);
    \draw[->,very thick,red!70!black]
        (3.53,3.60) .. controls (4.20,2.75) and (4.65,2.05) .. (4.90,1.22);
    \draw[->,very thick,green!50!black]
        (4.15,0.62) .. controls (3.05,0.18) and (2.00,0.18) .. (1.42,0.62);

    \draw[dashed,gray!60] (p)--(e1);
    \draw[dashed,gray!60] (p)--(e2);
    \draw[dashed,gray!60] (p)--(e3);

    \node[below left] at (e1) {$\e_1$};
    \node[below right] at (e2) {$\e_2$};
    \node[above] at (e3) {$\e_3$};

    \node[align=center] at (3,-0.65)
        {weak cyclic rule: stay in the present code or move to the next one};
\end{tikzpicture}
\caption{Weak cyclic coding and residence regions.  Outside the center
\(U_0\), the code can stay unchanged or advance by one.  The residence regions
\(R_j\) are the portions of the cyclic code corresponding to prescribed cyclic
intervals.}
\label{fig:weak-cyclic-coding}
\end{figure}
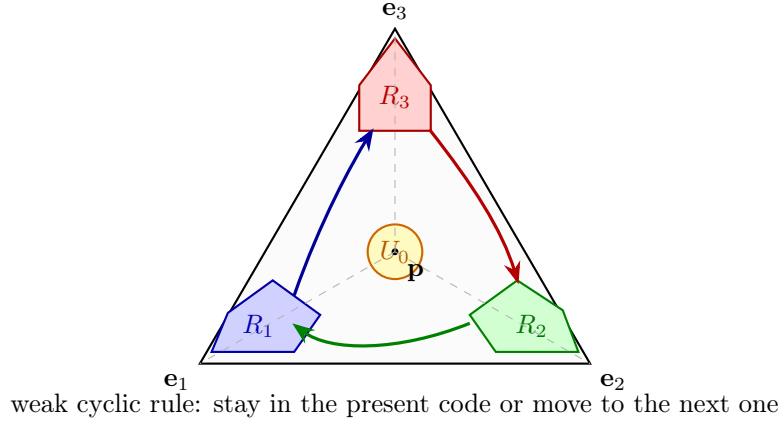

\begin{figure}[htbp]
\centering
\begin{tikzpicture}[x=0.62cm,y=0.72cm,>=Stealth,every node/.style={font=\small}]
    \draw[->,thick] (0,0)--(17,0) node[right] {time};

    \foreach \x in {0,1,2,3,4,5,6,7,8,9,10,11,12,13,14,15,16} {
        \draw (\x,0.08)--(\x,-0.08);
        \fill[gray!45] (\x,0) circle (1.2pt);
    }

    \fill[blue!15] (5,-0.35) rectangle (12,0.35);
    \draw[blue!65!black,thick] (5,-0.35) rectangle (12,0.35);
    \foreach \x in {5,6,7,8,9,10,11,12} {
        \fill[blue!70!black] (\x,0) circle (1.9pt);
    }

    \node[below] at (5,-0.12) {$b$};
    \node[above] at (12,0.48) {$b+\ell-1$};

    \draw[decorate,decoration={brace,amplitude=5pt},thick]
        (5,0.62)--(12,0.62)
        node[midway,above=7pt] {$\ell$ consecutive iterates in $R_j$};

    \node[align=center] at (8.5,-1.25)
        {$K^{b+r}(x)\in R_j$ for $0\le r\le\ell-1$};

    \draw[->,thick,red!70!black] (2,1.45)--(5,0.45);
    \node[align=center,red!70!black] at (2.2,1.75)
        {first time};
    \draw[->,thick,red!70!black] (14.6,1.45)--(13,0.45);
    \node[align=center,red!70!black] at (14.6,1.75)
        {last time};
\end{tikzpicture}
\caption{A residence block through a residence region.  The
logarithmic residence estimate turns the small value of
\(\varphi(K^b(x))\) near the boundary into a lower bound for the length
\(\ell\) of the block.}
\label{fig:entrance-exit-block}
\end{figure}
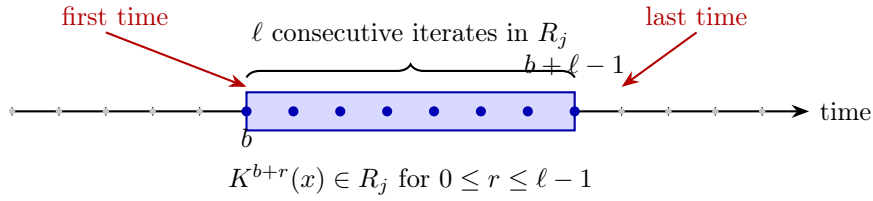

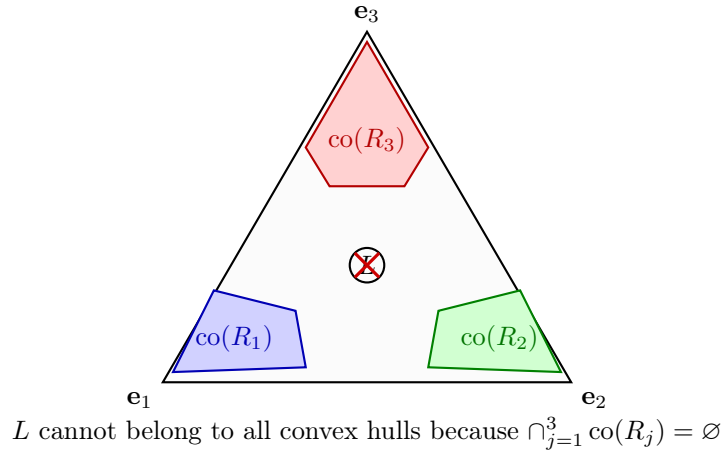
\begin{figure}[htbp]
\centering
\begin{tikzpicture}[scale=0.90,>=Stealth,every node/.style={font=\small}]
    \coordinate (e1) at (0,0);
    \coordinate (e2) at (6,0);
    \coordinate (e3) at (3,5.15);
    \fill[gray!3] (e1)--(e2)--(e3)--cycle;
    \draw[thick] (e1)--(e2)--(e3)--cycle;

    \fill[blue!18,draw=blue!70!black,thick]
        (0.15,0.15)--(2.10,0.22)--(1.95,1.05)--(0.75,1.35)--cycle;
    \node[blue!70!black] at (1.05,0.65) {$\co(R_1)$};

    \fill[green!18,draw=green!50!black,thick]
        (3.90,0.22)--(5.85,0.15)--(5.25,1.35)--(4.05,1.05)--cycle;
    \node[green!45!black] at (4.95,0.65) {$\co(R_2)$};

    \fill[red!18,draw=red!70!black,thick]
        (2.10,3.45)--(3,5.00)--(3.90,3.45)--(3.55,2.88)--(2.45,2.88)--cycle;
    \node[red!70!black] at (3,3.55) {$\co(R_3)$};

    \node[draw,circle,inner sep=1.2pt,thick] (L) at (3,1.72) {$L$};
    \draw[red!80!black,very thick] ($(L)+(-0.18,-0.18)$)--($(L)+(0.18,0.18)$);
    \draw[red!80!black,very thick] ($(L)+(-0.18,0.18)$)--($(L)+(0.18,-0.18)$);

    \node[align=center] at (3,-0.75)
        {$L$ cannot belong to all convex hulls because $\cap_{j=1}^3\co(R_j)=\varnothing$};

    \node[below left] at (e1) {$\e_1$};
    \node[below right] at (e2) {$\e_2$};
    \node[above] at (e3) {$\e_3$};
\end{tikzpicture}
\caption{The convex-hull obstruction.  Long residence blocks force any
possible Ces\`aro limit to lie in each \(\co(R_j)\).  The defining separation
condition rules out such a common limit.}
\label{fig:convex-hull-obstruction}
\end{figure}

\begin{figure}[htbp]
\centering
\begin{tikzpicture}[scale=0.78,>=Stealth,every node/.style={font=\small}]
    \coordinate (A1) at (0,0);
    \coordinate (B1) at (4.2,0);
    \coordinate (C1) at (2.1,3.65);
    \coordinate (P1) at (2.1,1.22);
    \fill[gray!3] (A1)--(B1)--(C1)--cycle;
    \draw[thick] (A1)--(B1)--(C1)--cycle;
    \fill[yellow!30] (P1) circle (0.24);
    \draw[orange!80!black] (P1) circle (0.24);
    \node at (P1) {$\p$};
    \draw[->,thick,blue!70!black]
        (0.55,0.45) .. controls (1.10,1.35) and (1.40,2.20) .. (1.90,2.85);
    \draw[->,thick,blue!70!black]
        (2.25,2.92) .. controls (3.05,2.30) and (3.45,1.25) .. (3.65,0.52);
    \draw[->,thick,blue!70!black]
        (3.05,0.34) .. controls (2.20,0.10) and (1.35,0.10) .. (0.80,0.34);
    \node at (2.1,-0.55) {$K$};

    \draw[->,very thick] (4.95,1.75)--(6.15,1.75)
        node[midway,above] {$h$};

    \coordinate (A2) at (7,0);
    \coordinate (B2) at (11.2,0);
    \coordinate (C2) at (9.1,3.65);
    \coordinate (P2) at (9.1,1.22);
    \fill[gray!3] (A2)--(B2)--(C2)--cycle;
    \draw[thick] (A2)--(B2)--(C2)--cycle;
    \fill[yellow!30] (P2) circle (0.24);
    \draw[orange!80!black] (P2) circle (0.24);
    \node at (P2) {$\p_h$};

    \fill[blue!70!black] (A2) circle (1.7pt);
    \fill[blue!70!black] ($(A2)!0.15!(B2)$) circle (1.7pt);
    \draw[->,blue!70!black,thick] ($(A2)+(0.18,0.12)$)--($(A2)!0.15!(B2)+(0,0.12)$);

    \draw[->,thick,green!50!black]
        (7.35,0.55) .. controls (8.10,1.55) and (8.10,2.55) .. (8.85,3.00);
    \draw[->,thick,green!50!black]
        (9.45,2.95) .. controls (10.00,2.15) and (10.50,1.40) .. (10.65,0.52);
    \draw[->,thick,green!50!black]
        (10.15,0.34) .. controls (9.15,0.08) and (8.20,0.10) .. (7.65,0.34);
    \node at (9.1,-0.55) {$K_h=h^{-1}\circ K\circ h$};

    \node[align=center] at (5.6,-1.35)
        {small boundary-preserving conjugacy transports the weak admissible data};
\end{tikzpicture}
\caption{Stability under small boundary-preserving conjugacy.  The center,
Lyapunov structure, cyclic code and residence blocks are transported by
\(h\).  Smallness of \(h\) preserves the convex-hull separation.}
\label{fig:boundary-conjugacy}
\end{figure}

\begin{lemma}\label{lem:separation-criterion}
Let \(0<\eta<1-1/m\), and let
\[
        V_j(\eta)=\{\x\in\DeltaM:x_j\ge1-\eta\},
        \qquad j=1,\ldots,m.
\]
If compact sets \(R_j\) satisfy \(R_j\subset V_j(\eta)\) for all \(j\), then
\[
        \bigcap_{j=1}^m\co(R_j)=\varnothing.
\]
\end{lemma}

\begin{proof}
The set \(V_j(\eta)\) is convex.  Hence \(R_j\subset V_j(\eta)\) implies
\(\co(R_j)\subset V_j(\eta)\).  If
\(\x\in\cap_{j=1}^m\co(R_j)\), then \(x_j\ge1-\eta\) for all \(j\).  Therefore
\[
        1=\sum_{j=1}^m x_j\ge m(1-\eta)>1,
\]
which is impossible.
\end{proof}

\section{Basic properties}

\begin{lemma}\label{lem:escape}
Let \(K\) be a weak admissible Stein--Ulam type map.  Then, for every
\(\x\in\inte\DeltaM\setminus\{\p\}\),
\[
        \varphi(K^n(\x))\longrightarrow0.
\]
Consequently,
\[
        \omega_K(\x)\subset\bd\DeltaM.
\]
In particular, since \(\overline{U_0}\subset\inte\DeltaM\), there exists
\(n_0=n_0(\x)\) such that
\[
        K^n(\x)\notin U_0,
        \qquad n\ge n_0.
\]
\end{lemma}

\begin{proof}
Put \(\x_n=K^n(\x)\) and \(u_n=\varphi(\x_n)\).  From the multiplicative
identity and \(0\le\psi\le1\),
\[
        u_{n+1}=u_n\psi(\x_n)\le u_n.
\]
Thus \((u_n)\) is non-increasing and converges to a number \(L\ge0\).

Assume, for contradiction, that \(L>0\).  Since \(\DeltaM\) is compact, the
omega-limit set \(\omega_K(\x)\) is non-empty.  Choose
\(\y\in\omega_K(\x)\).  Then there is a subsequence \(n_k\to\infty\) such that
\(\x_{n_k}\to\y\).  By continuity,
\[
        \varphi(\y)=\lim_{k\to\infty}\varphi(\x_{n_k})=L.
\]
Moreover, \(K\y\in\omega_K(\x)\), because
\(\x_{n_k+1}=K(\x_{n_k})\to K\y\).  Hence also
\[
        \varphi(K\y)=L.
\]
Using the Lyapunov identity at \(\y\), we get
\[
        L=\varphi(K\y)=\varphi(\y)\psi(\y)=L\psi(\y).
\]
Since \(L>0\), this gives \(\psi(\y)=1\).  By \textup{(W2)},
\(\psi^{-1}(1)=\{\p\}\), so \(\y=\p\).  Therefore
\[
        L=\varphi(\y)=\varphi(\p).
\]
But \(\x\ne\p\) and \(\varphi\) has a strict unique maximum at \(\p\), so
\[
        L\le u_0=\varphi(\x)<\varphi(\p),
\]
a contradiction.  Hence \(L=0\).

Every accumulation point \(\y\) of the orbit satisfies
\(\varphi(\y)=\lim\varphi(\x_{n_k})=0\), and therefore
\(\y\in\varphi^{-1}(0)=\bd\DeltaM\).  Thus \(\omega_K(\x)\subset\bd\DeltaM\).
Finally, \(\varphi\) is positive on the compact set \(\overline{U_0}\).  Since
\(\varphi(\x_n)\to0\), the orbit can visit \(U_0\) only finitely many times.
\end{proof}

\begin{lemma}
\label{lem:blocks}
Let \(K:\DeltaM\to\DeltaM\) be a weak admissible Stein--Ulam type map.
Then, for every
\(
        \x\in\inte\DeltaM\setminus\{\p\}
\)
and every \(j\in\{1,\ldots,m\}\), there exist sequences
\[
        b_{j,q}\to\infty,
        \qquad
        \ell_{j,q}\in\mathbb N,
\]
such that
\[
        K^{b_{j,q}+r}(\x)\in R_j,
        \qquad
        0\le r\le \ell_{j,q}-1.
\]
Moreover, the pairs \((b_{j,q},\ell_{j,q})\) are maximal lifted residence
blocks through \(I_j\) for an admissible lifted itinerary supplied by
\textup{(W5)--(W6)}.  In particular, every non-fixed interior orbit has
infinitely many residence blocks through each residence region \(R_j\).
\end{lemma}

\begin{proof}
Fix
\(
        \x\in\inte\DeltaM\setminus\{\p\}.
\)
By Lemma~\ref{lem:escape}, the orbit of \(\x\) eventually stays outside the
central neighborhood \(U_0\).  Hence there exists \(n_0\in\mathbb N\) such that
\[
        K^n(\x)\notin U_0,
        \qquad
        n\ge n_0.
\]
Choose an admissible lifted itinerary \((r_n)_{n\ge n_0}\) as in
\textup{(W5)--(W6)}.  Thus
\[
        K^n(\x)\in G_{r_n\,({\rm mod}\,N)},
        \qquad
        r_{n+1}-r_n\in\{0,1\},
        \qquad
        r_n\to+\infty.
\]

Now fix \(j\in\{1,\ldots,m\}\).  Since \(I_j\subset \Zmod{N}\) is a non-empty
cyclic interval, choose integers \(\alpha_j\le \beta_j\) with
\[
        0\le \beta_j-\alpha_j<N
\]
such that \(I_j\) is represented by the residue classes
\[
        \alpha_j,\alpha_j+1,\ldots,\beta_j
        \pmod N.
\]
For each \(q\in\mathbb Z\), define
\[
        \widetilde I_{j,q}
        =
        \{\alpha_j+qN,\alpha_j+qN+1,\ldots,\beta_j+qN\}.
\]
Since \(r_n\to+\infty\) and \(r_{n+1}-r_n\in\{0,1\}\), the lifted itinerary
cannot jump over any integer value.  Therefore, for all sufficiently large
\(q\), the set
\[
        E_{j,q}
        =
        \{n\ge n_0:\ r_n\in \widetilde I_{j,q}\}
\]
is non-empty.  It is finite and is an interval of integers, because \((r_n)\) is
non-decreasing and the lifted interval \(\widetilde I_{j,q}\) is finite.  Thus
we may write
\[
        E_{j,q}=\{b_{j,q},b_{j,q}+1,\ldots,b_{j,q}+\ell_{j,q}-1\},
\]
where \(\ell_{j,q}\ge1\).  This is precisely a maximal lifted residence block
through \(I_j\).

For every \(0\le r\le\ell_{j,q}-1\), the congruence class of
\(r_{b_{j,q}+r}\) belongs to \(I_j\).  Since
\[
        K^{b_{j,q}+r}(\x)
        \in
        G_{r_{b_{j,q}+r}\,({\rm mod}\,N)}
        \subset
        \bigcup_{i\in I_j}G_i
        =R_j,
\]
we obtain a residence block through \(R_j\).  Finally, the lifted intervals
\(\widetilde I_{j,q}\) drift to \(+\infty\), so \(b_{j,q}\to\infty\).  This
proves the lemma.
\end{proof}

\begin{proposition}
\label{prop:linear-residence}
Let \(K\) be a weak admissible Stein--Ulam type map.  Let
\(
        \x\in\inte\DeltaM\setminus\{\p\},
\)
and let
\(
        (b_{j,q},\ell_{j,q})_{q\ge1}
\)
be the lifted residence blocks through \(R_j\) obtained in
Lemma~\ref{lem:blocks}.  Then
\[
        \liminf_{q\to\infty}
        \frac{\ell_{j,q}}{b_{j,q}}>0.
\]
\end{proposition}

\begin{proof}
By Lemma~\ref{lem:escape}, the orbit of \(\x\) eventually stays outside the
central neighborhood \(U_0\).  Thus there exists \(n_0\in\mathbb N\) such that
\[
        K^n(\x)\notin U_0,
        \qquad n\ge n_0.
\]
Since \(\p\in U_0\), the compact set \(\DeltaM\setminus U_0\) does not contain
\(\p\).  By continuity of \(\psi\), together with
\[
        0\le \psi\le1,
        \qquad
        \psi^{-1}(1)=\{\p\},
\]
we have
\[
        \theta:=\max_{\DeltaM\setminus U_0}\psi<1.
\]
Choose \(\rho\in(\theta,1)\).  Then
\[
        \psi(\y)\le \rho,
        \qquad \y\in\DeltaM\setminus U_0.
\]

Since \(b_{j,q}\to\infty\) by Lemma~\ref{lem:blocks}, we have
\(b_{j,q}\ge n_0\) for all sufficiently large \(q\).  For such \(q\),
\[
\begin{aligned}
        \varphi(K^{b_{j,q}}(\x))
        &=
        \varphi(K^{n_0}(\x))
        \prod_{n=n_0}^{b_{j,q}-1}
        \psi(K^n(\x))                                      \\
        &\le
        \rho^{b_{j,q}-n_0}
        \varphi(K^{n_0}(\x)).
\end{aligned}
\]
The number \(\varphi(K^{n_0}(\x))\) is positive because the orbit remains in the
interior.

Applying the logarithmic residence estimate \eqref{eq:log-residence} to the
lifted block \((b_{j,q},\ell_{j,q})\), and then using the displayed upper bound,
we get, for all large \(q\),
\[
\begin{aligned}
        \ell_{j,q}
        &\ge
        A_j
        \log_{C_j}^{+}
        \frac{B_j}{\varphi(K^{b_{j,q}}(\x))}                  \\
        &\ge
        A_j
        \log_{C_j}
        \frac{B_j}{\rho^{b_{j,q}-n_0}\varphi(K^{n_0}(\x))}    \\
        &=
        A_j
        \left(
        (b_{j,q}-n_0)\log_{C_j}\frac1\rho
        +
        \log_{C_j}\frac{B_j}{\varphi(K^{n_0}(\x))}
        \right).
\end{aligned}
\]
Here we used \(\log_C^+(t)\ge\log_C(t)\) for \(t>0\).  Dividing by
\(b_{j,q}\) and taking lower limits gives
\[
        \liminf_{q\to\infty}
        \frac{\ell_{j,q}}{b_{j,q}}
        \ge
        A_j\log_{C_j}\frac1\rho
        >0.
\]
\end{proof}

\begin{lemma}
\label{lem:block}
Let \(X\) be compact and convex in a finite-dimensional normed space.
Let \(T:X\to X\) be a map, let \(Y\subset X\) be compact, and let \(x\in X\).
Assume that
\[
        A_n(T)(x):=
        \frac1n\sum_{k=0}^{n-1}T^k(x)
        \longrightarrow L.
\]
Assume also that there exist sequences
\[
        b_n\in\mathbb N_0,
        \qquad
        r_n\in\mathbb N,
\]
such that
\[
        b_n\to\infty,
\]
\[
        T^{b_n+q}(x)\in Y,
        \qquad
        0\le q\le r_n-1,
\]
and
\[
        \liminf_{n\to\infty}\frac{r_n}{b_n}>0.
\]
Then
\[
        L\in \co(Y).
\]
\end{lemma}

\begin{proof}
For every \(n\), define the average over the \(n\)-th block by
\[
        B_n
        =
        \frac1{r_n}
        \sum_{q=0}^{r_n-1}T^{b_n+q}(x).
\]
Since every summand belongs to \(Y\), we have
\[
        B_n\in \co(Y).
\]

Put
\[
        \lambda_n
        =
        \frac{r_n}{b_n+r_n}.
\]
Since
\[
        \liminf_{n\to\infty}\frac{r_n}{b_n}>0,
\]
there exists \(\lambda_0>0\) such that
\[
        \lambda_n\ge \lambda_0
\]
for all sufficiently large \(n\).  Indeed, if \(r_n/b_n\ge c>0\) for all large
\(n\), then
\[
        \lambda_n
        =
        \frac{r_n/b_n}{1+r_n/b_n}
        \ge
        \frac{c}{1+c}.
\]

For all sufficiently large \(n\), \(b_n\ge1\), and the Ces\`aro average at the
end of the block decomposes as
\[
\begin{aligned}
        A_{b_n+r_n}(T)(x)
        &=
        \frac1{b_n+r_n}
        \sum_{k=0}^{b_n+r_n-1}T^k(x)                         \\
        &=
        \frac{b_n}{b_n+r_n}
        \left(
        \frac1{b_n}\sum_{k=0}^{b_n-1}T^k(x)
        \right)
        +
        \frac{r_n}{b_n+r_n}
        \left(
        \frac1{r_n}\sum_{q=0}^{r_n-1}T^{b_n+q}(x)
        \right)                                               \\
        &=
        (1-\lambda_n)A_{b_n}(T)(x)+\lambda_n B_n .
\end{aligned}
\]
Hence
\[
        B_n-L
        =
        \frac{A_{b_n+r_n}(T)(x)-L}{\lambda_n}
        -
        \frac{1-\lambda_n}{\lambda_n}
        \bigl(A_{b_n}(T)(x)-L\bigr).
\]
Since
\[
        b_n\to\infty
        \qquad\text{and}\qquad
        b_n+r_n\to\infty,
\]
the assumed convergence of \(A_n(T)(x)\) gives
\[
        A_{b_n}(T)(x)\to L,
        \qquad
        A_{b_n+r_n}(T)(x)\to L.
\]
Moreover, the coefficients
\[
        \frac1{\lambda_n},
        \qquad
        \frac{1-\lambda_n}{\lambda_n}
\]
are bounded for all sufficiently large \(n\), because \(\lambda_n\ge\lambda_0>0\).
Therefore
\[
        B_n\to L.
\]

Since \(Y\) is compact and the ambient space is finite-dimensional, its convex
hull \(\co(Y)\) is compact, hence closed.  Because
\[
        B_n\in\co(Y)
        \qquad\text{and}\qquad
        B_n\to L,
\]
we conclude that
\[
        L\in\co(Y).
\]
\end{proof}

\section{Non-statistical behavior}

\begin{theorem}\label{thm:main-divergence}
Let \(K\) be a weak admissible Stein--Ulam type map on \(\DeltaM\).  For
every
\(
        \x\in\inte\DeltaM\setminus\{\p\},
\)
the vector Ces\`aro averages
\[
        A_n(K)(\x)=\frac{1}{n}\sum_{k=0}^{n-1}K^k(\x)
\]
do not converge.  Equivalently, \(K\) has vector-valued historical behavior for
the identity observable at every non-fixed interior point.
\end{theorem}

\begin{proof}
Assume, for contradiction, that \(A_n(K)(\x)\to L\).  Fix
\(j\in\{1,\ldots,m\}\).  By Lemma~\ref{lem:blocks}, there are residence blocks
\((b_{j,q},\ell_{j,q})\) through \(R_j\), with \(b_{j,q}\to\infty\).  By
Proposition~\ref{prop:linear-residence},
\[
        \liminf_{q\to\infty}\frac{\ell_{j,q}}{b_{j,q}}>0.
\]
Set
\[
        r_{j,q}=\ell_{j,q}.
\]
Then
\[
        K^{b_{j,q}+r}(\x)\in R_j,
        \qquad 0\le r\le r_{j,q}-1,
\]
and
\[
        \liminf_{q\to\infty}\frac{r_{j,q}}{b_{j,q}}>0.
\]
Applying Lemma~\ref{lem:block} with \(T=K\) and \(Y=R_j\), we obtain
\[
        L\in\co(R_j).
\]
Since \(j\) was arbitrary,
\[
        L\in\bigcap_{j=1}^m\co(R_j),
\]
which contradicts the separation condition \eqref{eq:convhull-separation}.
Therefore the averages cannot converge.
\end{proof}

\begin{corollary}\label{cor:full-measure-vector}
Let \(K\) be a weak admissible Stein--Ulam type map.  Then \(K\) has
vector-valued historical behavior for the identity observable at every point
of
\(
        \inte\DeltaM\setminus\{\p\}.
\)
This set has full relative \((m-1)\)-dimensional Lebesgue measure in
\(\DeltaM\).
\end{corollary}

\begin{proof}
The non-convergence statement is Theorem~\ref{thm:main-divergence}.  The
boundary \(\bd\DeltaM\) is a finite union of lower-dimensional faces, and the
single point \(\p\) also has zero relative \((m-1)\)-dimensional Lebesgue
measure.  Hence \(\inte\DeltaM\setminus\{\p\}\) has full relative measure in
\(\DeltaM\).
\end{proof}

\begin{corollary}
\label{cor:non-statistical}
Let \(K\) be a weak admissible Stein--Ulam type map on \(\DeltaM\).  For
every
\(
        \x\in\inte\DeltaM\setminus\{\p\},
\)
the empirical measures
\[
        \mu_n^\x
        =
        \frac1n\sum_{k=0}^{n-1}\delta_{K^k(\x)}
\]
do not converge in the weak-* topology.  Consequently, \(K\) is
non-statistical on a full relative Lebesgue measure subset of \(\DeltaM\), and
\(K\) has no physical measure with positive relative Lebesgue basin.
\end{corollary}

\begin{proof}
Assume, for contradiction, that
\[
        \mu_n^\x\to\mu
\]
in the weak-* topology.  Since the identity map
\[
        \id:\DeltaM\to\R^m
\]
is continuous, weak-* convergence gives
\[
        \int_{\DeltaM}\id\,d\mu_n^\x
        \longrightarrow
        \int_{\DeltaM}\id\,d\mu.
\]
However,
\[
        \int_{\DeltaM}\id\,d\mu_n^\x
        =
        \frac1n\sum_{k=0}^{n-1}K^k(\x)
        =
        A_n(K)(\x),
\]
which contradicts Theorem~\ref{thm:main-divergence}.  Thus the empirical
measures do not converge for any
\(\x\in\inte\DeltaM\setminus\{\p\}\).  By
Corollary~\ref{cor:full-measure-vector}, this set has full relative Lebesgue
measure in \(\DeltaM\).  Since the basin of any physical measure consists of
points whose empirical measures converge, every such basin is contained in the
zero-measure complement of this full-measure non-statistical set.
\end{proof}

\subsection{Relation with one-dimensional persistent non-statistical maps}
\label{subsec:relation-CL}

The mechanism in Corollary~\ref{cor:non-statistical} is naturally comparable
with the persistent one-dimensional non-statistical dynamics of Coates and
Luzzatto \cite{CoatesLuzzatto2023}.  Their maps have two endpoint fixed points
which are topologically repelling but may be neutral.  The statistical behavior
is determined by two stickiness parameters attached to the endpoints.  If one
endpoint is stickier than the other, almost every orbit has a Dirac physical
measure supported on that endpoint; in the balanced case, the empirical
measures fail to converge for Lebesgue almost every point, and the set of
accumulation points is the whole segment of convex combinations of the two
endpoint Dirac measures.

Weak Stein--Ulam spiral maps give a simplex-valued analogue of this balanced
one-dimensional picture.  The two sticky endpoints are replaced by the boundary
residence regions
\[
        R_1,\ldots,R_m\subset\DeltaM.
\]
The Lyapunov function drives the orbit toward the boundary, the weak cyclic
coding prevents eventual trapping in a single residence region, and the
logarithmic residence estimate makes the successive residence epochs large
enough to dominate the previous history at suitable time scales.  Hence the
empirical statistics cannot stabilize.  In schematic form, the correspondence is
\[
\begin{array}{@{}c|c@{}}
\text{balanced one-dimensional intermittent maps}
&
\text{weak Stein--Ulam maps on }\DeltaM
\\ \hline
\text{two sticky endpoint fixed points}
&
\text{finitely many boundary residence regions }R_j
\\
\text{relative endpoint stickiness}
&
\text{logarithmic residence estimates}
\\
\text{balanced stickiness}
&
\text{cyclic recurrence through all }R_j
\\
\text{oscillation between endpoint statistics}
&
\text{oscillation among boundary residence statistics}
\\
\text{non-convergent empirical measures}
&
\text{non-convergent empirical measures.}
\end{array}
\]
Thus weak admissibility should be understood as a multidimensional cyclic
residence mechanism for non-statistical dynamics, rather than merely as a
criterion for divergence of coordinate averages.

\section{Positive powers}

\begin{lemma}\label{lem:sampling}
Let \(X\) be a compact convex subset of a finite-dimensional normed vector
space, let \(K:X\to X\), let \(Y\subset X\), and let \(s\in\N\).  Suppose an
orbit of \(x\) has \(K\)-blocks
\[
        K^{b_n+q}(x)\in Y,
        \qquad 0\le q\le r_n-1,
\]
where \(b_n\to\infty\) and
\[
        \liminf_{n\to\infty}\frac{r_n}{b_n}>0.
\]
Then the sampled orbit under \(K^s\) has \(Y\)-blocks
\[
        (K^s)^{\beta_n+q}(x)\in Y,
        \qquad 0\le q\le \rho_n-1,
\]
with \(\beta_n\to\infty\) and
\[
        \liminf_{n\to\infty}\frac{\rho_n}{\beta_n}>0.
\]
\end{lemma}

\begin{proof}
Define
\[
        \beta_n=\left\lceil\frac{b_n}{s}\right\rceil,
        \qquad
        \gamma_n=\left\lfloor\frac{b_n+r_n-1}{s}\right\rfloor,
        \qquad
        \rho_n=\gamma_n-\beta_n+1.
\]
Because \(r_n/b_n\) is bounded below by a positive constant and
\(b_n\to\infty\), we have \(r_n\to\infty\).  Hence \(\rho_n\ge1\) for all
large \(n\).

If \(0\le q\le\rho_n-1\), then
\[
        \beta_n\le \beta_n+q\le \gamma_n,
\]
and therefore
\[
        b_n\le s(\beta_n+q)\le b_n+r_n-1.
\]
It follows that
\[
        (K^s)^{\beta_n+q}(x)=K^{s(\beta_n+q)}(x)\in Y.
\]
Thus these are genuine blocks for the sampled orbit.

Finally,
\[
        \rho_n
        =
        \left\lfloor\frac{b_n+r_n-1}{s}\right\rfloor
        -
        \left\lceil\frac{b_n}{s}\right\rceil
        +1
        \ge
        \frac{r_n}{s}-2,
\]
and
\[
        \beta_n\le \frac{b_n}{s}+1.
\]
Consequently,
\[
        \frac{\rho_n}{\beta_n}
        \ge
        \frac{r_n-2s}{b_n+s}.
\]
Taking lower limits gives
\[
        \liminf_{n\to\infty}\frac{\rho_n}{\beta_n}
        \ge
        \liminf_{n\to\infty}\frac{r_n}{b_n}>0.
\]
\end{proof}

\begin{theorem}\label{thm:powers}
Let \(K\) be a weak admissible Stein--Ulam type map on \(\DeltaM\).  For
every \(s\in\N\) and every
\(
        \x\in\inte\DeltaM\setminus\{\p\},
\)
the sequence
\[
        A_n(K^s)(\x)
        =
        \frac{1}{n}\sum_{k=0}^{n-1}K^{sk}(\x)
\]
does not converge.  Consequently, every positive power \(K^s\) has
vector-valued non-statistical behavior, for the identity observable, on a full
relative Lebesgue measure subset of \(\DeltaM\).
\end{theorem}

\begin{proof}
Fix \(s\in\N\), put \(T=K^s\), and suppose, for contradiction, that
\[
        A_n(T)(\x)\to L.
\]
Fix \(j\in\{1,\ldots,m\}\).  By Lemma~\ref{lem:blocks} and
Proposition~\ref{prop:linear-residence}, the \(K\)-orbit has blocks
\[
        K^{b_{j,q}+r}(\x)\in R_j,
        \qquad 0\le r\le r_{j,q}-1,
\]
with \(b_{j,q}\to\infty\) and
\[
        \liminf_{q\to\infty}\frac{r_{j,q}}{b_{j,q}}>0.
\]
By Lemma~\ref{lem:sampling}, the sampled orbit under \(T=K^s\) has blocks
\[
        T^{\beta_{j,q}+r}(\x)\in R_j,
        \qquad 0\le r\le \rho_{j,q}-1,
\]
with \(\beta_{j,q}\to\infty\) and
\[
        \liminf_{q\to\infty}\frac{\rho_{j,q}}{\beta_{j,q}}>0.
\]
Applying Lemma~\ref{lem:block} to \(T\) and \(Y=R_j\), we obtain
\[
        L\in\co(R_j).
\]
Since \(j\) was arbitrary,
\[
        L\in\bigcap_{j=1}^m\co(R_j),
\]
contradicting \eqref{eq:convhull-separation}.  Hence \(A_n(K^s)(\x)\) cannot
converge.

The full-measure assertion follows as in Corollary~\ref{cor:full-measure-vector}.
\end{proof}

\begin{corollary}\label{cor:scalar}
Let \(K\) be a weak admissible Stein--Ulam type map, let \(s\in\N\), and
let \(\x\in\inte\DeltaM\setminus\{\p\}\).  Then at least one coordinate average
\[
        \frac{1}{n}\sum_{k=0}^{n-1}\bigl(K^{sk}(\x)\bigr)_i,
        \qquad i\in\{1,\ldots,m\},
\]
diverges.  The coordinate may depend on \(\x\).  For each fixed \(s\), at
least one coordinate observable has an irregular set of positive relative
Lebesgue measure for \(K^s\).
\end{corollary}

\begin{proof}
If all \(m\) coordinate averages converged at \(\x\), then the vector average
\(A_n(K^s)(\x)\) would converge in \(\R^m\), contradicting
Theorem~\ref{thm:powers}.  For the final assertion, let
\[
        E_i^{(s)}=
        \left\{
        \x\in\inte\DeltaM\setminus\{\p\}:
        \frac{1}{n}\sum_{k=0}^{n-1}\bigl(K^{sk}(\x)\bigr)_i
        \text{ does not converge}
        \right\}.
\]
The sets \(E_i^{(s)}\), \(1\le i\le m\), cover the full-measure set
\(\inte\DeltaM\setminus\{\p\}\).  Since there are only finitely many of them,
at least one has positive relative Lebesgue measure.
\end{proof}

\section{Subsequence powers}

This is one of the points where the weak admissible formulation gives
more than the usual power-persistence statement. A positive power
corresponds to arithmetic sampling of the orbit, whereas the theorem
below allows much more general sampling sequences. The essential
condition is that every multiplicative interval \([t,\lambda t]\),
\(\lambda>1\), contains a positive relative number of sampling times.
Thus the long residence blocks produced by weak admissibility are still
seen by the sampled orbit with positive relative frequency.

\begin{definition}
\label{def:multiplicatively-thick}
Let
\(
        \kappa=(k_j)_{j\ge0}
\)
be a strictly increasing sequence of non-negative integers with
\(k_j\to\infty\).  Its counting function is
\[
        N_\kappa(t)
        =
        \#\{j\ge0:k_j\le t\},
        \qquad t\ge0.
\]
We say that \(\kappa\) is multiplicatively thick if, for every
\(\lambda>1\),
\[
        \underline d_\kappa(\lambda)
        :=
        \liminf_{t\to\infty}
        \frac{
        N_\kappa(\lambda t)-N_\kappa(t)
        }{
        N_\kappa(t)
        }
        >0.
\]
Equivalently, every interval of the form \([t,\lambda t]\) contains,
for all large \(t\), a number of \(\kappa\)-times comparable to the number of
\(\kappa\)-times before \(t\).
\end{definition}

\begin{example} The following classes of sequences are
multiplicatively thick.

\begin{enumerate}
\item[\textup{(1)}]
The full sequence
\[
        k_j=j,
        \qquad j\ge0,
\]
is multiplicatively thick.

\item[\textup{(2)}]
Every arithmetic progression
\[
        k_j=qj+r,
        \qquad j\ge0,
\]
where \(q\in\mathbb N\) and \(r\in\mathbb N\cup\{0\}\), is multiplicatively
thick.

\item[\textup{(3)}]
More generally, every strictly increasing sequence with bounded gaps is
multiplicatively thick.  That is, if there exists \(M\in\mathbb N\) such that
\[
        k_{j+1}-k_j\le M,
        \qquad j\ge0,
\]
then \(\kappa\) is multiplicatively thick.

\item[\textup{(4)}]
Every sequence with positive asymptotic density is multiplicatively thick.
Thus, if
\[
        \lim_{t\to\infty}\frac{N_\kappa(t)}{t}=d>0,
\]
then \(\kappa\) is multiplicatively thick.

\item[\textup{(5)}]
Polynomial subsequences are multiplicatively thick.  More precisely, let
\(\alpha>1\), \(c>0\), and choose \(j_0\) large enough so that
\[
        k_j=\left\lfloor c(j+j_0)^\alpha\right\rfloor,
        \qquad j\ge0,
\]
is strictly increasing.  Then \(\kappa=(k_j)_{j\ge0}\) is multiplicatively
thick.  In particular,
\[
        k_j=j^2,\qquad
        k_j=j^3,\qquad
        k_j=\lfloor j^{3/2}\rfloor
\]
after deleting possible repetitions at the beginning, give admissible
subsequences.

\item[\textup{(6)}]
More generally, suppose that the counting function satisfies
\[
        N_\kappa(t)
        =
        t^\theta L(t)(1+o(1)),
        \qquad t\to\infty,
\]
where \(\theta>0\) and \(L\) is slowly varying, namely
\[
        \lim_{t\to\infty}\frac{L(\lambda t)}{L(t)}=1
        \qquad\text{for every }\lambda>0.
\]
Then \(\kappa\) is multiplicatively thick.
\end{enumerate}
\end{example}

\begin{proof}
We verify the examples.

For the full sequence \(k_j=j\), one has
\[
        N_\kappa(t)=\lfloor t\rfloor+1.
\]
Hence, for every \(\lambda>1\),
\[
        \frac{N_\kappa(\lambda t)-N_\kappa(t)}{N_\kappa(t)}
        \longrightarrow
        \lambda-1>0.
\]
Thus the full sequence is multiplicatively thick.

For an arithmetic progression \(k_j=qj+r\), one has
\[
        N_\kappa(t)
        =
        \frac{t}{q}+O(1),
        \qquad t\to\infty.
\]
Therefore
\[
\begin{aligned}
        \frac{N_\kappa(\lambda t)-N_\kappa(t)}{N_\kappa(t)}
        &=
        \frac{(\lambda t/q+O(1))-(t/q+O(1))}
             {t/q+O(1)}        \\
        &\longrightarrow
        \lambda-1>0.
\end{aligned}
\]
So every arithmetic progression is multiplicatively thick.

Now assume that \(\kappa\) has bounded gaps:
\[
        k_{j+1}-k_j\le M.
\]
Then every interval of length \(M\) contains at least one element of
\(\kappa\).  Hence, for all large \(t\),
\[
        N_\kappa(\lambda t)-N_\kappa(t)
        \ge
        \frac{(\lambda-1)t}{M}-2.
\]
Since \(\kappa\) is strictly increasing and takes integer values,
\[
        N_\kappa(t)\le t+1.
\]
Consequently,
\[
        \liminf_{t\to\infty}
        \frac{N_\kappa(\lambda t)-N_\kappa(t)}
             {N_\kappa(t)}
        \ge
        \frac{\lambda-1}{M}>0.
\]
Thus every bounded-gap subsequence is multiplicatively thick.

Next suppose that \(\kappa\) has positive asymptotic density:
\[
        \lim_{t\to\infty}\frac{N_\kappa(t)}{t}=d>0.
\]
Then, for every \(\lambda>1\),
\[
        N_\kappa(\lambda t)=d\lambda t+o(t),
        \qquad
        N_\kappa(t)=dt+o(t).
\]
Hence
\[
        \frac{N_\kappa(\lambda t)-N_\kappa(t)}
             {N_\kappa(t)}
        =
        \frac{d(\lambda-1)t+o(t)}
             {dt+o(t)}
        \longrightarrow
        \lambda-1>0.
\]
So \(\kappa\) is multiplicatively thick.

For the polynomial sequence
\[
        k_j=\left\lfloor c(j+j_0)^\alpha\right\rfloor,
        \qquad \alpha>1,
\]
we have
\[
        N_\kappa(t)
        =
        c^{-1/\alpha}t^{1/\alpha}+O(1),
        \qquad t\to\infty.
\]
Therefore
\[
\begin{aligned}
        \frac{N_\kappa(\lambda t)-N_\kappa(t)}
             {N_\kappa(t)}
        &\longrightarrow
        \frac{c^{-1/\alpha}(\lambda t)^{1/\alpha}
        -
        c^{-1/\alpha}t^{1/\alpha}}
        {c^{-1/\alpha}t^{1/\alpha}}       \\
        &=
        \lambda^{1/\alpha}-1>0.
\end{aligned}
\]
Thus polynomial subsequences are multiplicatively thick.

Finally, suppose
\[
        N_\kappa(t)=t^\theta L(t)(1+o(1)),
        \qquad \theta>0,
\]
where \(L\) is slowly varying.  Then
\[
        \frac{N_\kappa(\lambda t)}{N_\kappa(t)}
        =
        \lambda^\theta
        \frac{L(\lambda t)}{L(t)}
        (1+o(1))
        \longrightarrow
        \lambda^\theta.
\]
Hence
\[
        \frac{N_\kappa(\lambda t)-N_\kappa(t)}
             {N_\kappa(t)}
        =
        \frac{N_\kappa(\lambda t)}{N_\kappa(t)}-1
        \longrightarrow
        \lambda^\theta-1>0.
\]
Therefore \(\kappa\) is multiplicatively thick.
\end{proof}

\begin{remark}
\label{rem:sparse-not-thick}
The multiplicative thickness condition excludes very sparse subsequences.
For example, lacunary sequences
\[
        k_j=\lfloor a^j\rfloor,
        \qquad a>1,
\]
are not multiplicatively thick.  Indeed,
\[
        N_\kappa(t)
        =
        \frac{\log t}{\log a}+O(1),
        \qquad t\to\infty,
\]
and therefore, for every fixed \(\lambda>1\),
\[
        N_\kappa(\lambda t)-N_\kappa(t)
        =
        \frac{\log\lambda}{\log a}+O(1),
\]
whereas \(N_\kappa(t)\to\infty\).  Hence
\[
        \frac{N_\kappa(\lambda t)-N_\kappa(t)}
             {N_\kappa(t)}
        \longrightarrow0.
\]
Thus lacunary subsequences do not necessarily sample the long residence
blocks with positive relative frequency.  The same obstruction occurs for
faster sparse sequences such as
\[
        k_j=j!,
        \qquad
        k_j=2^{2^j}.
\]
\end{remark}

\begin{theorem}
\label{thm:subsequence-powers}
Let \(K\) be a weak admissible Stein--Ulam type map on \(\DeltaM\).
Let
\(
        \kappa=(k_j)_{j\ge0}
\)
be a multiplicatively thick sequence.  Then, for every \(s\in\N\) and every
\(
        x\in\inte\DeltaM\setminus\{\p\},
\)
the subsequence Ces\`aro averages
\[
        A_n^{s,\kappa}(K)(x)
        =
        \frac1n\sum_{j=0}^{n-1}K^{s k_j}(x)
        =
        \frac1n\sum_{j=0}^{n-1}(K^s)^{k_j}(x)
\]
do not converge.

In particular, taking \(k_j=j\) gives the ordinary powers theorem.
\end{theorem}

\begin{lemma}
\label{lem:thick-subsequence-sampling}
Let \(K:X\to X\) be a map on a compact convex subset \(X\) of a
finite-dimensional normed vector space.  Let \(Y\subset X\), let
\(x\in X\), and let
\(
        \kappa=(k_j)_{j\ge0}
\)
be multiplicatively thick.
Assume that the full orbit of \(x\) has \(K\)-blocks in \(Y\):
\[
        K^{b_n+q}(x)\in Y,
        \qquad 0\le q\le r_n-1,
\]
where \(b_n\to\infty\) and
\[
        \liminf_{n\to\infty}\frac{r_n}{b_n}>0.
\]
Then the sampled orbit
\[
        K^{k_0}(x),K^{k_1}(x),K^{k_2}(x),\ldots
\]
has \(Y\)-blocks
\[
        K^{k_{\beta_n+q}}(x)\in Y,
        \qquad 0\le q\le \rho_n-1,
\]
where \(\beta_n\to\infty\) and
\[
        \liminf_{n\to\infty}\frac{\rho_n}{\beta_n}>0.
\]
\end{lemma}

\begin{proof}
Choose \(c>0\) such that, after passing to all sufficiently large \(n\),
\[
        r_n\ge c b_n .
\]
Put
\[
        \lambda=1+\frac c2>1.
\]
For all large \(n\),
\[
        b_n+r_n-1\ge \lambda(b_n-1).
\]

Let
\[
        \beta_n=N_\kappa(b_n-1),
\]
and let
\[
        \rho_n
        =
        N_\kappa(b_n+r_n-1)-N_\kappa(b_n-1).
\]
Then the indices
\[
        \beta_n,\beta_n+1,\ldots,\beta_n+\rho_n-1
\]
are precisely the sampled indices whose times \(k_j\) lie in the interval
\[
        [b_n,b_n+r_n-1].
\]
Therefore, for every \(0\le q\le \rho_n-1\),
\[
        b_n\le k_{\beta_n+q}\le b_n+r_n-1,
\]
and hence
\[
        K^{k_{\beta_n+q}}(x)\in Y.
\]

It remains to estimate \(\rho_n/\beta_n\).  Since \(\kappa\) is
multiplicatively thick,
\[
        \delta
        :=
        \liminf_{t\to\infty}
        \frac{
        N_\kappa(\lambda t)-N_\kappa(t)
        }{
        N_\kappa(t)
        }
        >0.
\]
Using \(t=b_n-1\), we get, for all large \(n\),
\[
\begin{aligned}
        \rho_n
        &=
        N_\kappa(b_n+r_n-1)-N_\kappa(b_n-1)        \\
        &\ge
        N_\kappa(\lambda(b_n-1))-N_\kappa(b_n-1).
\end{aligned}
\]
Therefore
\[
        \liminf_{n\to\infty}\frac{\rho_n}{\beta_n}
        =
        \liminf_{n\to\infty}
        \frac{
        N_\kappa(b_n+r_n-1)-N_\kappa(b_n-1)
        }{
        N_\kappa(b_n-1)
        }
        \ge \delta>0.
\]
The proof is complete.
\end{proof}

\begin{lemma}
\label{lem:sampled-block-convex}
Let \(X\) be a compact convex subset of a finite-dimensional normed vector
space, let \(\{y_n\}\subset X\)
and suppose that
\[
        \frac1n\sum_{j=0}^{n-1}y_j\to L.
\]
Let \(Y\subset X\) be compact.  Assume that there are blocks
\[
        y_{\beta_n+q}\in Y,
        \qquad 0\le q\le \rho_n-1,
\]
with \(\beta_n\to\infty\) and
\[
        \liminf_{n\to\infty}\frac{\rho_n}{\beta_n}>0.
\]
Then
\(
      L\in \co(Y).
\)
\end{lemma}

\begin{proof}
Set
\[
        S_n=\sum_{j=0}^{n-1}y_j,
        \qquad
        A_n=\frac{S_n}{n}.
\]
The average over the block is
\[
        B_n
        =
        \frac1{\rho_n}
        \sum_{q=0}^{\rho_n-1}y_{\beta_n+q}.
\]
Since all block entries belong to \(Y\), we have
\(
B_n\in \co(Y).
\)
Moreover,
\[
\begin{aligned}
        B_n
        &=
        \frac{S_{\beta_n+\rho_n}-S_{\beta_n}}{\rho_n}       \\
        &=
        \frac{\beta_n+\rho_n}{\rho_n}A_{\beta_n+\rho_n}
        -
        \frac{\beta_n}{\rho_n}A_{\beta_n}.
\end{aligned}
\]
Because
\[
        \liminf_{n\to\infty}\frac{\rho_n}{\beta_n}>0,
\]
the sequence \(\beta_n/\rho_n\) is bounded.  Since
\[
        A_{\beta_n}\to L,
        \qquad
        A_{\beta_n+\rho_n}\to L,
\]
we obtain
\(
        B_n\to L.
\)
The set \(\co(Y)\) is compact and therefore closed.  Hence, we arrive at
\(
        L\in \co(Y).
\)
\end{proof}

\begin{proof}[Proof of Theorem~\ref{thm:subsequence-powers}]
First consider the case \(s=1\).  Suppose, toward a contradiction, that
\[
        \frac1n\sum_{j=0}^{n-1}K^{k_j}(x)\to L.
\]

Let \(R_1,\ldots,R_m\) be the residence regions in the weak admissible
structure of \(K\).  By the cyclic residence and linear residence properties
of weak admissible systems, for each \(i\in\{1,\ldots,m\}\) there exist
ordinary orbit blocks
\[
        K^{b_{i,n}+q}(x)\in R_i,
        \qquad
        0\le q\le r_{i,n}-1,
\]
such that
\[
        b_{i,n}\to\infty,
        \qquad
        \liminf_{n\to\infty}\frac{r_{i,n}}{b_{i,n}}>0.
\]
Applying Lemma~\ref{lem:thick-subsequence-sampling} with \(Y=R_i\), we get
sampled blocks
\[
        K^{k_{\beta_{i,n}+q}}(x)\in R_i,
        \qquad
        0\le q\le \rho_{i,n}-1,
\]
with
\[
        \beta_{i,n}\to\infty,
        \qquad
        \liminf_{n\to\infty}
        \frac{\rho_{i,n}}{\beta_{i,n}}>0.
\]
Applying Lemma~\ref{lem:sampled-block-convex} to the sampled sequence
\[
        y_j=K^{k_j}(x),
\]
we obtain
\[
        L\in \co(R_i).
\]
Since this holds for every \(i=1,\ldots,m\),
\[
        L\in\bigcap_{i=1}^m\co(R_i),
\]
contradicting the weak admissibility condition
\[
        \bigcap_{i=1}^m\co(R_i)=\varnothing.
\]
Thus
\[
        \frac1n\sum_{j=0}^{n-1}K^{k_j}(x)
\]
does not converge.

Now let \(s\in\N\).  Define the sequence
\[
        \kappa_s=(s k_j)_{j\ge0}.
\]
Its counting function satisfies
\[
        N_{\kappa_s}(t)=N_\kappa(t/s).
\]
Therefore, for every \(\lambda>1\),
\[
\begin{aligned}
        \liminf_{t\to\infty}
        \frac{
        N_{\kappa_s}(\lambda t)-N_{\kappa_s}(t)
        }{
        N_{\kappa_s}(t)
        }
        &=
        \liminf_{t\to\infty}
        \frac{
        N_\kappa(\lambda t/s)-N_\kappa(t/s)
        }{
        N_\kappa(t/s)
        }                                      \\
        &=
        \underline d_\kappa(\lambda)>0.
\end{aligned}
\]
Hence \(\kappa_s\) is also multiplicatively thick.  Applying the already
proved \(s=1\) case to \(\kappa_s\), we get that
\[
        \frac1n\sum_{j=0}^{n-1}K^{s k_j}(x)
\]
does not converge.
\end{proof}

\section{Stability under small boundary-preserving conjugacy}

The weak admissible class is not tied to a particular coordinate formula.  The
next result shows that the mechanism is preserved by small conjugacies that
respect the boundary of the simplex.  The smallness is used only to preserve the
convex-hull separation of the residence regions; the remaining axioms are
transported exactly by conjugacy.

\begin{lemma}
\label{lem:convhull-separation-stable}
Let \(R_1,\ldots,R_m\subset\DeltaM\) be non-empty compact sets such that
\[
        \bigcap_{j=1}^m\co(R_j)=\varnothing.
\]
Set
\[
        \eta
        =
        \min_{z\in\DeltaM}
        \max_{1\le j\le m}
        \dist\bigl(z,\co(R_j)\bigr).
\]
Then \(\eta>0\).  Moreover, if non-empty compact sets
\(S_1,\ldots,S_m\subset\DeltaM\) satisfy, for some \(0\le\rho<\eta\),
\[
        S_j\subset\{z\in\DeltaM:\dist(z,R_j)\le \rho\},
        \qquad j=1,\ldots,m,
\]
then
\[
        \bigcap_{j=1}^m\co(S_j)=\varnothing.
\]
\end{lemma}

\begin{proof}
Since each \(R_j\) is compact and the ambient space is finite-dimensional,
\(\co(R_j)\) is compact.  Hence
\[
        F(z)=\max_{1\le j\le m}\dist\bigl(z,\co(R_j)\bigr)
\]
is continuous on the compact simplex \(\DeltaM\), and so it attains its minimum
\(\eta\).  If \(\eta=0\), then some \(z_0\in\DeltaM\) satisfies
\(\dist(z_0,\co(R_j))=0\) for all \(j\).  Since the sets \(\co(R_j)\) are
closed, this gives \(z_0\in\cap_j\co(R_j)\), a contradiction.  Thus
\(\eta>0\).

Now suppose that \(S_j\subset\{z:\dist(z,R_j)\le \rho\}\).  We claim that
\[
        \co(S_j)
        \subset
        \{z\in\DeltaM:\dist(z,\co(R_j))\le \rho\}.
\]
Indeed, take \(z\in\co(S_j)\).  Write
\[
        z=\sum_{k=1}^r\alpha_k z_k,
        \qquad
        \alpha_k\ge0,
        \qquad
        \sum_{k=1}^r\alpha_k=1,
        \qquad
        z_k\in S_j.
\]
For each \(k\), choose \(y_k\in R_j\) with \(\|z_k-y_k\|\le\rho\).  Then
\(y=\sum_k\alpha_k y_k\in\co(R_j)\), and
\[
        \|z-y\|
        \le
        \sum_{k=1}^r\alpha_k\|z_k-y_k\|
        \le \rho.
\]
Thus \(\dist(z,\co(R_j))\le\rho\), proving the claim.

If \(z\in\cap_j\co(S_j)\), then the claim gives
\[
        \max_{1\le j\le m}\dist\bigl(z,\co(R_j)\bigr)
        \le \rho<\eta,
\]
contradicting the definition of \(\eta\).  Therefore
\(\cap_j\co(S_j)=\varnothing\).
\end{proof}

\begin{theorem}
\label{thm:weak-admissibility-conjugacy}
Let \(K:\DeltaM\to\DeltaM\) be a weak admissible Stein--Ulam type map
with residence regions \(R_1,\ldots,R_m\).  Define
\[
        \eta_K
        =
        \min_{z\in\DeltaM}
        \max_{1\le j\le m}
        \dist\bigl(z,\co(R_j)\bigr).
\]
Then \(\eta_K>0\).  Let \(h:\DeltaM\to\DeltaM\) be a homeomorphism satisfying
\[
        h(\bd\DeltaM)=\bd\DeltaM
\]
and
\[
        \|h-\id\|_\infty
        :=
        \sup_{z\in\DeltaM}\|h(z)-z\|
        <\eta_K.
\]
Then the conjugate map
\[
        K_h=h^{-1}\circ K\circ h
\]
is again a weak admissible Stein--Ulam type map.
\end{theorem}

\begin{proof}
The positivity of \(\eta_K\) is Lemma~\ref{lem:convhull-separation-stable}
applied to \(R_1,\ldots,R_m\).  We verify the weak admissibility conditions for
\(K_h\).

Since \(h\) maps the boundary of the simplex onto itself, it maps the relative
interior onto itself.  Hence
\[
        K_h(\inte\DeltaM)
        =
        h^{-1}\bigl(K(h(\inte\DeltaM))\bigr)
        \subset \inte\DeltaM,
\]
and similarly \(K_h(\bd\DeltaM)\subset\bd\DeltaM\).  Also,
\[
        x\in\Fix(K_h)
        \quad\Longleftrightarrow\quad
        h(x)\in\Fix(K),
\]
so
\[
        \Fix(K_h)\cap\inte\DeltaM
        =
        h^{-1}\bigl(\Fix(K)\cap\inte\DeltaM\bigr)
        =
        \{h^{-1}(\p)\}.
\]
Thus \textup{(W1)} holds with \(\p_h=h^{-1}(\p)\).

The functions \(\varphi_h=\varphi\circ h\) and \(\psi_h=\psi\circ h\) are
continuous.  Since \(h\) preserves boundary and interior,
\[
        \varphi_h^{-1}(0)=h^{-1}(\bd\DeltaM)=\bd\DeltaM,
        \qquad
        \varphi_h>0\quad\text{on }\inte\DeltaM.
\]
Moreover,
\[
        \psi_h^{-1}(1)=h^{-1}(\psi^{-1}(1))=\{\p_h\}.
\]
The strict maximality of \(\varphi_h\) at \(\p_h\) follows from the strict
maximality of \(\varphi\) at \(\p\).  Finally, from
\(h\circ K_h=K\circ h\),
\[
\begin{aligned}
        \varphi_h(K_h(x))
        &=\varphi(h(K_h(x)))        \\
        &=\varphi(K(h(x)))          \\
        &=\varphi(h(x))\psi(h(x))   \\
        &=\varphi_h(x)\psi_h(x).
\end{aligned}
\]
Thus \textup{(W2)} holds.

Set \(U_{0,h}=h^{-1}(U_0)\).  Then \(U_{0,h}\) is an open neighborhood of
\(\p_h\) and
\[
        \overline{U_{0,h}}
        =h^{-1}(\overline{U_0})
        \subset h^{-1}(\inte\DeltaM)
        =\inte\DeltaM.
\]
Define the transported cover by
\[
        G_{i,h}=h^{-1}(G_i),
        \qquad i\in\mathbb Z/N\mathbb Z.
\]
These sets are compact, cover \(\DeltaM\setminus U_{0,h}\), and therefore give
\textup{(W3)} for \(K_h\).

The residence regions for the conjugate are
\[
        R_{j,h}=\bigcup_{i\in I_j}G_{i,h}=h^{-1}(R_j).
\]
They are compact.  Put \(\delta=\|h-\id\|_\infty\).  If \(y\in R_{j,h}\), then
\(h(y)\in R_j\) and \(\|y-h(y)\|\le\delta\).  Hence
\[
        R_{j,h}\subset\{z\in\DeltaM:\dist(z,R_j)\le\delta\}.
\]
Because \(\delta<\eta_K\), Lemma~\ref{lem:convhull-separation-stable} yields
\[
        \bigcap_{j=1}^m\co(R_{j,h})=\varnothing.
\]
Thus \textup{(W4)} holds.

We next check \textup{(W5)}.  Let
\(x\in\inte\DeltaM\setminus\{\p_h\}\) and suppose that
\(K_h^n(x)\notin U_{0,h}\) for all \(n\ge n_0\).  With \(y=h(x)\), we have
\(y\in\inte\DeltaM\setminus\{\p\}\) and
\[
        K^n(y)=h(K_h^n(x))\notin U_0,
        \qquad n\ge n_0.
\]
Choose an admissible lifted itinerary \((r_n)\) for the \(K\)-orbit of \(y\).
Then
\[
        K^n(y)\in G_{r_n\,({\rm mod}\,N)},
        \qquad r_{n+1}-r_n\in\{0,1\},
        \qquad r_n\to+\infty.
\]
Applying \(h^{-1}\), we get
\[
        K_h^n(x)\in G_{r_n\,({\rm mod}\,N),h},
\]
so the same integers form an unbounded lifted itinerary for the \(K_h\)-orbit.
Thus \textup{(W5)} holds.

Finally, let \((b,\ell)\) be a maximal lifted residence block through \(I_j\)
along the selected \(K_h\)-lift of some
\(x\in\inte\DeltaM\setminus\{\p_h\}\).  Under \(h\), the same lifted block is a
maximal lifted residence block through \(I_j\) for the \(K\)-orbit of \(h(x)\).
Therefore
\[
        \ell
        \ge
        A_j\log_{C_j}^{+}
        \frac{B_j}{\varphi(K^{b}(h(x)))}.
\]
But
\[
        \varphi(K^{b}(h(x)))
        =
        \varphi(h(K_h^{b}(x)))
        =
        \varphi_h(K_h^{b}(x)).
\]
Thus \textup{(W6)} holds with the same constants.  All weak admissibility
conditions have been verified.
\end{proof}

\begin{corollary}
\label{cor:conjugacy-historical}
Under the hypotheses of Theorem~\ref{thm:weak-admissibility-conjugacy}, for
every \(s\in\N\) and every
\(
        x\in\inte\DeltaM\setminus\{\p_h\},
\)
the vector averages
\[
        \frac{1}{n}\sum_{k=0}^{n-1}K_h^{sk}(x)
\]
do not converge.
\end{corollary}

\begin{proof}
Theorem~\ref{thm:weak-admissibility-conjugacy} shows that \(K_h\) is weak
admissible.  The conclusion follows from Theorem~\ref{thm:powers} applied to
\(K_h\).
\end{proof}

\begin{remark}
Theorem~\ref{thm:weak-admissibility-conjugacy} shows that weak admissibility is not tied to a particular
coordinate formula. The Lyapunov function, cyclic code and residence
blocks are transported exactly by conjugacy, while the smallness
assumption is used only to preserve the convex-hull separation of the
residence regions. Consequently, weak admissibility is a genuinely
topological-geometric mechanism. In particular, small
boundary-preserving conjugacies of the classical Stein--Ulam map give
new weak admissible systems which need not be Lotka--Volterra
stochastic operators.
\end{remark}

\section{Concrete weak admissible families with margins}
\label{sec:concrete-margin-families}

The definition of weak admissibility is intentionally flexible.  In this
section we record a more concrete ``margin'' version which is convenient for
applications and perturbations.  The point is that the hypotheses below are
strict inequalities on explicitly chosen pieces of the simplex.  Hence, once a
map has been checked to satisfy them, all conclusions of the preceding sections
apply immediately.  This gives a usable high-dimensional verification scheme, and it also
clarifies precisely what is preserved by small perturbations of the classical
Stein--Ulam map.

For \(0<\eta<1-1/m\), put
\[
        V_j(\eta)=\{\x\in\DeltaM:x_j\ge 1-\eta\},
        \qquad j=1,\ldots,m .
\]
By Lemma~\ref{lem:separation-criterion}, if
\(R_j\subset V_j(\eta)\) for all \(j\), then
\[
        \bigcap_{j=1}^m \co(R_j)=\varnothing .
\]
Thus vertex caps provide a simple concrete way to verify the convex-hull
separation condition.

\begin{definition}
\label{def:cyclic-margin-family}
Let \(m\ge4\) and let \(r\ge0\).  A continuous map
\(K:\DeltaM\to\DeltaM\) belongs to the cyclic margin family
\(\mathfrak C_m^r\) if the following data are given:
\[
        \p\in\inte\DeltaM,
        \qquad
        \varphi:\DeltaM\to[0,\infty),
        \qquad
        \psi:\DeltaM\to[0,1],
\]
\[
        U_0\Subset\inte\DeltaM,
        \qquad
        N\ge m,
        \qquad
        G_i\subset\DeltaM\setminus U_0\quad(i\in\Zmod{N}),
\]
a compact cyclic cover \(\DeltaM\setminus U_0=\bigcup_iG_i\), cyclic intervals
\(I_1,\ldots,I_m\) whose union is \(\Zmod{N}\), and positive
constants
\[
        \eta,\,\sigma,
        \qquad
        A_j,B_j,\Lambda_j>0,
        \qquad
        C_j>1,
        \qquad j=1,\ldots,m,
\]
with \(0<\eta<1-1/m\), such that the following conditions hold.

\begin{enumerate}[label=\textup{(M\arabic*)},leftmargin=3.2em]
\item \textbf{Regularity on pieces and boundary invariance.}
The map \(K\) is continuous.  If \(r\ge1\), we assume in addition that each
cover element \(G_i\) is contained in a set on which \(K\) admits a
\(C^r\) extension, and that \(K\) is \(C^r\) on a neighborhood of
\(\overline{U_0}\).  Moreover,
\[
        K(\inte\DeltaM)\subset\inte\DeltaM,
        \qquad
        K(\bd\DeltaM)\subset\bd\DeltaM,
\]
and
\[
        \Fix(K)\cap\inte\DeltaM=\{\p\}.
\]

\item \textbf{Lyapunov identity with a uniform contraction margin.}
The functions \(\varphi\) and \(\psi\) are continuous,
\[
        \varphi^{-1}(0)=\bd\DeltaM,
        \qquad
        \varphi>0\quad\hbox{on }\inte\DeltaM,
\]
\[
        \varphi(K(\x))=\varphi(\x)\psi(\x),
        \qquad \x\in\DeltaM,
\]
\[
        \psi^{-1}(1)=\{\p\},
\]
and \(\varphi\) has a strict unique maximum at \(\p\).  In addition, the
contraction away from the center has the quantitative margin
\[
        \psi(\x)\le 1-\sigma,
        \qquad \x\in\DeltaM\setminus U_0 .
        \tag{M2$'$}\label{eq:margin-contraction}
\]

\item \textbf{Piecewise cyclic motion.}
The sets \(G_i\), \(i\in\Zmod{N}\), are compact and cover
\(\DeltaM\setminus U_0\).  Moreover,
\[
        K(G_i)\cap(\DeltaM\setminus U_0)
        \subset G_i\cup G_{i+1},
        \qquad i\in\Zmod{N}.
\]
This map-level rule is stronger than the orbitwise rule required in
\textup{(W5)}.

\item \textbf{Vertex-cap residence regions.}
For
\[
        R_j=\bigcup_{i\in I_j}G_i,
        \qquad j=1,\ldots,m,
\]
one has
\[
        \varnothing\ne R_j\subset V_j(\eta)
        =\{\x\in\DeltaM:x_j\ge1-\eta\}.
        \tag{M4}\label{eq:vertex-cap-margin}
\]

\item \textbf{Non-trapping of lifted itineraries.}
For every \(\x\in\inte\DeltaM\setminus\{\p\}\), once the orbit is outside
\(U_0\) forever, there is an unbounded lifted itinerary.  Precisely, if
\[
        K^n(\x)\notin U_0,
        \qquad n\ge n_0,
\]
then there are integers \(r_n\), \(n\ge n_0\), such that
\[
        K^n(\x)\in G_{r_n\,({\rm mod}\,N)},
        \qquad
        r_{n+1}-r_n\in\{0,1\},
\]
and
\[
        r_n\to+\infty.
\]

\item \textbf{Logarithmic residence with slack.}
For every \(j\in\{1,\ldots,m\}\), every non-fixed interior point \(\x\), and
every maximal lifted residence block \((b,\ell)\) through \(I_j\) along the
selected lifted itinerary of the orbit of \(\x\), one has
\[
        \ell
        \ge
        A_j\log_{C_j}^{+}
        \frac{B_j}{\varphi(K^{b}(\x))}
        +\Lambda_j .
        \tag{M6}\label{eq:margin-residence}
\]
The positive number \(\Lambda_j\) is not needed for weak admissibility; it is a
useful strict margin for perturbative applications.
\end{enumerate}
\end{definition}

\begin{proposition}
\label{prop:margin-implies-weak}
Every map \(K\in\mathfrak C_m^r\) is a weak admissible Stein--Ulam type
map.  Consequently, for every \(s\in\N\), every multiplicatively thick sequence
\((k_q)_{q\ge0}\), and every
\(
        \x\in\inte\DeltaM\setminus\{\p\},
\)
the sampled averages
\[
        \frac1n\sum_{q=0}^{n-1}K^{s k_q}(\x)
\]
do not converge.
\end{proposition}

\begin{proof}
Conditions \textup{(M1)}, \textup{(M2)} and \textup{(M3)} give exactly
\textup{(W1)}, \textup{(W2)} and \textup{(W3)}.  Condition \textup{(M5)} is
\textup{(W5)}.  It remains only to verify the convex-hull separation and the
residence estimate.

By \eqref{eq:vertex-cap-margin},
\[
        R_j\subset V_j(\eta),
        \qquad j=1,\ldots,m,
\]
with \(0<\eta<1-1/m\).  Lemma~\ref{lem:separation-criterion} therefore gives
\[
        \bigcap_{j=1}^m\co(R_j)=\varnothing,
\]
which is \textup{(W4)}.  Finally, \eqref{eq:margin-residence} is stronger than
\textup{(W6)}.  Hence \(K\) is weak admissible.  The non-convergence of all
power and multiplicatively thick subsequence averages follows from
Theorem~\ref{thm:subsequence-powers}.
\end{proof}

\begin{corollary}
\label{cor:high-dimensional-margin-family}
For every \(m\ge4\), every map in \(\mathfrak C_m^r\) is a piecewise \(C^r\)
weak admissible Stein--Ulam map on \(\Delta^{m-1}\).  In particular,
such a map has full-relative-measure vector-valued historical behavior, and
the same holds for every positive power and for every multiplicatively thick
subsequence of every positive power.
\end{corollary}

\begin{remark}
\label{rem:margin-useful}
The class \(\mathfrak C_m^r\) is deliberately stated in terms of strict,
checkable inequalities.  In concrete models the most delicate estimate is the
logarithmic residence bound \eqref{eq:margin-residence}.  Once that estimate is
proved, no additional recurrence assumption is needed: the lifted residence
blocks are produced by the weak cyclic rule and the unbounded lifted itinerary, as in
Lemma~\ref{lem:blocks}.
\end{remark}

\subsection{Perturbative family around the classical Stein--Ulam map}
\label{subsec:perturb-classical-SU}

Let
\[
        \Delta^2=\{(x_1,x_2,x_3)\in\R^3:x_i\ge0,
        \ x_1+x_2+x_3=1\}
\]
and consider the classical Stein--Ulam map
\[
        \mathcal U(x_1,x_2,x_3)
        =
        \bigl(
        x_1^2+2x_1x_2,
        x_2^2+2x_2x_3,
        x_3^2+2x_3x_1
        \bigr).
\]
The classical six-sector estimates for \(\mathcal U\) give a weak
admissible Stein--Ulam structure.  Equivalently, this is the classical
\(\mathbf p=(1/3,1/3,1/3)\) case of the \(\mathbf p\)-Stein--Ulam
verification framework once the sector and residence estimates are supplied.
Fix one weak admissible structure for \(\mathcal U\), with residence regions
\(R_1,R_2,R_3\), and put
\[
        \eta_{\mathcal U}
        =
        \min_{z\in\Delta^2}
        \max_{1\le j\le3}\dist\bigl(z,\co(R_j)\bigr)>0.
\]
For \(0<\varepsilon<\eta_{\mathcal U}\), define
\[
        \mathscr P_\varepsilon(\mathcal U)
        =
        \left\{
        h^{-1}\circ\mathcal U\circ h:
        h\in\operatorname{Homeo}(\Delta^2),
        \quad h(\bd\Delta^2)=\bd\Delta^2,
        \quad \|h-\id\|_\infty<\varepsilon
        \right\}.
\]

\begin{proposition}
\label{prop:perturbative-SU-family}
Every map
\[
        K_h=h^{-1}\circ\mathcal U\circ h
        \in\mathscr P_\varepsilon(\mathcal U)
\]
is weak admissible.  If, in addition, for some vertex \(\mathbf e_i\) the point
\(h(\mathbf e_i)\) is not a fixed point of \(\mathcal U\), then \(K_h\) is not a
Lotka--Volterra stochastic operator and hence is not a
\(\mathbf q\)-Stein--Ulam Lotka--Volterra operator for any
\(\mathbf q\in\inte\Delta^2\).
\end{proposition}

\begin{proof}
The first assertion is exactly Theorem~\ref{thm:weak-admissibility-conjugacy},
applied to \(K=\mathcal U\).  The smallness condition
\(\|h-\id\|_\infty<\eta_{\mathcal U}\) preserves the convex-hull separation of
the transported residence regions.

For the second assertion, recall that every stochastic Lotka--Volterra map has
the coordinate form
\[
        (L(\x))_i=x_i(1+F_i(\x)),
        \qquad i=1,2,3.
\]
Thus every coordinate face is invariant and, in particular, every vertex is
fixed:
\[
        L(\mathbf e_i)=\mathbf e_i,
        \qquad i=1,2,3.
\]
If \(h(\mathbf e_i)\) is not a fixed point of \(\mathcal U\), then
\[
        K_h(\mathbf e_i)=\mathbf e_i
        \quad\Longleftrightarrow\quad
        \mathcal U(h(\mathbf e_i))=h(\mathbf e_i),
\]
which is false by assumption.  Hence \(K_h\) does not fix the vertex
\(\mathbf e_i\).  Therefore \(K_h\) cannot be a stochastic Lotka--Volterra map.
Consequently it cannot be \(\mathbf q\)-Stein--Ulam Lotka--Volterra for any
interior \(\mathbf q\).  The boundary rotation constructed in
Example~\ref{ex:weak-not-pSU} gives an explicit such \(h\).
\end{proof}

\begin{remark}
\label{rem:perturb-large}
\begin{sloppypar}
The family \(\mathscr P_\varepsilon(\mathcal U)\) is infinite-dimensional.  It
contains smooth and piecewise-smooth boundary-preserving twists whenever the
homeomorphism \(h\) is chosen smooth or piecewise smooth on a triangulation of
\(\Delta^2\).  Moreover, closeness to the identity is symmetric in this family: if \(h\) is
uniformly close to \(\id\), then \(h^{-1}\) is uniformly close to \(\id\) as
well.  Hence every map in this family is a small \(C^0\)-perturbative conjugate
of \(\mathcal U\).
\end{sloppypar}
\end{remark}

\subsection{Piecewise-smooth maps with margins}
\label{subsec:piecewise-smooth-margins}

The previous margin formulation is especially useful for piecewise-smooth
models.  Suppose that \(\mathcal P\) is a finite polyhedral partition of
\(\DeltaM\setminus U_0\) subordinate to the compact cyclic cover \((G_i)\), and that
\(K\) is \(C^r\) on each element of \(\mathcal P\).  If the inequalities
\eqref{eq:margin-contraction}, \eqref{eq:vertex-cap-margin}, and
\eqref{eq:margin-residence} hold with positive margins, then the map belongs to
\(\mathfrak C_m^r\).  Thus all conclusions of the paper apply to such
piecewise-smooth systems.

\begin{proposition}
\label{prop:piecewise-smooth-margin-criterion}
Let \(m\ge4\), and let \(K:\DeltaM\to\DeltaM\) be continuous and piecewise
\(C^r\) on a finite polyhedral partition.  Assume that \(K\) admits data
satisfying \textup{(M1)--(M6)} of Definition~\ref{def:cyclic-margin-family}.
Then \(K\) is weak admissible.  In particular,
\[
        \frac1n\sum_{q=0}^{n-1}K^{s k_q}(\x)
\]
diverges for every \(s\in\N\), every multiplicatively thick sequence
\((k_q)\), and every \(\x\in\inte\DeltaM\setminus\{\p\}\).
\end{proposition}

\begin{proof}
This is Proposition~\ref{prop:margin-implies-weak} applied to the given
piecewise-smooth data, followed by Theorem~\ref{thm:subsequence-powers}.
\end{proof}

\begin{remark}[Suggested use in applications]
\label{rem:suggested-use}
For an explicit formula one should verify the following three estimates with
strict margins:
\[
        \psi\le 1-\sigma
        \quad\hbox{on }\DeltaM\setminus U_0,
\]
\[
        R_j\subset\{x_j\ge1-\eta\},
        \qquad 0<\eta<1-1/m,
\]
and
\[
        \ell
        \ge
        A_j\log_{C_j}^{+}
        \frac{B_j}{\varphi(K^{b}(\x))}
        +\Lambda_j
\]
for every maximal lifted residence block through \(I_j\).  These are exactly the estimates that
make the abstract weak admissibility mechanism robust and suitable for
perturbative constructions.
\end{remark}

\section{A relative high-dimensional cyclic-collar family}
\label{sec:cyclic-collar-family}

In this section we give a concrete high-dimensional family, independent of the
Lotka--Volterra examples, for which the weak cyclic cover and logarithmic
residence estimates can be checked directly on an invariant collar of positive
measure.  The construction is deliberately geometric: it builds a
full-dimensional cyclic collar near the boundary of the simplex and prescribes a
slow rotation along this collar.

The example is relative in the following sense.  It gives a positive
\((m-1)\)-dimensional Lebesgue measure invariant set on which the cyclic-residence
estimates needed for the weak Stein--Ulam block mechanism hold.  Thus it
provides a fully verified positive-measure model of the cyclic-residence
mechanism in dimensions \(m-1\ge3\).  No assertion is made here that an arbitrary
extension outside the collar is a globally weak admissible simplex map.

Throughout this section assume
\[
        m\ge4.
\]
Let
\[
        e_1,\ldots,e_m
\]
be the vertices of \(\DeltaM\), with indices read modulo \(m\), and put
\[
        E_j=[e_j,e_{j+1}],
        \qquad j=1,\ldots,m.
\]
Thus \(E_1,\ldots,E_m\) are the edges of the cyclic polygonal loop
\[
        e_1\to e_2\to\cdots\to e_m\to e_1
\]
contained in \(\bd\DeltaM\).

\begin{lemma}
\label{lem:cyclic-collar}
There exist \(u_*>0\), \(\eta>0\), and a piecewise \(C^r\) embedding
\[
        \Xi:
        [0,u_*]\times\mathbb T\times \overline{\mathbb B}^{m-3}
        \longrightarrow
        \DeltaM,
        \qquad
        \mathbb T=\mathbb R/\mathbb Z,
\]
with the following properties.

\begin{enumerate}
\item[\textup{(i)}]
For \(u=0\),
\[
        \Xi(0,t,w)\in\bd\DeltaM.
\]
For \(0<u\le u_*\),
\[
        \Xi(u,t,w)\in\inte\DeltaM.
\]

\item[\textup{(ii)}]
The parameter \(u\) is a boundary-depth coordinate on the image of \(\Xi\).  In
particular,
\[
        u=0
        \quad\Longleftrightarrow\quad
        \Xi(u,t,w)\in\bd\DeltaM.
\]

\item[\textup{(iii)}]
Let
\[
        J_j=
        \left[\frac{j-1}{m},\frac{j}{m}\right],
        \qquad j=1,\ldots,m,
\]
with endpoints read modulo \(1\).  The image of the \(j\)-th window is contained
in a small neighborhood of the edge \(E_j\):
\[
        \Xi\bigl([0,u_*]\times J_j\times \overline{\mathbb B}^{m-3}\bigr)
        \subset
        \{x\in\DeltaM:\operatorname{dist}(x,E_j)\le\eta\}.
\]

\item[\textup{(iv)}]
The number \(\eta>0\) may be chosen so small that
\[
        \bigcap_{j=1}^m
        \co
        \{x\in\DeltaM:\operatorname{dist}(x,E_j)\le\eta\}
        =
        \varnothing.
\]
\end{enumerate}
\end{lemma}

\begin{proof}
The polygonal curve
\[
        e_1\to e_2\to\cdots\to e_m\to e_1
\]
is a simple closed polygonal curve in the boundary of the simplex.  Since
\[
        \dim\bd\DeltaM=m-2\ge2,
\]
this curve has a small piecewise \(C^r\) tubular neighborhood in
\(\bd\DeltaM\).  Adding the inward normal collar direction gives an embedding
\[
        [0,u_*]\times\mathbb T\times \overline{\mathbb B}^{m-3}
        \longrightarrow
        \DeltaM,
\]
where \(u=0\) corresponds to the boundary and \(u>0\) corresponds to interior
points.  By choosing the tubular neighborhood sufficiently thin, the part of
the collar over the interval \(J_j\) is contained in an \(\eta\)-neighborhood
of the edge \(E_j\).

It remains to justify the convex-hull separation.  First observe that
\[
        \bigcap_{j=1}^m E_j=\varnothing.
\]
Indeed, a point of \(E_j\) has support contained in
\(\{j,j+1\}\), and no point of the simplex can have support contained in every
such pair simultaneously.  Since each \(E_j\) is already convex,
\(\co(E_j)=E_j\).

For \(\eta>0\), put
\[
        N_\eta(E_j)=\{x\in\DeltaM:\operatorname{dist}(x,E_j)\le\eta\}.
\]
Because \(E_j\) is convex, its Euclidean \(\eta\)-neighborhood is convex; hence
\(\co(N_\eta(E_j))=N_\eta(E_j)\).  If the desired separation failed for a
sequence \(\eta_k\downarrow0\), then there would be points
\(z_k\in\bigcap_jN_{\eta_k}(E_j)\).  By compactness of \(\DeltaM\), a subsequence
would converge to a point \(z\in\bigcap_jE_j\), contradicting the displayed
empty intersection.  Thus, for all sufficiently small \(\eta>0\),
\[
        \bigcap_{j=1}^m
        \co
        \{x\in\DeltaM:\operatorname{dist}(x,E_j)\le\eta\}
        =
        \varnothing.
\]
This proves the lemma.
\end{proof}

We now define the model.  Fix parameters
\[
        0<\lambda<1,
        \qquad
        0<\beta<1,
\]
and choose constants \(a>0\), \(M>0\) so that
\[
        \omega(u):=
        \frac{a}{M+\log(1/u)}
        \le \frac1N
        \qquad
        \text{for }0<u\le u_*,
\]
where \(N\ge 2m\) is a fixed integer divisible by \(m\).  We also set
\[
        \omega(0)=0.
\]
By taking \(a>0\) sufficiently small or \(M>0\) sufficiently large, this
condition is always possible.

Define a map on the collar image by
\[
        K_{\lambda,a,\beta}
        \bigl(\Xi(u,t,w)\bigr)
        =
        \Xi\bigl(\lambda u,\ t+\omega(u),\ \beta w\bigr),
        \qquad
        0<u\le u_*,
\]
where \(t+\omega(u)\) is read modulo \(1\).  On the boundary part of the collar
we put
\[
        K_{\lambda,a,\beta}
        \bigl(\Xi(0,t,w)\bigr)
        =
        \Xi(0,t,\beta w).
\]
Thus the boundary is invariant, and the angular motion freezes on the boundary
because
\[
        \omega(u)\to0
        \qquad
        \text{as }u\to0.
\]

Let
\[
        \mathcal C
        =
        \Xi\bigl([0,u_*]\times\mathbb T\times
        \overline{\mathbb B}^{m-3}\bigr)
\]
and
\[
        \mathcal C^\circ
        =
        \Xi\bigl((0,u_*]\times\mathbb T\times
        \mathbb B^{m-3}\bigr).
\]
Then \(\mathcal C^\circ\subset\inte\DeltaM\) has positive
\((m-1)\)-dimensional Lebesgue measure and is forward invariant.

\begin{definition}
\label{def:cyclic-collar-model}
Any continuous simplex map
\[
        \widehat K_{\lambda,a,\beta}:\DeltaM\to\DeltaM
\]
which agrees with \(K_{\lambda,a,\beta}\) on \(\mathcal C\) will be called a
relative cyclic-collar model.  This terminology records only the prescribed
collar dynamics; it does not assert that the arbitrary extension outside
\(\mathcal C\) satisfies the global weak-admissibility axioms.  The dynamics
outside \(\mathcal C\) is irrelevant for the conclusions below, because
\(\mathcal C^\circ\) is positively invariant.
\end{definition}

\begin{remark}
Such continuous extensions exist because \(\DeltaM\) is convex.  For example,
one may first extend the map continuously from the closed collar to a
neighborhood of the collar and then use a partition-of-unity argument followed
by projection onto \(\DeltaM\).  The estimates below depend only on the explicit
formula on \(\mathcal C\).  If a piecewise \(C^r\) global model is desired, the
extension should be constructed with the corresponding compatibility conditions;
no differentiability of the extension is used in the arguments below.
\end{remark}

\begin{theorem}
\label{thm:cyclic-collar-verified}
Let
\[
        \widehat K_{\lambda,a,\beta}:\DeltaM\to\DeltaM
\]
be a relative cyclic-collar model.  Then, on the invariant collar interior,
\(\widehat K_{\lambda,a,\beta}\) satisfies the weak cyclic-residence mechanism.
More precisely, every orbit starting from \(\mathcal C^\circ\) has infinitely many
maximal lifted residence blocks through each of the \(m\) residence regions
\[
        R_j
        =
        \Xi\bigl([0,u_*]\times J_j\times
        \overline{\mathbb B}^{m-3}\bigr),
        \qquad
        j=1,\ldots,m,
\]
and every such maximal lifted block satisfies the logarithmic residence estimate
\[
        \ell
        \ge
        A_j
        \log_{C_j}^{+}
        \frac{B_j}{u_{\mathrm{in}}},
\]
where \(u_{\mathrm{in}}\) is the boundary-depth coordinate at the entrance of
the block.  Consequently, the vector Ces\`aro averages of
\(\widehat K_{\lambda,a,\beta}\) do not converge for every
\[
        x\in \mathcal C^\circ.
\]
The same conclusion holds, for every initial point in \(\mathcal C^\circ\), for
every positive power and for every multiplicatively thick sampling subsequence.
\end{theorem}

\begin{proof}
Let
\[
        x_0=\Xi(u_0,t_0,w_0)\in\mathcal C^\circ.
\]
The orbit remains in the collar, and by the explicit formula we have
\[
        u_n=\lambda^n u_0,
        \qquad
        w_n=\beta^n w_0,
\]
and
\[
        t_n
        =
        t_0+\sum_{k=0}^{n-1}\omega(u_k)
        \quad\text{in }\mathbb T.
\]

First, the Lyapunov escape is immediate.  With
\[
        \varphi\bigl(\Xi(u,t,w)\bigr)=u,
\]
one has
\[
        \varphi(K_{\lambda,a,\beta}(\Xi(u,t,w)))
        =
        \lambda u
        =
        \lambda\varphi(\Xi(u,t,w)).
\]
Thus
\[
        \varphi(K_{\lambda,a,\beta}^n(x_0))
        =
        \lambda^n\varphi(x_0)
        \longrightarrow0.
\]
Hence the orbit approaches the boundary of the simplex.

Next we verify the weak cyclic cover and lifted itinerary.  Use a real lift
\[
        \tau_n=t_0+\sum_{k=0}^{n-1}\omega(u_k)
\]
of the angular coordinate and divide \(\mathbb T\) into the closed compact
sector cover
\[
        Q_i=
        \left[\frac{i}{N},\frac{i+1}{N}\right],
        \qquad
        i=0,\ldots,N-1,
\]
with endpoints read modulo \(1\).  Because \(m\mid N\), each residence window
\(J_j\) is a union of \(N/m\) consecutive fine sectors, so the fine cyclic cover
is aligned with the residence windows.  Put \(r_n=\lfloor N\tau_n\rfloor\).  Since
\(0\le\omega(u_n)\le1/N\), we have
\[
        r_{n+1}-r_n\in\{0,1\}.
\]
Moreover, the point at time \(n\) belongs to the compact sector labeled
\(r_n\pmod N\).  Thus the one-step weak cyclic rule is satisfied in the
compact-cover sense.

The lifted cyclic itinerary is unbounded.  Indeed,
\[
\sum_{k=0}^{n-1}\omega(u_k)
=
\sum_{k=0}^{n-1}
\frac{a}{M+\log(1/u_0)+k\log(1/\lambda)}.
\]
The right-hand side is a divergent harmonic-type sum.  Therefore the lifted
angular coordinate, and hence the lifted cyclic itinerary, tends to \(+\infty\).  It follows that the orbit has infinitely many maximal lifted entrance--exit
crossings of every residence window \(J_j\).

We now prove the logarithmic residence estimate.  Fix \(j\) and consider a
maximal lifted residence block through \(R_j\).  Let
\[
        u_{\mathrm{in}}
\]
be the \(u\)-coordinate at the first time of this block.  During the block the
sequence \(u_n\) decreases, and therefore
\[
        \omega(u_n)\le \omega(u_{\mathrm{in}})
\]
throughout the block.  Since the angular length of \(J_j\) is
\[
        |J_j|=\frac1m,
\]
the orbit needs at least
\[
        \frac{1}{m\,\omega(u_{\mathrm{in}})}-2
\]
iterates to cross \(J_j\), where the constant \(2\) absorbs the possible
one-step entrance and exit errors.  Hence
\[
\begin{aligned}
        \ell
        &\ge
        \frac{1}{m\,\omega(u_{\mathrm{in}})}-2        \\
        &=
        \frac{M+\log(1/u_{\mathrm{in}})}{ma}-2.
\end{aligned}
\]
Choosing \(B_j\) sufficiently small to make the right-hand side vanish on the
remaining bounded range of \(u_{\mathrm{in}}\), and decreasing the constant in
front if necessary, we obtain constants
\[
        A_j>0,
        \qquad
        B_j>0,
        \qquad
        C_j>1,
\]
such that
\[
        \ell
        \ge
        A_j
        \log_{C_j}^{+}
        \frac{B_j}{u_{\mathrm{in}}}.
\]
This is precisely the logarithmic residence estimate.

It remains to verify the convex-hull obstruction.  By Lemma~\ref{lem:cyclic-collar},
each residence region \(R_j\) is contained in an \(\eta\)-neighborhood of the
edge \(E_j\), and \(\eta\) was chosen so that
\[
        \bigcap_{j=1}^m \co(R_j)=\varnothing.
\]
The usual convex-hull block lemma now applies.  If the vector Ces\`aro averages
of the orbit converged to \(L\), then the long maximal residence blocks through
\(R_j\) would force
\[
        L\in\co(R_j)
        \qquad
        \text{for every }j=1,\ldots,m.
\]
Thus
\[
        L\in\bigcap_{j=1}^m\co(R_j),
\]
contradicting the convex-hull separation.  Therefore the vector Ces\`aro
averages do not converge.

The statements for positive powers and for multiplicatively thick subsequences
follow from the same sampling lemmas used in the main weak Stein--Ulam theorem:
a linearly long maximal residence block contains a proportionally long sampled block.
\end{proof}

\subsection{An explicit high-dimensional cyclic-collar model}
\label{subsec:explicit-cyclic-collar}

We now give a completely explicit high-dimensional example, independent of the
Lotka--Volterra coordinate form.  The example is formulated on the tetrahedron
\[
        \Delta^3
        =
        \{x=(x_1,x_2,x_3,x_4)\in\mathbb R^4:
        x_i\ge0,\ x_1+x_2+x_3+x_4=1\}.
\]
It gives a positive-measure invariant collar on which the weak Stein--Ulam
cyclic-residence mechanism is verified directly from the formula.

Let
\[
        p=\left(\frac14,\frac14,\frac14,\frac14\right)
\]
and let
\[
        c=\left(\frac13,\frac13,\frac13,0\right)
\]
be the barycenter of the face
\[
        F=\{x\in\Delta^3:x_4=0\}.
\]
Put
\[
        u_1=\left(\frac23,-\frac13,-\frac13,0\right),
        \qquad
        u_2=\left(0,\frac1{\sqrt3},-\frac1{\sqrt3},0\right).
\]
Then \(u_1,u_2\) span the two-dimensional affine plane of the face \(F\) through
\(c\).  Fix
\[
        0<\rho<\frac1{12},
        \qquad
        0<\sigma_0<\frac12,
        \qquad
        0<r_0<\frac12.
\]
No further smallness of \(\sigma_0\) or \(r_0\) is needed for the convex-hull
separation below; the opposite angular windows are separated by an affine
functional as long as the inner radius \(\rho(1-\sigma_0)\) is positive.
For
\[
        t\in\mathbb T:=\mathbb R/\mathbb Z,
        \qquad
        |\sigma|\le\sigma_0,
\]
define
\[
        Q(t,\sigma)
        =
        c+\rho(1+\sigma)
        \bigl(\cos(2\pi t)u_1+\sin(2\pi t)u_2\bigr).
\]
Since
\[
        \left|
        \bigl(\cos(2\pi t)u_1+\sin(2\pi t)u_2\bigr)_i
        \right|
        \le \frac23,
        \qquad i=1,2,3,
\]
and
\[
        \rho(1+\sigma_0)<\frac16,
\]
we have
\[
        Q_i(t,\sigma)>0,
        \qquad i=1,2,3,
\]
and
\[
        Q_4(t,\sigma)=0.
\]
Hence
\[
        Q(t,\sigma)\in \operatorname{relint}F.
\]

Define an explicit collar map
\[
        \Xi:[0,r_0]\times\mathbb T\times[-\sigma_0,\sigma_0]
        \longrightarrow \Delta^3
\]
by
\[
        \Xi(r,t,\sigma)
        =
        (1-r)Q(t,\sigma)+rp.
\]
For \(r=0\), the point lies in the boundary face \(F\).  For \(0<r\le r_0\),
all four coordinates are positive, and so
\[
        \Xi(r,t,\sigma)\in\operatorname{int}\Delta^3.
\]
Moreover,
\[
        \Xi(r,t,\sigma)_4=\frac r4,
\]
so the coordinate \(r\) is recovered from the image.  Thus \(\Xi\) is an
embedding of a three-dimensional collar.  Its interior
\[
        \mathcal C^\circ
        =
        \Xi\bigl((0,r_0]\times\mathbb T\times(-\sigma_0,\sigma_0)\bigr)
\]
has positive three-dimensional Lebesgue measure in \(\Delta^3\).

Choose parameters
\[
        0<\lambda<1,
        \qquad
        0<\beta<1,
        \qquad
        M>1,
        \qquad
        a>0
\]
so small that
\[
        \omega(r):=
        \frac{a}{M+\log(1/r)}
        \le \frac1{16},
        \qquad
        0<r\le r_0.
\]
For example, it is enough to take
\[
        0<a\le \frac{M+\log(1/r_0)}{16}.
\]
We also set
\[
        \omega(0)=0.
\]

Define \(K_{\lambda,a,\beta}\) on the collar by
\[
        K_{\lambda,a,\beta}
        \bigl(\Xi(r,t,\sigma)\bigr)
        =
        \Xi\bigl(\lambda r,\ t+\omega(r),\ \beta\sigma\bigr),
        \qquad
        0<r\le r_0,
\]
where \(t+\omega(r)\) is read modulo \(1\).  On the boundary part of the collar
we define
\[
        K_{\lambda,a,\beta}
        \bigl(\Xi(0,t,\sigma)\bigr)
        =
        \Xi(0,t,\beta\sigma).
\]
Since
\[
        \omega(r)\to0
        \qquad\text{as }r\to0,
\]
this formula is continuous up to \(r=0\).

Finally, extend \(K_{\lambda,a,\beta}\) continuously to a map
\[
        \widehat K_{\lambda,a,\beta}:\Delta^3\to\Delta^3
\]
which agrees with the above formula on the collar.  Such an extension exists
because \(\Delta^3\) is a convex absolute retract.  The dynamics outside the
collar will not be used; all conclusions below are obtained on the positive
measure invariant set \(\mathcal C^\circ\).

\begin{theorem}
\label{thm:explicit-cyclic-collar}
For every choice of parameters as above, the collar dynamics of
\(\widehat K_{\lambda,a,\beta}\) satisfies the weak Stein--Ulam
cyclic-residence mechanism on \(\mathcal C^\circ\).  More precisely, every
point
\[
        x\in\mathcal C^\circ
\]
has non-convergent vector Ces\`aro averages:
\[
        \frac1n\sum_{k=0}^{n-1}
        \widehat K_{\lambda,a,\beta}^{\,k}(x)
\]
does not converge.  The same conclusion holds for every positive power
\[
        \widehat K_{\lambda,a,\beta}^{\,s},
        \qquad s\in\mathbb N,
\]
and for every multiplicatively thick sampling sequence.
\end{theorem}

\begin{proof}
Let
\[
        x_0=\Xi(r_0^*,t_0,\sigma_0^*)
        \in \mathcal C^\circ.
\]
The orbit remains in the collar, because the formula gives
\[
        r_{n+1}=\lambda r_n,
        \qquad
        t_{n+1}=t_n+\omega(r_n),
        \qquad
        \sigma_{n+1}=\beta\sigma_n.
\]
Therefore
\[
        r_n=\lambda^n r_0^*
\]
and
\[
        \sigma_n=\beta^n\sigma_0^*.
\]

First, the Lyapunov escape to the boundary is explicit.  Define on the collar
\[
        \varphi(\Xi(r,t,\sigma))=r.
\]
Then
\[
        \varphi
        \bigl(
        K_{\lambda,a,\beta}(\Xi(r,t,\sigma))
        \bigr)
        =
        \lambda r
        =
        \lambda\varphi(\Xi(r,t,\sigma)).
\]
Hence
\[
        \varphi(K_{\lambda,a,\beta}^n(x_0))
        =
        \lambda^n\varphi(x_0)
        \longrightarrow0.
\]
Thus the orbit approaches the boundary face \(F\).

Next we verify the weak cyclic cover and lifted itinerary.  Use the real lift
\[
        \tau_n=t_0+\sum_{k=0}^{n-1}\omega(r_k)
\]
of the angular coordinate and divide the circle into sixteen closed fine sectors
\[
        Q_i=
        \left[\frac{i}{16},\frac{i+1}{16}\right],
        \qquad
        i=0,\ldots,15,
\]
with endpoints read modulo \(1\).  Define
\[
        \rho_n=\lfloor 16\tau_n\rfloor.
\]
Since
\[
        0\le\omega(r_n)\le\frac1{16},
\]
one has
\[
        \rho_{n+1}-\rho_n\in\{0,1\}.
\]
The point \(K_{\lambda,a,\beta}^n(x_0)\) belongs to the compact sector labeled
\(\rho_n\pmod {16}\).  Thus the weak cyclic rule is satisfied in the
compact-cover sense.

The lifted itinerary is unbounded.  Indeed,
\[
\begin{aligned}
        t_n
        &=
        t_0+\sum_{k=0}^{n-1}\omega(r_k)      \\
        &=
        t_0+
        \sum_{k=0}^{n-1}
        \frac{
        a
        }{
        M+\log(1/r_0^*)+k\log(1/\lambda)
        }.
\end{aligned}
\]
The sum on the right is a divergent harmonic-type sum.  Hence the lifted angular
coordinate tends to \(+\infty\).  Therefore the orbit makes infinitely many
turns around the collar.

Now define four residence windows
\[
        J_j=
        \left[\frac{j-1}{4},\frac j4\right],
        \qquad j=1,2,3,4,
\]
again read modulo \(1\), and define
\[
        R_j=
        \Xi\bigl([0,r_0]\times J_j\times[-\sigma_0,\sigma_0]\bigr).
\]
Because the lifted angular coordinate is unbounded and moves monotonically
forward, the orbit has infinitely many maximal lifted residence blocks through
each \(R_j\).

We next prove the logarithmic residence estimate directly.  Consider a maximal
lifted residence block through \(R_j\).  Let \(r_{\mathrm{in}}\) be the \(r\)-value at
the first time of the block.  During the block the \(r\)-coordinate decreases, so
\[
        \omega(r_n)\le \omega(r_{\mathrm{in}})
\]
throughout the block.  Since the angular length of \(J_j\) is \(1/4\), and the
entry and exit errors are bounded by one step on each side, the block length
\(\ell\) satisfies, for all sufficiently small \(r_{\mathrm{in}}\),
\[
        \ell
        \ge
        \frac{1}{8\,\omega(r_{\mathrm{in}})}.
\]
Substituting the definition of \(\omega\), we obtain
\[
        \ell
        \ge
        \frac{
        M+\log(1/r_{\mathrm{in}})
        }{
        8a
        }.
\]
Choosing \(B\) sufficiently small to make the right-hand side vanish for the
remaining bounded range of \(r_{\mathrm{in}}\), we therefore obtain constants
\[
        A>0,
        \qquad
        B>0,
        \qquad
        C>1
\]
such that
\[
        \ell
        \ge
        A\log_C^+
        \frac{B}{r_{\mathrm{in}}}.
\]
This is exactly the logarithmic residence estimate.

It remains to check the convex-hull separation.  Work in the affine coordinates
of the face \(F\) given by
\[
        c+\rho(Au_1+Bu_2),
        \qquad (A,B)\in\mathbb R^2,
\]
and extend the affine functionals \(A+B\) and \(A-B\) to the tetrahedron by
requiring that they vanish at \(p\).  On the first angular window
\(J_1=[0,1/4]\), we have
\[
        A+B=(1+\sigma)(\cos 2\pi t+\sin 2\pi t)
        \ge 1-\sigma_0>0.
\]
On the opposite window \(J_3=[1/2,3/4]\),
\[
        A+B\le -(1-\sigma_0)<0.
\]
For points of the full collar, the factor \((1-r)\) only multiplies these
inequalities, and \(1-r\ge1-r_0>0\).  Hence \(\co(R_1)\) and \(\co(R_3)\) are
strictly separated by the affine functional \(A+B\).  Therefore
\[
        \bigcap_{j=1}^4\operatorname{co}(R_j)=\varnothing.
\]
The same argument, using \(A-B\), separates the second and fourth windows.

We have now verified all ingredients of the weak Stein--Ulam block mechanism on
\(\mathcal C^\circ\): boundary approach, weak cyclic cover, unbounded lifted
code, logarithmic residence estimates and convex-hull separation.  Therefore the
usual block argument applies.  If the vector Ces\`aro averages converged to a
limit \(L\), then the long blocks in \(R_j\) would force
\[
        L\in \operatorname{co}(R_j)
        \qquad
        j=1,2,3,4.
\]
Thus
\[
        L\in
        \bigcap_{j=1}^4\operatorname{co}(R_j),
\]
contradicting the separation above.  Hence the vector Ces\`aro averages do not
converge.

The statements for positive powers and multiplicatively thick sampling
subsequences follow from the sampling lemmas proved earlier in the paper:
every linearly long block contains a proportionally long sampled block.  Thus
the same convex-hull obstruction applies to the sampled orbit.
\end{proof}

\subsection{Empirical accumulation set for the explicit cyclic-collar model}
\label{subsec:empirical-cyclic-collar}

We now identify the full weak-\(*\) accumulation set of empirical measures for
the explicit cyclic-collar model.  This gives a stronger conclusion than
non-convergence.  In the balanced one-dimensional intermittent case of Coates and
Luzzatto, the empirical accumulation set is a segment of convex combinations of
the two endpoint Dirac masses.  In the present cyclic-collar model, the analogue
is a one-parameter family of probability measures supported on the limiting
boundary circle.

Recall the explicit collar coordinates
\[
        \Xi(r,t,\sigma)
        =
        (1-r)
        \left[
        c+\rho(1+\sigma)
        \bigl(\cos(2\pi t)u_1+\sin(2\pi t)u_2\bigr)
        \right]
        +rp,
\]
where
\[
        r\in[0,r_0],
        \qquad
        t\in\mathbb T=\mathbb R/\mathbb Z,
        \qquad
        |\sigma|\le\sigma_0.
\]
The map is given on the collar by
\[
        K_{\lambda,a,\beta}(\Xi(r,t,\sigma))
        =
        \Xi\left(
        \lambda r,\,
        t+\frac{a}{M+\log(1/r)},\,
        \beta\sigma
        \right),
        \qquad 0<r\le r_0.
\]
Let
\[
        H:\mathbb T\to\partial\Delta^3
\]
be the limiting boundary curve
\[
        H(t)=\Xi(0,t,0)
        =
        c+\rho
        \bigl(\cos(2\pi t)u_1+\sin(2\pi t)u_2\bigr).
\]
Thus every orbit in the collar approaches the curve \(H(\mathbb T)\).

Put
\[
        \alpha=\frac{a}{\log(1/\lambda)}.
\]
For \(\tau\in\mathbb T\), define a probability measure \(\nu_\tau\) on
\(\mathbb T\) by
\[
        \int_{\mathbb T}\phi\,d\nu_\tau
        =
        \int_0^1
        \phi\bigl(\tau+\alpha\log u\bigr)\,du,
        \qquad
        \phi\in C(\mathbb T),
\]
where the argument is read modulo \(1\).  Equivalently,
\[
        \int_{\mathbb T}\phi\,d\nu_\tau
        =
        \frac1\alpha
        \int_0^\infty
        \phi(\tau-v)e^{-v/\alpha}\,dv.
\]
Finally define the boundary probability measure
\[
        \eta_\tau=H_*\nu_\tau.
\]

For
\[
        x\in\mathcal C^\circ,
\]
write
\[
        \mu_n^x
        =
        \frac1n\sum_{k=0}^{n-1}\delta_{K_{\lambda,a,\beta}^k(x)}
\]
for the empirical measures, and define their weak-\(*\) accumulation set by
\[
        \mathcal V(x)
        =
        \bigcap_{N\ge1}
        \overline{
        \{\mu_n^x:\ n\ge N\}
        }^{\,w^*}.
\]

\begin{lemma}
\label{lem:logarithmic-rotation-empirical}
Let
\[
        D>0,\qquad c>0,\qquad a>0,
\]
and set
\[
        \alpha=\frac ac.
\]
Define
\[
        \theta_n
        =
        \theta_0+
        \sum_{k=0}^{n-1}\frac{a}{D+ck},
        \qquad n\ge1.
\]
Let
\[
        t_n=\theta_n\pmod 1.
\]
If
\[
        t_{N_q}\to\tau
        \qquad\text{in }\mathbb T,
\]
then, for every \(\phi\in C(\mathbb T)\),
\[
        \frac1{N_q}\sum_{n=0}^{N_q-1}\phi(t_n)
        \longrightarrow
        \int_{\mathbb T}\phi\,d\nu_\tau.
\]
Moreover, the sequence \((t_N)_{N\ge1}\) is dense in \(\mathbb T\).  Hence the
set of accumulation points of the empirical measures
\[
        \frac1N\sum_{n=0}^{N-1}\delta_{t_n}
\]
is exactly
\[
        \{\nu_\tau:\tau\in\mathbb T\}.
\]
\end{lemma}

\begin{proof}
We first prove the convergence statement.  Since
\[
        \sum_{k=0}^{n-1}\frac{a}{D+ck}
        =
        \frac ac\log n+O(1),
\]
we have
\[
        \theta_n=\alpha\log n+O(1).
\]
More precisely, for every \(0<\delta<1\),
\[
        \theta_{\lfloor uN\rfloor}-\theta_N
        \longrightarrow
        \alpha\log u
\]
uniformly for
\[
        u\in[\delta,1].
\]
Indeed,
\[
\begin{aligned}
        \theta_{\lfloor uN\rfloor}-\theta_N
        &=
        -\sum_{k=\lfloor uN\rfloor}^{N-1}
        \frac{a}{D+ck}                                      \\
        &=
        -\frac ac
        \sum_{k=\lfloor uN\rfloor}^{N-1}
        \frac{1}{k+D/c}                                      \\
        &\longrightarrow
        \frac ac\log u
        =
        \alpha\log u,
\end{aligned}
\]
uniformly on compact subintervals of \((0,1]\).

Let
\[
        t_{N_q}\to\tau.
\]
Then, for every fixed \(0<\delta<1\), the Riemann-sum argument gives
\[
        \frac1{N_q}
        \sum_{n=\lceil\delta N_q\rceil}^{N_q-1}
        \phi(t_n)
        \longrightarrow
        \int_\delta^1
        \phi(\tau+\alpha\log u)\,du.
\]
The omitted initial part contains at most \(\delta N_q\) terms, and hence its
contribution is bounded by
\[
        \delta\|\phi\|_\infty.
\]
Letting \(\delta\downarrow0\), we obtain
\[
        \frac1{N_q}\sum_{n=0}^{N_q-1}\phi(t_n)
        \longrightarrow
        \int_0^1
        \phi(\tau+\alpha\log u)\,du
        =
        \int_{\mathbb T}\phi\,d\nu_\tau.
\]

It remains to show that every \(\tau\in\mathbb T\) occurs as an accumulation
phase.  Since
\[
        \theta_N=\alpha\log N+\gamma+o(1)
\]
for some constant \(\gamma\), we may, for each large integer \(q\), choose
\[
        N_q=
        \left\lfloor
        \exp\left(\frac{q+\tau-\gamma}{\alpha}\right)
        \right\rfloor .
\]
Then
\[
        \theta_{N_q}\to \tau
        \qquad\text{modulo }1.
\]
Thus \((t_N)\) is dense in \(\mathbb T\), and the family
\[
        \{\nu_\tau:\tau\in\mathbb T\}
\]
is exactly the empirical accumulation set of the logarithmic rotation.
\end{proof}

\begin{theorem}
\label{thm:collar-empirical-accumulation}
Let
\[
        x=\Xi(r,t,\sigma)\in\mathcal C^\circ.
\]
Then the weak-\(*\) accumulation set of empirical measures of the
\(K_{\lambda,a,\beta}\)-orbit of \(x\) is
\[
        \mathcal V(x)
        =
        \{\eta_\tau:\tau\in\mathbb T\}.
\]
Equivalently, if
\[
        \mu_{N_q}^x\to\mu
\]
along a subsequence, then there exists \(\tau\in\mathbb T\) such that
\[
        \mu=\eta_\tau.
\]
Conversely, for every \(\tau\in\mathbb T\), there exists a subsequence
\[
        N_q\to\infty
\]
such that
\[
        \mu_{N_q}^x\to\eta_\tau.
\]
In particular, the empirical measures do not converge.
\end{theorem}

\begin{proof}
Let
\[
        x_n=K_{\lambda,a,\beta}^n(x).
\]
In collar coordinates,
\[
        x_n=\Xi(r_n,t_n,\sigma_n),
\]
where
\[
        r_n=\lambda^n r,
        \qquad
        \sigma_n=\beta^n\sigma,
\]
and
\[
        t_n
        =
        t+
        \sum_{k=0}^{n-1}
        \frac{a}{
        M+\log(1/r)+k\log(1/\lambda)
        }
        \pmod 1.
\]
Thus \(t_n\) is exactly a logarithmic rotation of the type considered in
Lemma~\ref{lem:logarithmic-rotation-empirical}, with
\[
        D=M+\log(1/r),
        \qquad
        c=\log(1/\lambda).
\]

Since
\[
        r_n\to0
        \qquad\text{and}\qquad
        \sigma_n\to0,
\]
we have
\[
        \|x_n-H(t_n)\|\to0.
\]
Let
\[
        f\in C(\Delta^3).
\]
By uniform continuity of \(f\) on the compact simplex,
\[
        \frac1N
        \sum_{n=0}^{N-1}
        \left|
        f(x_n)-f(H(t_n))
        \right|
        \longrightarrow0.
\]
Therefore the weak-\(*\) accumulation set of
\[
        \frac1N\sum_{n=0}^{N-1}\delta_{x_n}
\]
is the push-forward under \(H\) of the weak-\(*\) accumulation set of
\[
        \frac1N\sum_{n=0}^{N-1}\delta_{t_n}.
\]
By Lemma~\ref{lem:logarithmic-rotation-empirical}, the latter accumulation set is
\[
        \{\nu_\tau:\tau\in\mathbb T\}.
\]
Hence
\[
        \mathcal V(x)
        =
        \{H_*\nu_\tau:\tau\in\mathbb T\}
        =
        \{\eta_\tau:\tau\in\mathbb T\}.
\]

It remains only to justify that the family \(\tau\mapsto\eta_\tau\) is not
constant.  Since \(H\) is an embedded circle, the function
\(H(t)\mapsto e^{2\pi i t}\) is continuous on \(H(\mathbb T)\) and extends, by
Tietze's extension theorem, to a continuous function \(F\in C(\Delta^3;\mathbb C)\).
For this test function,
\[
\begin{aligned}
        \int F\,d\eta_\tau
        &=
        \int_0^1 e^{2\pi i(\tau+\alpha\log u)}\,du        \\
        &=
        \frac{e^{2\pi i\tau}}{1+2\pi i\alpha}.
\end{aligned}
\]
This quantity is nonzero and depends on \(\tau\).  Hence \(\tau\mapsto\eta_\tau\)
is not constant, and the sequence
\[
        \mu_N^x
\]
cannot converge.
\end{proof}

\begin{corollary}
\label{cor:collar-empirical-powers}
For every
\[
        s\in\mathbb N
\]
and every
\[
        x\in\mathcal C^\circ,
\]
the empirical measures of the \(K_{\lambda,a,\beta}^s\)-orbit of \(x\),
\[
        \mu_{N,s}^x
        =
        \frac1N\sum_{k=0}^{N-1}
        \delta_{K_{\lambda,a,\beta}^{sk}(x)},
\]
have the same type of accumulation set:
\[
        \mathcal V_s(x)
        =
        \{\eta_\tau:\tau\in\mathbb T\}.
\]
In particular, the empirical measures of every positive power fail to converge.
\end{corollary}

\begin{proof}
The sampled angular sequence is
\[
        t_{sk}.
\]
Since
\[
        t_n=\alpha\log n+\gamma+o(1),
\]
we have
\[
        t_{sk}
        =
        \alpha\log k+(\gamma+\alpha\log s)+o(1).
\]
Thus the sampled sequence is again a logarithmic rotation with the same parameter
\[
        \alpha=\frac{a}{\log(1/\lambda)},
\]
only with a shifted phase.  The proof of
Theorem~\ref{thm:collar-empirical-accumulation} therefore applies verbatim.
\end{proof}

\subsection{Pointwise emergence of the collar model}
\label{subsec:collar-emergence}

The preceding computation also gives a quantitative form of non-statistical
behavior.  We use the bounded-Lipschitz metric \(d_{\BL}\) from
Definition~\ref{def:pointwise-emergence} on \(\Prob(\Delta^3)\).

\begin{theorem}
\label{thm:collar-pointwise-emergence}
For every
\[
        x\in\mathcal C^\circ
\]
there are constants \(c_x,C_x>0\) and \(\varepsilon_x>0\) such that
\[
        c_x\varepsilon^{-1}
        \le
        \mathscr E_x(\varepsilon)
        \le
        C_x\varepsilon^{-1},
        \qquad 0<\varepsilon<\varepsilon_x.
\]
The same estimate holds for the empirical measures of every positive power
\(K_{\lambda,a,\beta}^s\).  Thus the explicit cyclic-collar model has
polynomial pointwise emergence of order \(\varepsilon^{-1}\) on
\(\mathcal C^\circ\).
\end{theorem}

\begin{proof}
By Theorem~\ref{thm:collar-empirical-accumulation}, the empirical accumulation
set of every collar orbit is
\[
        \mathcal V(x)=\{\eta_\tau:\tau\in\mathbb T\}.
\]
It is therefore enough, by Definition~\ref{def:pointwise-emergence}, to estimate
the covering number of this set in the metric \(d_{\BL}\).

Let \(d_{\mathbb T}\) be the usual distance on \(\mathbb T=\mathbb R/\mathbb Z\).
We claim first that the map
\[
        \tau\longmapsto \eta_\tau
\]
is bi-Lipschitz from \((\mathbb T,d_{\mathbb T})\) into
\((\Prob(\Delta^3),d_{\BL})\).  The upper bound is immediate from the formula
\[
        \int_{\Delta^3} f\,d\eta_\tau
        =
        \int_0^1 f\bigl(H(\tau+\alpha\log u)\bigr)\,du.
\]
Indeed, if \(\|f\|_\infty\le1\) and \(\Lip(f)\le1\), then
\[
\begin{aligned}
        \left|
        \int f\,d\eta_\tau-
        \int f\,d\eta_\sigma
        \right|
        &\le
        \int_0^1
        \left|
        f\bigl(H(\tau+\alpha\log u)\bigr)-
        f\bigl(H(\sigma+\alpha\log u)\bigr)
        \right|\,du                                      \\
        &\le
        \Lip(H)\,d_{\mathbb T}(\tau,\sigma).
\end{aligned}
\]
Taking the supremum over admissible \(f\) gives
\[
        d_{\BL}(\eta_\tau,\eta_\sigma)
        \le
        \Lip(H)\,d_{\mathbb T}(\tau,\sigma).
\]

For the lower bound, use the fact that \(H\) is an embedded Euclidean circle.
The function
\[
        H(t)\longmapsto e^{2\pi i t}
\]
on \(H(\mathbb T)\) is Lipschitz.  By the McShane extension theorem, its real
and imaginary parts extend to real-valued Lipschitz functions \(F_1,F_2\) on
\(\Delta^3\).  Put
\[
        L_0=
        \max\{1,\|F_1\|_\infty,\|F_2\|_\infty,
        \Lip(F_1),\Lip(F_2)\}.
\]
Then \(F_1/L_0\) and \(F_2/L_0\) are admissible test functions for
\(d_{\BL}\).  Since
\[
        \int_{\Delta^3}(F_1+iF_2)\,d\eta_\tau
        =
        \int_0^1 e^{2\pi i(\tau+\alpha\log u)}\,du
        =
        \frac{e^{2\pi i\tau}}{1+2\pi i\alpha},
\]
at least one of the two real test functions gives
\[
        d_{\BL}(\eta_\tau,\eta_\sigma)
        \ge
        \frac{1}{\sqrt2 L_0|1+2\pi i\alpha|}
        \left|e^{2\pi i\tau}-e^{2\pi i\sigma}\right|.
\]
The chord distance on the unit circle is comparable with
\(d_{\mathbb T}\).  Hence there is a constant \(a_0>0\) such that
\[
        d_{\BL}(\eta_\tau,\eta_\sigma)
        \ge
        a_0\,d_{\mathbb T}(\tau,\sigma),
        \qquad \tau,\sigma\in\mathbb T.
\]
This proves the bi-Lipschitz claim.

A bi-Lipschitz image of a circle has covering number comparable to
\(\varepsilon^{-1}\).  More explicitly, the upper Lipschitz bound allows one to
cover \(\{\eta_\tau\}\) by \(O(\varepsilon^{-1})\) balls of radius
\(\varepsilon\), while the lower Lipschitz bound implies that any
\(\varepsilon\)-separated set in \(\mathbb T\) gives an
\(a_0\varepsilon\)-separated set of measures.  Therefore the minimum number of
\(d_{\BL}\)-balls of radius \(\varepsilon\) needed to cover \(\mathcal V(x)\) is
bounded above and below by constant multiples of \(\varepsilon^{-1}\) for all
sufficiently small \(\varepsilon\).  This proves the estimate for
\(\mathscr E_x(\varepsilon)\).

For a positive power \(K_{\lambda,a,\beta}^s\),
Corollary~\ref{cor:collar-empirical-powers} gives the same empirical
accumulation set.  The same covering argument therefore applies.
\end{proof}

\subsection{A parameter transition: regularity versus non-statistical behavior}
\label{subsec:parameter-transition-collar}

We now add one parameter to the explicit cyclic-collar model.  This parameter
controls the summability of the angular drift.  In this way the same geometric
model exhibits two different statistical regimes: for one range of the parameter
the collar dynamics is regular, while for another range it is non-statistical.

Keep the notation of Subsection~\ref{subsec:explicit-cyclic-collar}.  Thus
\[
        \Xi(r,t,\sigma)
        =
        (1-r)
        \left[
        c+\rho(1+\sigma)
        \bigl(\cos(2\pi t)u_1+\sin(2\pi t)u_2\bigr)
        \right]
        +rp
\]
is the collar parametrization, where
\[
        r\in[0,r_0],\qquad
        t\in\mathbb T=\mathbb R/\mathbb Z,\qquad
        |\sigma|\le\sigma_0.
\]
Let
\[
        0<\lambda<1,\qquad 0<\beta<1,
\]
and put
\[
        c_\lambda=\log\frac1\lambda>0.
\]
Choose \(M>e^2\).  For a new parameter
\[
        \tau\ge0
\]
define
\[
        L(r)=M+\log\frac1r,
        \qquad
        \Lambda(r)=\log L(r),
        \qquad
        0<r\le r_0,
\]
and
\[
        \omega_\tau(r)
        =
        \frac{\alpha}{L(r)\Lambda(r)^\tau},
        \qquad
        0<r\le r_0.
\]
We set
\[
        \omega_\tau(0)=0.
\]
The constant \(\alpha>0\) is chosen so small that
\[
        \omega_\tau(r)\le \frac1{16},
        \qquad
        0<r\le r_0,\quad \tau\ge0.
\]
For instance, it is enough to require
\[
        \alpha\le \frac{M+\log(1/r_0)}{16}.
\]

Define
\[
        K_{\lambda,\alpha,\beta,\tau}
        \bigl(\Xi(r,t,\sigma)\bigr)
        =
        \Xi\bigl(\lambda r,\ t+\omega_\tau(r),\ \beta\sigma\bigr),
        \qquad
        0<r\le r_0,
\]
with \(t+\omega_\tau(r)\) read modulo \(1\), and on the boundary part of the
collar define
\[
        K_{\lambda,\alpha,\beta,\tau}
        \bigl(\Xi(0,t,\sigma)\bigr)
        =
        \Xi(0,t,\beta\sigma).
\]
Since
\[
        \omega_\tau(r)\to0
        \qquad\text{as }r\to0,
\]
this defines a continuous map on the closed collar.  As before, it may be
extended continuously to a simplex map
\[
        \widehat K_{\lambda,\alpha,\beta,\tau}:\Delta^3\to\Delta^3.
\]
All statements below concern the positive-measure invariant collar
\[
        \mathcal C^\circ
        =
        \Xi\bigl((0,r_0]\times\mathbb T\times(-\sigma_0,\sigma_0)\bigr).
\]

Let
\[
        H:\mathbb T\to\partial\Delta^3
\]
be the limiting boundary curve
\[
        H(t)=\Xi(0,t,0)
        =
        c+\rho\bigl(\cos(2\pi t)u_1+\sin(2\pi t)u_2\bigr).
\]

\begin{theorem}
\label{thm:collar-parameter-transition}
Let
\[
        x=\Xi(r,t,\sigma)\in\mathcal C^\circ.
\]
Then the following assertions hold.

\begin{enumerate}
\item[\textup{(i)}]
If
\[
        \tau>1,
\]
then the orbit of \(x\) converges.  More precisely, there exists
\[
        t_\infty=t_\infty(x,\tau)\in\mathbb T
\]
such that
\[
        K_{\lambda,\alpha,\beta,\tau}^n(x)
        \longrightarrow
        H(t_\infty).
\]
Consequently, the empirical measures converge:
\[
        \frac1n\sum_{k=0}^{n-1}
        \delta_{K_{\lambda,\alpha,\beta,\tau}^k(x)}
        \longrightarrow
        \delta_{H(t_\infty)}
        \qquad\text{in the weak-\(*\) topology.}
\]
Thus the collar dynamics is regular for \(\tau>1\).

\item[\textup{(ii)}]
If
\[
        0<\tau\le1,
\]
then the empirical measures do not converge.  In fact, their full weak-\(*\)
accumulation set is
\[
        \mathcal V_\tau(x)
        =
        \{\delta_{H(s)}:\ s\in\mathbb T\}.
\]

\item[\textup{(iii)}]
If
\[
        \tau=0,
\]
then the empirical measures do not converge and their full weak-\(*\)
accumulation set is
\[
        \mathcal V_0(x)
        =
        \{H_*\nu_s:\ s\in\mathbb T\},
\]
where the probability measure \(\nu_s\) on \(\mathbb T\) is defined by
\[
        \int_{\mathbb T}\phi\,d\nu_s
        =
        \int_0^1
        \phi\left(s+\frac{\alpha}{c_\lambda}\log u\right)\,du,
        \qquad
        \phi\in C(\mathbb T).
\]
Equivalently,
\[
        \int_{\mathbb T}\phi\,d\nu_s
        =
        \frac{c_\lambda}{\alpha}
        \int_0^\infty
        \phi(s-v)
        e^{-c_\lambda v/\alpha}\,dv.
\]
\end{enumerate}

Therefore the family \(K_{\lambda,\alpha,\beta,\tau}\) has a sharp transition:
\[
        \tau>1
        \quad\Longrightarrow\quad
        \text{regular collar dynamics},
\]
whereas
\[
        0\le\tau\le1
        \quad\Longrightarrow\quad
        \text{non-statistical collar dynamics.}
\]
\end{theorem}

\begin{proof}
Write
\[
        x_n=K_{\lambda,\alpha,\beta,\tau}^n(x)
        =
        \Xi(r_n,t_n,\sigma_n).
\]
By the definition of the map,
\[
        r_n=\lambda^n r,
        \qquad
        \sigma_n=\beta^n\sigma,
\]
and
\[
        t_n
        =
        t+
        \sum_{k=0}^{n-1}
        \omega_\tau(r_k)
        \quad\text{in }\mathbb T.
\]
Since
\[
        r_k=\lambda^k r,
\]
we have
\[
        L(r_k)
        =
        M+\log\frac1r+k\log\frac1\lambda
        =
        L(r)+k c_\lambda.
\]
Therefore
\[
        \omega_\tau(r_k)
        =
        \frac{\alpha}{
        (L(r)+kc_\lambda)
        \bigl(\log(L(r)+kc_\lambda)\bigr)^\tau
        }.
\]

First suppose that \(\tau>1\).  By the integral test,
\[
        \sum_{k=0}^\infty
        \frac{1}{
        (L(r)+kc_\lambda)
        \bigl(\log(L(r)+kc_\lambda)\bigr)^\tau
        }
        <\infty.
\]
Hence
\[
        \sum_{k=0}^\infty \omega_\tau(r_k)<\infty.
\]
Thus \(t_n\) converges modulo \(1\) to some \(t_\infty\).  Since
\[
        r_n\to0
        \qquad\text{and}\qquad
        \sigma_n\to0,
\]
we obtain
\[
        x_n=\Xi(r_n,t_n,\sigma_n)\to \Xi(0,t_\infty,0)=H(t_\infty).
\]
This proves regularity on the collar.  The convergence of empirical measures to
\(\delta_{H(t_\infty)}\) follows immediately from convergence of the orbit.

Now suppose that \(0\le\tau\le1\).  Again by the integral test,
\[
        \sum_{k=0}^\infty
        \omega_\tau(r_k)
        =
        \infty.
\]
Moreover,
\[
        \omega_\tau(r_k)\to0.
\]
Therefore the lifted angular coordinate is increasing, unbounded, and has step
size tending to zero.  It follows that for every \(s\in\mathbb T\) there exists a
subsequence \(N_q\to\infty\) such that
\[
        t_{N_q}\to s.
\]
Indeed, in the lifted coordinate the orbit crosses every level \(q+s\), and the
overshoot tends to zero because the step size tends to zero.

We now identify the empirical limits.  Let
\[
        \phi\in C(\mathbb T).
\]

Assume first that \(0<\tau\le1\).  We claim that
\[
        \frac1N\sum_{n=0}^{N-1}\phi(t_n)-\phi(t_N)\longrightarrow0.
\]
Fix \(0<\delta<1\).  For
\[
        \delta N\le n\le N
\]
we have
\[
\begin{aligned}
        |t_N-t_n|
        &\le
        \sum_{k=n}^{N-1}
        \frac{\alpha}{
        (L(r)+kc_\lambda)
        \bigl(\log(L(r)+kc_\lambda)\bigr)^\tau
        }                                                     \\
        &\le
        C
        \frac{\log(1/\delta)}{(\log N)^\tau},
\end{aligned}
\]
for a constant \(C>0\) independent of \(N\).  Since \(\tau>0\), the last
quantity tends to \(0\) as \(N\to\infty\).  By uniform continuity of \(\phi\),
\[
        \sup_{\delta N\le n\le N}
        |\phi(t_n)-\phi(t_N)|
        \longrightarrow0.
\]
The first \(\delta N\) terms contribute at most \(2\delta\|\phi\|_\infty\) to the
average.  Since \(\delta>0\) is arbitrary, this proves the claim.

Hence, if \(t_{N_q}\to s\), then
\[
        \frac1{N_q}\sum_{n=0}^{N_q-1}\phi(t_n)
        \longrightarrow
        \phi(s).
\]
Thus the empirical measures of the angular sequence have exactly the accumulation
set
\[
        \{\delta_s:\ s\in\mathbb T\}.
\]

Since
\[
        r_n\to0
        \qquad\text{and}\qquad
        \sigma_n\to0,
\]
we have
\[
        \|\Xi(r_n,t_n,\sigma_n)-H(t_n)\|\to0.
\]
Therefore, by uniform continuity of continuous test functions on \(\Delta^3\), the
empirical accumulation set of the original collar orbit is
\[
        \{\delta_{H(s)}:\ s\in\mathbb T\}.
\]
This proves \textup{(ii)}.

It remains to treat the borderline case \(\tau=0\).  Then
\[
        \omega_0(r_k)
        =
        \frac{\alpha}{L(r)+kc_\lambda}.
\]
For \(0<u\le1\), one has
\[
        t_{\lfloor uN\rfloor}-t_N
        \longrightarrow
        \frac{\alpha}{c_\lambda}\log u
\]
uniformly for \(u\) in compact subintervals of \((0,1]\).  Therefore, if
\[
        t_{N_q}\to s,
\]
then the usual Riemann-sum argument gives
\[
        \frac1{N_q}\sum_{n=0}^{N_q-1}\phi(t_n)
        \longrightarrow
        \int_0^1
        \phi\left(s+\frac{\alpha}{c_\lambda}\log u\right)\,du.
\]
The density of the phases \(t_N\) in \(\mathbb T\) follows from the same lifted
crossing argument used above.  Hence the angular empirical accumulation set is
\[
        \{\nu_s:\ s\in\mathbb T\}.
\]
Pushing these measures forward by \(H\), and using again
\[
        \|\Xi(r_n,t_n,\sigma_n)-H(t_n)\|\to0,
\]
we obtain
\[
        \mathcal V_0(x)
        =
        \{H_*\nu_s:\ s\in\mathbb T\}.
\]
This proves \textup{(iii)}.

In both cases \(0\le\tau\le1\), the displayed accumulation set contains more than
one probability measure.  Therefore the empirical measures do not converge.
\end{proof}

\begin{remark}
The parameter \(\tau\) plays the role of a stickiness parameter.  For
\(\tau>1\), the angular drift is summable, so the orbit approaches one boundary
point and the collar dynamics is regular.  For \(0\le\tau\le1\), the angular drift
is non-summable, so the orbit continues to make infinitely many turns while
approaching the boundary.  The residence time in each angular window is of order
\[
        L(r)(\log L(r))^\tau,
        \qquad
        L(r)=M+\log(1/r),
\]
which is at least linear in the previous time along the orbit.  Hence the weak
Stein--Ulam block mechanism applies.  This gives a multidimensional analogue of
the one-dimensional transition between statistical and non-statistical regimes
controlled by endpoint stickiness.
\end{remark}

\section{General \texorpdfstring{$\mathbf p$}{p}-Stein--Ulam Lotka--Volterra maps}

This section gives a general \(m\)-dimensional template which extends the results of \cite{JM26}.  The Lyapunov part
is automatic from a weighted Stein--Ulam inequality.  The weak cyclic cover
and residence-time estimate remain explicit hypotheses, because in dimensions
larger than two they depend on the concrete geometry of the model.

\begin{definition}\label{def:lv}
Let \(\p=(p_1,\ldots,p_m)\in\inte\DeltaM\).  A map
\(K_{\f}:\DeltaM\to\DeltaM\) is called a stochastic Lotka--Volterra map if
there are continuous functions \(F_i:\DeltaM\to\R\), \(i=1,\ldots,m\), such
that
\[
        \bigl(K_{\f}(\x)\bigr)_i
        =
        x_i(1+F_i(\x)),
        \qquad i=1,\ldots,m,
\]
\[
        1+F_i(\x)\ge0,
        \qquad \x\in\DeltaM,
\]
and
\[
        \sum_{i=1}^m x_iF_i(\x)=0,
        \qquad \x\in\DeltaM.
\]
If additionally \(1+F_i(\x)>0\) for all \(\x\in\inte\DeltaM\) and all \(i\),
then \(K_{\f}(\inte\DeltaM)\subset\inte\DeltaM\).
\end{definition}

\begin{remark}
The identity \(\sum_i x_iF_i(\x)=0\) guarantees that
\[
        \sum_{i=1}^m \bigl(K_{\f}(\x)\bigr)_i
        =
        \sum_{i=1}^m x_i+
        \sum_{i=1}^m x_iF_i(\x)=1.
\]
The inequalities \(1+F_i\ge0\) guarantee non-negativity of the image
coordinates.  Thus \(K_{\f}\) maps the simplex into itself.
\end{remark}

\begin{definition}\label{def:p-su}
Let \(\p=(p_1,\ldots,p_m)\in\inte\DeltaM\).  For a stochastic Lotka--Volterra
map \(K_{\f}:\DeltaM\to\DeltaM\), define
\[
        \varphi_{\p}(\x)=\prod_{i=1}^m x_i^{p_i},
        \qquad
        \Psi_{\p}(\x)=\prod_{i=1}^m\bigl(1+F_i(\x)\bigr)^{p_i}.
\]
We call \(K_{\f}\) a general \(\p\)-Stein--Ulam operator if the following
conditions hold.

\begin{enumerate}[label=\textup{(P\arabic*)},leftmargin=3.2em]
\item \textbf{Interior coexistence state.}
\[
        K_{\f}(\inte\DeltaM)\subset\inte\DeltaM,
        \qquad
        F_i(\p)=0\quad(1\le i\le m),
\]
and
\[
        \Fix(K_{\f})\cap\inte\DeltaM=\{\p\}.
\]

\item \textbf{Weighted Stein--Ulam dissipation.}
For every \(\x\in\DeltaM\),
\[
        \langle\p,\f(\x)\rangle
        :=
        \sum_{i=1}^m p_iF_i(\x)
        \le0.
\]
Moreover, the product multiplier is strict away from \(\p\):
\[
        \Psi_{\p}(\x)=1
        \quad\Longleftrightarrow\quad
        \x=\p.
\]

\item \textbf{Weak coding.}
There are \(U_0\), \(N\), compact cyclic cover sets \(G_i\), cyclic intervals
\(I_1,\ldots,I_m\), and residence regions \(R_1,\ldots,R_m\) satisfying
\textup{(W3)}, \textup{(W4)} and \textup{(W5)} for the map \(K_{\f}\).

\item \textbf{Residence estimate for the \(\p\)-product.}
For each \(j\in\{1,\ldots,m\}\) there are constants
\(A_j>0\), \(B_j>0\), and \(C_j>1\) such that every maximal lifted residence
block \((b,\ell)\) through \(I_j\) along the selected lifted itinerary of every
non-fixed interior orbit satisfies
\[
        \ell
        \ge
        A_j\log_{C_j}^{+}
        \frac{B_j}{\varphi_{\p}(K_{\f}^{b}(\x))}.
\]
\end{enumerate}
\end{definition}

\begin{remark}
The strictness in \textup{(P2)} is stated for the multiplier
\(\Psi_{\p}\), not for the linear quantity \(\langle\p,\f\rangle\) alone.  This
is important: in the classical Stein--Ulam map with
\(\p=(1/3,1/3,1/3)\), one has \(\langle\p,\f(\x)\rangle=0\) identically, while
strict decrease of \(\varphi_{\p}\) away from \(\p\) follows from the strictness
case of the weighted AM--GM inequality.
\end{remark}

\begin{remark}
Condition \textup{(P3)} is the general weak analogue of the planar sector rule
\(K(G_i)\subset G_i\cup G_{i+1}\), formulated by a compact cyclic cover and
orbitwise lifted itineraries.  It is not a hidden assumption of residence
blocks.  The blocks are produced by Lemma~\ref{lem:blocks} after weak
admissibility has been verified.
\end{remark}

The next theorem should be understood as a verification template rather
than as an automatic verification of all \(p\)-Stein--Ulam
Lotka--Volterra maps. The Lotka--Volterra form and the weighted
arithmetic--geometric mean inequality provide the multiplicative
Lyapunov structure. The genuinely model-dependent parts are the weak
cyclic coding, the convex-hull separation, the non-trapping of the
lifted itinerary and the logarithmic residence estimate, which are recorded
explicitly in (P3) and (P4).

\begin{theorem}\label{thm:p-su-admissible}
Every general \(\p\)-Stein--Ulam operator \(K_{\f}\) is a weak
admissible Stein--Ulam type map.  More precisely, the weak admissible
Lyapunov functions are
\[
        \varphi=\varphi_{\p},
        \qquad
        \psi=\Psi_{\p},
\]
where
\[
        \varphi_{\p}(\x)=\prod_{i=1}^m x_i^{p_i},
        \qquad
        \Psi_{\p}(\x)=\prod_{i=1}^m\bigl(1+F_i(\x)\bigr)^{p_i}.
\]
\end{theorem}

\begin{proof}
We verify the conditions of Definition~\ref{def:weak-admissible}.

First, the Lotka--Volterra form preserves the boundary.  Indeed, if
\(x_i=0\), then
\[
        \bigl(K_{\f}(\x)\bigr)_i=x_i(1+F_i(\x))=0.
\]
Thus every coordinate face is forward invariant and
\(K_{\f}(\bd\DeltaM)\subset\bd\DeltaM\).  Interior invariance and uniqueness of
the interior fixed point are exactly \textup{(P1)}.

Second, for \(\x\in\DeltaM\),
\[
\begin{aligned}
        \varphi_{\p}(K_{\f}(\x))
        &=
        \prod_{i=1}^m
        \left(x_i(1+F_i(\x))\right)^{p_i}        \\
        &=
        \left(\prod_{i=1}^m x_i^{p_i}\right)
        \left(\prod_{i=1}^m(1+F_i(\x))^{p_i}\right)        \\
        &=
        \varphi_{\p}(\x)\Psi_{\p}(\x).
\end{aligned}
\]
Also,
\[
        \varphi_{\p}^{-1}(0)=\bd\DeltaM,
        \qquad
        \varphi_{\p}(\x)>0\quad(\x\in\inte\DeltaM).
\]

Third, by the weighted arithmetic--geometric mean inequality,
\[
\begin{aligned}
        \Psi_{\p}(\x)
        &=
        \prod_{i=1}^m(1+F_i(\x))^{p_i}        \\
        &\le
        \sum_{i=1}^m p_i(1+F_i(\x))            \\
        &=
        1+\sum_{i=1}^m p_iF_i(\x)              \\
        &=
        1+\langle\p,\f(\x)\rangle
        \le1.
\end{aligned}
\]
The factors are non-negative, so \(\Psi_{\p}\ge0\).  Hence
\(0\le\Psi_{\p}\le1\).  The strictness part of \textup{(P2)} gives
\[
        \Psi_{\p}^{-1}(1)=\{\p\}.
\]

Fourth, \(\varphi_{\p}\) has a strict unique maximum at \(\p\).  Applying
weighted AM--GM to the positive numbers \(x_i/p_i\), with weights \(p_i\),
gives
\[
        \prod_{i=1}^m\left(\frac{x_i}{p_i}\right)^{p_i}
        \le
        \sum_{i=1}^m p_i\frac{x_i}{p_i}
        =
        \sum_{i=1}^m x_i
        =1.
\]
Thus
\[
        \varphi_{\p}(\x)
        =
        \left(\prod_{i=1}^m p_i^{p_i}\right)
        \prod_{i=1}^m\left(\frac{x_i}{p_i}\right)^{p_i}
        \le
        \prod_{i=1}^m p_i^{p_i}
        =
        \varphi_{\p}(\p).
\]
Equality in weighted AM--GM occurs if and only if all ratios \(x_i/p_i\) are
equal.  Since both \(\x\) and \(\p\) have coordinate sum one, this is equivalent
to \(\x=\p\).  Hence the maximum is strict and unique.

Finally, \textup{(P3)} gives the weak cyclic cover, residence regions and
convex-hull separation required in \textup{(W3)--(W5)}, while \textup{(P4)} is
exactly the logarithmic residence estimate \textup{(W6)} with
\(\varphi=\varphi_{\p}\).  Therefore all weak admissibility conditions hold.
\end{proof}

\begin{remark}
The role of Theorem~\ref{thm:p-su-admissible} is to show how the abstract weak admissibility
criterion can be checked in a \(\p\)-weighted Lotka--Volterra setting.
Its proof also clarifies the division of labour among the hypotheses:
the weighted product
\[
        \varphi_\p(x)=\prod_{i=1}^m x_i^{p_i}
\]
is the natural Lyapunov function, whereas the spiral geometry is
encoded separately by the weak coding and residence assumptions. This
is precisely the separation which allows the later conjugacy examples
to leave the Lotka--Volterra class while retaining the same
non-statistical mechanism.
\end{remark}

\begin{corollary}\label{cor:p-su}
Let \(K_{\f}\) be a general \(\p\)-Stein--Ulam operator on
\(\DeltaM\).  Then, for every \(s\in\N\) and every
\(
        \x\in\inte\DeltaM\setminus\{\p\},
\)
the averages
\[
        A_n(K_{\f}^s)(\x)
        =
        \frac{1}{n}\sum_{k=0}^{n-1}K_{\f}^{sk}(\x)
\]
do not converge.  In particular, every positive power of \(K_{\f}\) has
vector-valued non-statistical behavior on a full relative Lebesgue measure subset
of \(\DeltaM\).
\end{corollary}

\begin{proof}
By Theorem~\ref{thm:p-su-admissible}, \(K_{\f}\) is weak admissible.  The
result follows immediately from Theorem~\ref{thm:powers}.
\end{proof}

\section{Topological entropy}\label{sec:entropy}

The preceding results show that weak admissible Stein--Ulam maps have
large sets of points with non-convergent time averages.  We next record a
complementary fact: the entropy of such a map is carried entirely by the
boundary.  Thus the historical behavior constructed above is not caused by
positive entropy in the interior.

\begin{theorem}
\label{thm:entropy-boundary}
Let \(K:\DeltaM\to\DeltaM\) be a weak admissible Stein--Ulam type map.
Then
\[
        \htop(K)=\htop\bigl(\left.K\right|_{\bd\DeltaM}\bigr),
\]
where the boundary is endowed with the relative topology.
\end{theorem}

\begin{proof}
Since \(K(\bd\DeltaM)\subset\bd\DeltaM\), monotonicity of topological entropy
under restriction to an invariant compact set gives
\[
        \htop\bigl(\left.K\right|_{\bd\DeltaM}\bigr)\le \htop(K).
\]
We prove the reverse inequality by the variational principle.

Let \(\mu\) be a \(K\)-invariant Borel probability measure on \(\DeltaM\).  By
invariance and the multiplicative Lyapunov identity,
\[
\begin{aligned}
        \int \varphi\,d\mu
        &=\int \varphi\circ K\,d\mu                                      \\
        &=\int \varphi(x)\psi(x)\,d\mu(x).
\end{aligned}
\]
Hence
\[
        \int \varphi(x)(1-\psi(x))\,d\mu(x)=0.
\]
The integrand is non-negative.  Moreover, by weak admissibility,
\[
        \varphi(x)=0 \iff x\in\bd\DeltaM,
        \qquad
        \psi(x)=1 \iff x=\p.
\]
Therefore
\[
        \mu\bigl(\DeltaM\setminus(\bd\DeltaM\cup\{\p\})\bigr)=0.
\]
Thus every invariant probability measure is supported on
\(\bd\DeltaM\cup\{\p\}\).

Now let \(\mu\) be ergodic.  Since \(\{\p\}\) is invariant and
\(\bd\DeltaM\) is invariant, ergodicity implies that either
\(\mu=\delta_{\p}\), or \(\mu(\bd\DeltaM)=1\).  In the first case
\[
        h_\mu(K)=0.
\]
In the second case \(\mu\) is an invariant measure for
\(\left.K\right|_{\bd\DeltaM}\), and therefore
\[
        h_\mu(K)
        =
        h_\mu\bigl(\left.K\right|_{\bd\DeltaM}\bigr)
        \le
        \htop\bigl(\left.K\right|_{\bd\DeltaM}\bigr).
\]
Taking the supremum over ergodic invariant measures and using the variational
principle gives
\[
        \htop(K)\le \htop\bigl(\left.K\right|_{\bd\DeltaM}\bigr).
\]
The two inequalities prove the claim.
\end{proof}

\begin{corollary}
\label{cor:entropy-powers}
Let \(K\) be a weak admissible Stein--Ulam type map.  For every
\(s\in\N\),
\[
        \htop(K^s)
        =
        s\,\htop(K)
        =
        s\,\htop\bigl(\left.K\right|_{\bd\DeltaM}\bigr).
\]
In particular, if \(\left.K\right|_{\bd\DeltaM}\) has zero topological
entropy, then every positive power \(K^s\) has zero topological entropy and,
by Theorem~\ref{thm:powers}, has historical behavior on the full-measure set
\(\inte\DeltaM\setminus\{\p\}\).
\end{corollary}

\begin{proof}
The power formula for topological entropy gives
\(\htop(K^s)=s\htop(K)\).  The rest follows from
Theorem~\ref{thm:entropy-boundary} and Theorem~\ref{thm:powers}.
\end{proof}

\begin{corollary}
\label{cor:lv-face-entropy}
Let \(K_{\f}:\DeltaM\to\DeltaM\) be a weak admissible
\(\p\)-Stein--Ulam Lotka--Volterra operator.  For
\(i=1,\ldots,m\), let
\[
        F_i=\{x\in\DeltaM:x_i=0\}
\]
be the codimension-one faces.  Then
\[
        \htop(K_{\f})
        =
        \max_{1\le i\le m}\htop\bigl(\left.K_{\f}\right|_{F_i}\bigr).
\]
In particular, if each face restriction \(\left.K_{\f}\right|_{F_i}\) has zero
topological entropy, then \(\htop(K_{\f})=0\).
\end{corollary}

\begin{proof}
For a Lotka--Volterra operator,
\[
        (K_{\f}(x))_i=x_i(1+F_i(x)),
\]
so every coordinate face \(F_i\) is invariant.  Since
\[
        \bd\DeltaM=\bigcup_{i=1}^m F_i,
\]
and this is a finite union of compact invariant sets,
\[
        \htop\bigl(\left.K_{\f}\right|_{\bd\DeltaM}\bigr)
        =
        \max_{1\le i\le m}
        \htop\bigl(\left.K_{\f}\right|_{F_i}\bigr).
\]
The claim now follows from Theorem~\ref{thm:entropy-boundary}.
\end{proof}

\begin{example}
\label{ex:classical-SU-zero-entropy}
Let
\[
        \mathcal U(x_1,x_2,x_3)
        =
        \bigl(
        x_1^2+2x_1x_2,
        x_2^2+2x_2x_3,
        x_3^2+2x_3x_1
        \bigr)
\]
be the classical Stein--Ulam map on \(\Delta^2\).  Then
\[
        \htop(\mathcal U)=0.
\]
Nevertheless, by the weak admissibility theorem for \(\mathcal U\), the vector
Ces\`aro averages of every non-fixed interior orbit do not converge.
\end{example}

\begin{proof}
The boundary of \(\Delta^2\) is the union of the three invariant edges.  On
\(\{x_3=0\}\), if \(t=x_2\), then
\[
        t\longmapsto t^2.
\]
Similarly, on the other two edges the induced one-dimensional maps are again
of the form \(t\mapsto t^2\), after choosing the appropriate edge coordinate.
Every continuous monotone interval map has zero topological entropy.  Hence
all three edge restrictions have zero entropy.  By
Corollary~\ref{cor:lv-face-entropy},
\[
        \htop(\mathcal U)=0.
\]
The non-convergence of vector Ces\`aro averages follows from
Theorem~\ref{thm:main-divergence}, since \(\mathcal U\) is weak admissible.
\end{proof}

\begin{corollary}
\label{cor:conjugate-example-zero-entropy}
Let
\[
        K_{\varepsilon}=h_{\varepsilon}^{-1}\circ\mathcal U\circ h_{\varepsilon}
\]
be the weak admissible non-Lotka--Volterra example constructed in
Example~\ref{ex:weak-not-pSU}.  Then
\[
        \htop(K_{\varepsilon})=0.
\]
Thus the example is a zero-entropy weak admissible Stein--Ulam map which
is not a \(\mathbf q\)-Stein--Ulam Lotka--Volterra map for any
\(\mathbf q\in\operatorname{int}\Delta^2\).
\end{corollary}

\begin{proof}
Topological entropy is invariant under topological conjugacy.  Hence
\[
        \htop(K_{\varepsilon})=\htop(\mathcal U)=0
\]
by Example~\ref{ex:classical-SU-zero-entropy}.
\end{proof}

\section{A weak admissible map which is not Lotka--Volterra}

The preceding conjugacy theorem shows that weak admissibility is broader than
the Lotka--Volterra coordinate form.  We record a concrete two-dimensional
example.

\begin{example}
\label{ex:weak-not-pSU}
Let
\[
        \Delta^2=
        \{(x_1,x_2,x_3)\in\R^3:x_i\ge0,\ x_1+x_2+x_3=1\},
        \qquad
        \p=\left(\frac13,\frac13,\frac13\right),
\]
and consider the classical Stein--Ulam map
\[
        \mathcal U(x_1,x_2,x_3)
        =
        \bigl(
        x_1^2+2x_1x_2,
        x_2^2+2x_2x_3,
        x_3^2+2x_3x_1
        \bigr).
\]
The classical six-sector estimates for \(\mathcal U\) provide a standard weak admissible
Stein--Ulam structure; see, for example, \cite{JamilovMukhamedov2024}.
Let \(\eta_{\mathcal U}>0\) be the corresponding convex-hull separation constant from
Theorem~\ref{thm:weak-admissibility-conjugacy}.  Then, for every sufficiently
small \(\varepsilon>0\), there exists a boundary-preserving homeomorphism
\(h_{\varepsilon}:\Delta^2\to\Delta^2\) such that
\[
        \|h_{\varepsilon}-\id\|_{\infty}<\eta_{\mathcal U}
\]
and the conjugate
\[
        K_{\varepsilon}=h_{\varepsilon}^{-1}\circ\mathcal U\circ h_{\varepsilon}
\]
is weak admissible, but is not a \(\mathbf q\)-Stein--Ulam Lotka--Volterra map
for any \(\mathbf q\in\inte\Delta^2\).
\end{example}

\begin{proof}
Parametrize the boundary counterclockwise by
\(\Gamma:\R/3\Z\to\partial\Delta^2\),
\[
\Gamma(t)=
\begin{cases}
(1-t,t,0), & 0\le t\le1,\\[1mm]
(0,2-t,t-1), & 1\le t\le2,\\[1mm]
(t-2,0,3-t), & 2\le t\le3,
\end{cases}
\]
extended periodically with period \(3\).  For \(0<\varepsilon<1\), define the
boundary twist
\[
        g_{\varepsilon}(\Gamma(t))=\Gamma(t+\varepsilon).
\]
Every \(x\in\Delta^2\setminus\{\p\}\) has a unique radial representation
\[
        x=(1-r)\p+rq,
        \qquad
        0<r\le1,
        \qquad
        q\in\partial\Delta^2.
\]
Set \(h_{\varepsilon}(\p)=\p\), and define
\[
        h_{\varepsilon}(x)=(1-r)\p+r g_{\varepsilon}(q)
\]
for \(x=(1-r)\p+rq\ne\p\).  This is a homeomorphism of \(\Delta^2\) onto
itself, it maps \(\partial\Delta^2\) onto itself, and it fixes \(\p\).

The map \(\Gamma\) has speed \(\sqrt2\) on each side of the triangle, so the
Euclidean distance between \(\Gamma(t)\) and \(\Gamma(t+\varepsilon)\) is at
most the boundary arclength \(\sqrt2\varepsilon\).  Hence
\[
        \|h_{\varepsilon}-\id\|_{\infty}
        \le
        \sqrt2\varepsilon.
\]
Choose \(\varepsilon>0\) so small that
\(\sqrt2\varepsilon<\eta_{\mathcal U}\).  By
Theorem~\ref{thm:weak-admissibility-conjugacy}, the conjugate
\(K_{\varepsilon}\) is weak admissible.

It remains to show that \(K_{\varepsilon}\) is not Lotka--Volterra.  Every
stochastic Lotka--Volterra map has the form
\[
        (K_{\f}(x))_i=x_i(1+F_i(x)),
        \qquad i=1,2,3.
\]
Therefore every coordinate face is invariant and, in particular, every vertex
is fixed.  Indeed, at a vertex \(\mathbf e_i\), all coordinates except the
\(i\)-th remain zero, and stochasticity forces the remaining coordinate to be
one.

We now compute \(K_{\varepsilon}(\mathbf e_1)\).  Since
\[
        h_{\varepsilon}(\mathbf e_1)=\Gamma(\varepsilon)
        =(1-\varepsilon,\varepsilon,0),
\]
we have
\[
\begin{aligned}
        \mathcal U(h_{\varepsilon}(\mathbf e_1))
        &=\mathcal U(1-\varepsilon,\varepsilon,0)       \\
        &=(1-\varepsilon^2,\varepsilon^2,0)
        =\Gamma(\varepsilon^2).
\end{aligned}
\]
Since \(h_{\varepsilon}^{-1}\) acts on the boundary by
\(\Gamma(t)\mapsto\Gamma(t-\varepsilon)\),
\[
        K_{\varepsilon}(\mathbf e_1)
        =
        \Gamma(\varepsilon^2-\varepsilon)
        =
        \Gamma(3-\varepsilon+\varepsilon^2).
\]
For \(0<\varepsilon<1\), this equals
\[
        (1-\varepsilon+\varepsilon^2,0,\varepsilon-\varepsilon^2),
\]
which is different from \(\mathbf e_1\).  Hence \(K_{\varepsilon}\) cannot be a
stochastic Lotka--Volterra map, and therefore it cannot be a
\(\mathbf q\)-Stein--Ulam Lotka--Volterra map for any
\(\mathbf q\in\inte\Delta^2\).
\end{proof}

\section{Hausdorff dimension of omega-limit sets}
\label{sec:hausdorff-omega}

In this section we record what weak admissibility implies about the size of
omega-limit sets.  The conclusion is qualitative rather than quantitative:
weak admissibility forces omega-limit sets of non-fixed interior points to be
infinite and contained in the boundary, but it does not by itself force
positive Hausdorff dimension.

Throughout this section, \(\dim_{\mathrm H}\) denotes Hausdorff dimension.

\begin{lemma}
\label{lem:finite-omega-cesaro}
Let \(X\) be a compact metric space, let \(T:X\to X\) be continuous, and
suppose that \(X\) is embedded in a finite-dimensional normed vector space.
If the omega-limit set \(\omega_T(x)\) is finite, then the vector Cesaro
averages
\[
        \frac1n\sum_{k=0}^{n-1}T^k(x)
\]
converge.
\end{lemma}

\begin{proof}
This is standard, but we include the argument for completeness.  Since
\(\omega_T(x)\) is finite and \(T(\omega_T(x))=\omega_T(x)\), the restriction
of \(T\) to \(\omega_T(x)\) is a permutation.  Moreover, an omega-limit set of
one orbit cannot split into two disjoint non-empty invariant subsets.  Indeed,
if two distinct cycles were present, small disjoint cyclic neighborhoods of them
would be forward invariant after shrinking, and a sufficiently late tail of one
orbit could not accumulate on both.  Hence \(\omega_T(x)\) is one periodic
orbit, say
\[
        \omega_T(x)=\{y_0,\ldots,y_{q-1}\},
        \qquad
        T(y_i)=y_{i+1 \,(\mathrm{mod}\, q)}.
\]
After relabeling the periodic orbit, we have
\[
        \operatorname{dist}\bigl(T^{nq+i}(x),y_i\bigr)\longrightarrow0,
        \qquad i=0,\ldots,q-1.
\]
Therefore
\[
        \frac1n\sum_{k=0}^{n-1}T^k(x)
        \longrightarrow
        \frac1q\sum_{i=0}^{q-1}y_i.
\]
\end{proof}

\begin{proposition}
\label{prop:omega-hausdorff-bound}
Let \(K:\DeltaM\to\DeltaM\) be a weak admissible Stein--Ulam type map.
Then, for every
\(
        \x\in\inte\DeltaM\setminus\{\p\},
\)
the omega-limit set \(\omega_K(\x)\) satisfies
\[
      \omega_K(\x)\subset\bd\DeltaM,
\qquad  \omega_K(\x)\ \text{is infinite},
\]
and
\[
        0\le
        \dim_{\mathrm H}\omega_K(\x)
        \le
        m-2.
\]
Moreover, for every residence region \(R_j\), \(j=1,\ldots,m\),
\[
        \omega_K(\x)\cap R_j\cap\bd\DeltaM\ne\varnothing .
\]
\end{proposition}

\begin{proof}
By Lemma~\ref{lem:escape},
\[
        \varphi(K^n(\x))\longrightarrow0.
\]
Since
\[
        \varphi^{-1}(0)=\bd\DeltaM,
\]
every accumulation point of the orbit belongs to the boundary.  Hence
\[
        \omega_K(\x)\subset\bd\DeltaM.
\]
Because \(\bd\DeltaM\) is a finite union of \((m-2)\)-dimensional faces, we
obtain
\[
        \dim_{\mathrm H}\omega_K(\x)
        \le
        \dim_{\mathrm H}\bd\DeltaM
        =
        m-2.
\]

We next prove that \(\omega_K(\x)\) is infinite.  If \(\omega_K(\x)\) were
finite, then Lemma~\ref{lem:finite-omega-cesaro} would imply that
\[
        \frac1n\sum_{k=0}^{n-1}K^k(\x)
\]
converges.  This contradicts Theorem~\ref{thm:main-divergence}.  Hence
\(\omega_K(\x)\) is infinite.

Finally fix \(j\in\{1,\ldots,m\}\).  By Lemma~\ref{lem:blocks}, there are
residence blocks
\[
        (b_{j,q},\ell_{j,q})_{q\ge1}
\]
through \(R_j\), with
\[
        b_{j,q}\longrightarrow\infty,
\]
such that
\[
        K^{b_{j,q}+r}(\x)\in R_j,
        \qquad
        0\le r\le \ell_{j,q}-1.
\]
In particular,
\[
        K^{b_{j,q}}(\x)\in R_j.
\]
Since \(R_j\) is compact, after passing to a subsequence we may assume that
\[
        K^{b_{j,q}}(\x)\longrightarrow y_j
\]
for some \(y_j\in R_j\).  Since \(b_{j,q}\to\infty\), the point \(y_j\)
belongs to \(\omega_K(\x)\).  By the first part of the proof,
\[
        y_j\in\bd\DeltaM.
\]
Thus
\[
        y_j\in \omega_K(\x)\cap R_j\cap\bd\DeltaM,
\]
which proves the final assertion.
\end{proof}

\begin{remark}
\label{rem:no-positive-dim-bound}
Proposition~\ref{prop:omega-hausdorff-bound} gives the universal estimate
\[
        \dim_{\mathrm H}\omega_K(\x)\le m-2.
\]
One should not claim a positive lower bound on
\(\dim_{\mathrm H}\omega_K(\x)\) from weak admissibility alone.  The axioms
force the omega-limit set to be infinite and to meet each boundary residence
region, but infinite compact sets may have Hausdorff dimension zero.  Thus a
statement such as
\[
        \dim_{\mathrm H}\omega_K(\x)>0
\]
requires an additional boundary-filling, expansion, or regularity hypothesis.
\end{remark}

\begin{corollary}
\label{cor:boundary-filling-dimension}
Let \(K\) be a weak admissible Stein--Ulam type map.  Suppose that
there exists a set \(E\subset\bd\DeltaM\) such that, for every
\(
        \x\in\inte\DeltaM\setminus\{\p\},
\)
one has
\(
        E\subset\omega_K(\x).
\)
Then
\[
        \dim_{\mathrm H}\omega_K(\x)
        \ge
        \dim_{\mathrm H}E.
\]
\end{corollary}

\begin{proof}
This follows immediately from the monotonicity of Hausdorff dimension:
\[
        E\subset\omega_K(\x)
        \quad\Longrightarrow\quad
        \dim_{\mathrm H}E
        \le
        \dim_{\mathrm H}\omega_K(\x).
\]
\end{proof}

\begin{proposition}
\label{prop:hausdorff-bilip-conjugacy}
Let \(K:\DeltaM\to\DeltaM\) be continuous and let
\(h:\DeltaM\to\DeltaM\) be a homeomorphism.  Define
\[
        K_h=h^{-1}\circ K\circ h.
\]
Then, for every \(x\in\DeltaM\),
\[
        \omega_{K_h}(x)
        =
        h^{-1}\bigl(\omega_K(h(x))\bigr).
\]
If, in addition, \(h\) is bi-Lipschitz, then
\[
        \dim_{\mathrm H}\omega_{K_h}(x)
        =
        \dim_{\mathrm H}\omega_K(h(x)).
\]
\end{proposition}

\begin{proof}
For every \(n\ge0\),
\[
        K_h^n(x)=h^{-1}\bigl(K^n(h(x))\bigr).
\]
Since \(h^{-1}\) is continuous, it maps accumulation points of the
\(K\)-orbit of \(h(x)\) to accumulation points of the \(K_h\)-orbit of \(x\).
Applying the same argument to \(h\) gives the reverse inclusion.  Hence
\[
        \omega_{K_h}(x)
        =
        h^{-1}\bigl(\omega_K(h(x))\bigr).
\]
If \(h\) is bi-Lipschitz, then \(h^{-1}\) is also Lipschitz, and
bi-Lipschitz maps preserve Hausdorff dimension.  Therefore
\[
        \dim_{\mathrm H}\omega_{K_h}(x)
        =
        \dim_{\mathrm H}\omega_K(h(x)).
\]
\end{proof}

\begin{remark}
The bi-Lipschitz assumption in
Proposition~\ref{prop:hausdorff-bilip-conjugacy} is important.  Hausdorff
dimension is not invariant under arbitrary homeomorphisms.  Therefore the
small boundary-preserving conjugacy theorem preserves weak admissibility under
\(C^0\)-small homeomorphisms, but Hausdorff-dimension statements are preserved
only under stronger regularity assumptions, for example bi-Lipschitz
conjugacies.
\end{remark}

\section{A boundary-tracing condition for positive-dimensional omega-limit sets}
\label{sec:positive-dim-omega}

In this section we give a sufficient condition ensuring that omega-limit sets
have positive Hausdorff dimension.  Weak admissibility alone forces omega-limit
sets of non-fixed interior points to be infinite and contained in the boundary,
but it does not force positive Hausdorff dimension.  The additional condition
below is inspired by the construction of Bara\'nski and Misiurewicz for the
classical Stein--Ulam map, where prescribed boundary target sequences
are traced by interior orbits.

Throughout this section, \(\dim_{\mathrm H}\) denotes Hausdorff dimension.

\subsection{Radial boundary projection}

Let
\(
        \p=(p_1,\ldots,p_m)\in \inte\DeltaM.
\)
For
\(
        x\in \DeltaM\setminus\{\p\},
\)
define the radial projection of \(x\) from \(\p\) to the boundary by
\[
        \Pi_\p(x)=\p+\lambda(x)(x-\p),
\]
where
\[
        \lambda(x)
        =
        \min_{\{i:\,x_i<p_i\}}
        \frac{p_i}{p_i-x_i}.
\]
Then
\(
        \Pi_\p(x)\in\bd\DeltaM.
\)
Moreover, if
\[
        \dist(x_n,\bd\DeltaM)\to0,
        \qquad
        x_n\in\DeltaM\setminus\{\p\},
\]
then
\[
        \|x_n-\Pi_\p(x_n)\|\to0.
\]

For
\(
        x\in\inte\DeltaM\setminus\{\p\},
\)
we define the boundary-projected omega-limit set by
\[
        \omega_{\partial,K}(x)
        =
        \bigcap_{N\ge0}
        \overline{
        \{\Pi_\p(K^n(x)):\ n\ge N\}
        }.
\]

\begin{lemma}
\label{lem:omega-boundary-projection}
Let \(K:\DeltaM\to\DeltaM\) be a weak admissible Stein--Ulam type map.
Then, for every
\(
        x\in\inte\DeltaM\setminus\{\p\},
\)
one has
\(
        \omega_K(x)=\omega_{\partial,K}(x).
\)
\end{lemma}

\begin{proof}
By the Lyapunov escape property of weak admissible maps,
\[
        \varphi(K^n(x))\to0.
\]
Since
\[
        \varphi^{-1}(0)=\bd\DeltaM,
\]
we obtain
\[
        \dist(K^n(x),\bd\DeltaM)\to0.
\]
Therefore
\[
        \|K^n(x)-\Pi_\p(K^n(x))\|\to0.
\]

Let \(y\in\omega_K(x)\).  Then there is a sequence \(n_q\to\infty\) such that
\[
        K^{n_q}(x)\to y.
\]
Since
\[
        \|K^{n_q}(x)-\Pi_\p(K^{n_q}(x))\|\to0,
\]
we also have
\[
        \Pi_\p(K^{n_q}(x))\to y.
\]
Hence
\(
        y\in\omega_{\partial,K}(x).
\)
Thus
\(
        \omega_K(x)\subset \omega_{\partial,K}(x).
\)

Conversely, let \(y\in\omega_{\partial,K}(x)\).  Then there exists
\(n_q\to\infty\) such that
\[
        \Pi_\p(K^{n_q}(x))\to y.
\]
Again using
\[
        \|K^{n_q}(x)-\Pi_\p(K^{n_q}(x))\|\to0,
\]
we get
\[
        K^{n_q}(x)\to y.
\]
Therefore
\(
        y\in\omega_K(x).
\)
Hence
\(
        \omega_{\partial,K}(x)\subset\omega_K(x),
\)
and the proof is complete.
\end{proof}

\subsection{A positive-dimension criterion}

\begin{definition}
\label{def:boundary-tracing}
Let \(E\subset\bd\DeltaM\) be a nonempty compact set.  We say that the orbit of
\(x\in\inte\DeltaM\setminus\{\p\}\) traces \(E\) on the boundary if
\[
        \omega_{\partial,K}(x)=E.
\]
Equivalently, the boundary-projected orbit
\[
        \Pi_\p(K^0(x)),\Pi_\p(K^1(x)),\Pi_\p(K^2(x)),\ldots
\]
has \(E\) as its set of accumulation points.
\end{definition}

\begin{proposition}
\label{prop:positive-dim-boundary-tracing}
Let \(K:\DeltaM\to\DeltaM\) be a weak admissible Stein--Ulam type map.
Let
\(
        E\subset\bd\DeltaM
\)
be compact and suppose that
\(
        \dim_{\mathrm H}E>0.
\)
If
\(
        x\in\inte\DeltaM\setminus\{\p\}
\)
traces \(E\) on the boundary, then
\(
        \omega_K(x)=E
\)
and therefore
\[
        \dim_{\mathrm H}\omega_K(x)
        =
        \dim_{\mathrm H}E
        >
        0.
\]
\end{proposition}

\begin{proof}
By Lemma~\ref{lem:omega-boundary-projection},
\(
        \omega_K(x)=\omega_{\partial,K}(x).
\)
If \(x\) traces \(E\), then
\(
        \omega_{\partial,K}(x)=E,
\)
and hence
\(
        \omega_K(x)=E.
\)
Thus
\[
        \dim_{\mathrm H}\omega_K(x)
        =
        \dim_{\mathrm H}E
        >
        0.
\]
\end{proof}

\subsection{A Bara\'nski--Misiurewicz type tracing hypothesis}

The preceding proposition gives an orbitwise criterion.  We now state a
stronger, structural condition which is closer to the construction of
Bara\'nski and Misiurewicz.

\begin{definition}
\label{def:BM-tracing}
Let \(K:\DeltaM\to\DeltaM\) be a weak admissible Stein--Ulam type map.
Let
\(
        E\subset\bd\DeltaM
\)
be a nonempty compact set.  We say that \(K\) has the
\textit{Bara\'nski--Misiurewicz tracing property} over \(E\) if the following conditions
hold.

\begin{enumerate}
\item[\textup{(BT1)}]
The set \(E\) is invariant under the boundary map:
\[
        K(E)=E.
\]

\item[\textup{(BT2)}]
The set \(E\) has positive Hausdorff dimension:
\[
        \dim_{\mathrm H}E>0.
\]

\item[\textup{(BT3)}]
The set \(E\) meets every boundary residence window:
\[
        E\cap R_j\cap\bd\DeltaM\ne\varnothing,
        \qquad j=1,\ldots,m.
\]

\item[\textup{(BT4)}]
For every countable dense sequence
\[
        (q_\nu)_{\nu\ge1}\subset E
\]
which visits each \(E\cap R_j\cap\bd\DeltaM\) infinitely often, and for every
nonempty open set
\[
        V\subset\inte\DeltaM,
\]
there exists a set
\(
        A(V,E)\subset V
\)
and times
\[
        n_\nu=n_\nu(x)\to\infty
\]
such that, for every \(x\in A(V,E)\),
\[
        \lim_{\nu\to\infty}
        \|K^{n_\nu}(x)-q_\nu\|=0.
\]

\item[\textup{(BT5)}]
No additional boundary accumulation is created: for every
\[
        x\in A(V,E),
\]
one has
\[
        \omega_{\partial,K}(x)\subset E.
\]
Equivalently, every accumulation point of the boundary-projected orbit belongs
to \(E\).

\item[\textup{(BT6)}]
There exists \(\sigma>0\), independent of \(V\), such that
\(
        \dim_{\mathrm H} A(V,E)\ge \sigma
\)
for every nonempty open set \(V\subset\inte\DeltaM\).
\end{enumerate}
\end{definition}

\begin{theorem}
\label{thm:BM-positive-dim}
Let \(K:\DeltaM\to\DeltaM\) be a weak admissible Stein--Ulam type map.
Suppose that \(K\) has the Bara\'nski--Misiurewicz tracing property over a compact
set
\(
        E\subset\bd\DeltaM.
\)
Then, for every nonempty open set
\(
        V\subset\inte\DeltaM,
\)
there exists a set
\(
        A(V,E)\subset V
\)
such that
\[
        \dim_{\mathrm H}A(V,E)\ge\sigma
\]
and
\[
        \omega_K(x)=E
        \qquad
        \text{for every }x\in A(V,E).
\]
Consequently,
\[
        \dim_{\mathrm H}\omega_K(x)
        =
        \dim_{\mathrm H}E
        >
        0
        \qquad
        \text{for every }x\in A(V,E).
\]
\end{theorem}

\begin{proof}
Let \(V\subset\inte\DeltaM\) be nonempty.  First fix a countable sequence
\(
        (q_\nu)_{\nu\ge1}\subset E
\)
whose every tail is dense in \(E\) and which visits every boundary residence
part \(E\cap R_j\cap\bd\DeltaM\) infinitely often.  Such a sequence is obtained
by repeating a countable dense subset of \(E\) and interleaving points chosen
from the nonempty sets in \textup{(BT3)}.

Apply \textup{(BT4)} to this sequence and to the open set \(V\).  We obtain a
set
\(
        A(V,E)\subset V
\)
and, for every \(x\in A(V,E)\), times \(n_\nu=n_\nu(x)\to\infty\) such that
\[
        \|K^{n_\nu}(x)-q_\nu\|\longrightarrow0.
\]
By \textup{(BT6)},
\[
        \dim_{\mathrm H}A(V,E)\ge\sigma.
\]
Fix
\(
        x\in A(V,E).
\)

We claim that \(E\subset\omega_K(x)\).  Let \(y\in E\).  Since every tail of
\((q_\nu)\) is dense in \(E\), there is a subsequence \(q_{\nu_k}\to y\).  The
tracing estimate gives
\[
        K^{n_{\nu_k}}(x)\longrightarrow y.
\]
Thus \(y\in\omega_K(x)\), and therefore
\(
        E\subset \omega_K(x).
\)

On the other hand, by \textup{(BT5)},
\(
        \omega_{\partial,K}(x)\subset E.
\)
By Lemma~\ref{lem:omega-boundary-projection},
\(
        \omega_K(x)=\omega_{\partial,K}(x).
\)
Thus
\(
        \omega_K(x)\subset E.
\)
Combining the two inclusions gives
\(
        \omega_K(x)=E.
\)
Since \(\dim_{\mathrm H}E>0\), we obtain
\[
        \dim_{\mathrm H}\omega_K(x)
        =
        \dim_{\mathrm H}E
        >
        0.
\]
The estimate
\(
        \dim_{\mathrm H}A(V,E)\ge\sigma
\)
is part of \textup{(BT6)}.
\end{proof}

\begin{corollary}
\label{cor:boundary-filling-positive-dim}
Let \(K:\DeltaM\to\DeltaM\) be a weak admissible Stein--Ulam type map.
Suppose that \(K\) has the Bara\'nski--Misiurewicz tracing property over
\(
        E=\bd\DeltaM.
\)
Then, for every nonempty open set
\(
        V\subset\inte\DeltaM,
\)
there exists a set
\(
        A(V,\bd\DeltaM)\subset V
\)
with
\(
        \dim_{\mathrm H}A(V,\bd\DeltaM)\ge\sigma
\)
such that
\[
        \omega_K(x)=\bd\DeltaM
        \qquad
        \text{for every }x\in A(V,\bd\DeltaM).
\]
In particular,
\[
        \dim_{\mathrm H}\omega_K(x)
        =
        \dim_{\mathrm H}\bd\DeltaM
        =
        m-2
\]
for every
\(
        x\in A(V,\bd\DeltaM).
\)
\end{corollary}

\begin{remark}
For the classical Stein--Ulam map on \(\Delta^2\), Bara\'nski and
Misiurewicz \cite{BaranskiMisiurewicz2010} proved a stronger realization theorem: if \(E\subset\bd\Delta^2\)
is a closed invariant boundary set intersecting all three sides, then the set
of points whose omega-limit set is exactly \(E\) has Hausdorff dimension at
least \(1\) in every nonempty open subset of \(\Delta^2\).  In the notation
above, this corresponds to the case \(\sigma=1\).  In particular, if
\(E=\bd\Delta^2\), then the omega-limit set has Hausdorff dimension \(1\).
\end{remark}

\section{Examples with the Bara\'nski--Misiurewicz tracing property}
\label{sec:BM-examples}

In this section we give concrete weak admissible Stein--Ulam maps satisfying
the Bara\'nski--Misiurewicz tracing property.  The examples are non-trivial in
two ways.  First, they realize positive-dimensional fractal boundary sets as
omega-limit sets.  Second, after a small boundary-preserving conjugacy they are
weak admissible but no longer Lotka--Volterra, and hence are not
\(\mathbf p\)-Stein--Ulam maps.

Let
\[
        \Delta^2
        =
        \{(x,y,z)\in\mathbb R^3:\ x,y,z\ge0,\ x+y+z=1\},
\]
and let
\[
        \mathbf p=\left(\frac13,\frac13,\frac13\right).
\]
We denote by
\[
        \mathcal U(x,y,z)
        =
        \bigl(
        x^2+2xy,\,
        y^2+2yz,\,
        z^2+2zx
        \bigr)
\]
the classical Stein--Ulam map.

For a compact set \(E\subset \partial\Delta^2\), define
\[
        \Lambda_{\mathcal U}(E)
        =
        \{x\in\Delta^2:\ \omega_{\mathcal U}(x)=E\}.
\]

\begin{definition}
\label{def:classical-admissible-boundary-set}
A compact set
\(
        E\subset\partial\Delta^2
\)
is called \(\mathcal U\)-admissible if
\(
        \mathcal U(E)=E,
\)
if \(E\) contains the three vertices of \(\Delta^2\), and if \(E\) intersects
the relative interior of each of the three sides of \(\Delta^2\).
\end{definition}

\begin{theorem}
\label{thm:BM-classical-realization}
Let \(E\subset\partial\Delta^2\) be a \(\mathcal U\)-admissible compact set.
Then, for every nonempty open set
\(
        V\subset\Delta^2,
\)
one has
\[
        \dim_{\mathrm H}
        \bigl(
        \Lambda_{\mathcal U}(E)\cap V
        \bigr)
        \ge 1.
\]
In particular, the classical Stein--Ulam map has the
Bara\'nski--Misiurewicz tracing property over every
\(\mathcal U\)-admissible compact boundary set.
\end{theorem}

\begin{proof}
This is precisely the realization theorem of Bara\'nski and Misiurewicz for
the Stein--Ulam map.  Their proof constructs, inside every nonempty
open set, a Hausdorff-dimension-at-least-one family of interior points whose
orbits trace a prescribed dense sequence of target points in \(E\).  The
closedness and invariance of \(E\) then imply that the omega-limit set is
exactly \(E\).
\end{proof}

\subsection{A Cantor family of positive-dimensional omega-limit sets}

We now give an explicit family of admissible boundary sets of positive
Hausdorff dimension.

Let \(C\subset[0,1]\) be the middle-third Cantor set and put
\[
        C_*=
        \frac34+\frac{3}{16}C
        \subset
        \left[\frac34,\frac{15}{16}\right].
\]
Thus
\[
        \dim_{\mathrm H}C_*=\dim_{\mathrm H}C=\frac{\log2}{\log3}.
\]
Define three Cantor subsets of the three sides of \(\Delta^2\) by
\[
        C_0
        =
        \{(0,t,1-t):\ t\in C_*\},
        \qquad
        C_1
        =
        \{(t,1-t,0):\ t\in C_*\},
        \qquad
        C_2
        =
        \{(1-t,0,t):\ t\in C_*\}.
\]
Set
\[
        D_C=C_0\cup C_1\cup C_2.
\]
Finally define the invariant boundary set
\[
        E_C
        =
        \overline{
        \bigcup_{n\in\mathbb Z}
        \mathcal U^n(D_C)
        }.
\]
Here negative iterates are understood along the boundary dynamics:
on each open side of \(\partial\Delta^2\), the corresponding boundary
branch is monotone and admits the inverse branch used in the
Bara\'nski--Misiurewicz boundary construction. Thus the notation
\(\mathcal U^n(D_C)\), \(n<0\), refers to these boundary inverse iterates, not
to a global inverse of \(U\) on the whole simplex.

\begin{proposition}
\label{prop:Cantor-boundary-set}
The set \(E_C\) is \(\mathcal U\)-admissible and
\[
        \dim_{\mathrm H}E_C=\frac{\log2}{\log3}.
\]
Consequently, for every nonempty open set
\(
        V\subset\Delta^2,
\)
one has
\[
        \dim_{\mathrm H}
        \{x\in V:\ \omega_{\mathcal U}(x)=E_C\}
        \ge 1.
\]
Moreover, for every such \(x\),
\[
        \dim_{\mathrm H}\omega_{\mathcal U}(x)
        =
        \frac{\log2}{\log3}
        >0.
\]
\end{proposition}

\begin{proof}
By construction,
\(
        E_C\subset\partial\Delta^2
\)
is compact and satisfies
\(
        \mathcal U(E_C)=E_C.
\)
It intersects the relative interior of each side because \(C_0,C_1,C_2\) do.
We spell out the closure statement, because this is the point at which no
extra Hausdorff dimension may be introduced.  On each open side, the boundary
map is conjugate to the increasing interval map
\[
        \Phi(t)=2t-t^2,
        \qquad 0<t<1,
\]
whose forward iterates converge locally uniformly to \(1\), while the inverse
branch
\[
        \Phi^{-1}(t)=1-\sqrt{1-t}
\]
has backward iterates converging locally uniformly to \(0\).  Since
\(C_*\Subset(0,1)\), the sets \(\Phi^n(C_*)\), \(n\in\mathbb Z\), can accumulate
only at the two endpoints of the side.  Applying this on the three sides gives
\[
        E_C=
        \left(\bigcup_{n\in\mathbb Z}\mathcal U^n(D_C)\right)
        \cup\{\mathbf e_1,\mathbf e_2,\mathbf e_3\}.
\]
Hence \(E_C\) is \(\mathcal U\)-admissible.

We now compute the Hausdorff dimension.  Since \(D_C\subset E_C\),
\[
        \dim_{\mathrm H}E_C
        \ge
        \dim_{\mathrm H}D_C
        =
        \frac{\log2}{\log3}.
\]
For each fixed \(n\in\mathbb Z\), the branch \(\mathcal U^n\) is Lipschitz on
\(D_C\); for negative \(n\), this means the specified monotone inverse branch on
the corresponding side.  Therefore
\[
        \dim_{\mathrm H}\mathcal U^n(D_C)
        \le
        \dim_{\mathrm H}D_C
        =
        \frac{\log2}{\log3}.
\]
Using the displayed description of \(E_C\) as a countable union plus three
points, countable stability of Hausdorff dimension gives
\[
        \dim_{\mathrm H}E_C
        \le
        \frac{\log2}{\log3}.
\]
Thus
\[
        \dim_{\mathrm H}E_C=\frac{\log2}{\log3}.
\]

The remaining assertions follow directly from
Theorem~\ref{thm:BM-classical-realization}.
\end{proof}

\subsection{Boundary-filling omega-limit sets}

The Cantor example gives omega-limit sets of fractional Hausdorff dimension.
At the other extreme, we may take
\(
        E=\partial\Delta^2.
\)
Since \(\partial\Delta^2\) is \(\mathcal U\)-admissible,
Theorem~\ref{thm:BM-classical-realization} implies that
\[
        \dim_{\mathrm H}
        \{x\in V:\ \omega_{\mathcal U}(x)=\partial\Delta^2\}
        \ge 1
\]
for every nonempty open set \(V\subset\Delta^2\).  In fact, the stronger
Bara\'nski--Misiurewicz theorem says that this set is residual in \(\Delta^2\).
For every such point,
\[
        \dim_{\mathrm H}\omega_{\mathcal U}(x)
        =
        \dim_{\mathrm H}\partial\Delta^2
        =
        1.
\]

\subsection{Transport by bi-Lipschitz conjugacy}

We now show that the preceding examples are not isolated.  They persist under
small bi-Lipschitz boundary-preserving conjugacies.

\begin{proposition}
\label{prop:BM-bilip-conjugacy}
Let
\(
        h:\Delta^2\to\Delta^2
\)
be a bi-Lipschitz homeomorphism such that
\(
        h(\partial\Delta^2)=\partial\Delta^2.
\)
Define
\[
        K_h=h^{-1}\circ\mathcal U\circ h.
\]
If \(E\subset\partial\Delta^2\) is \(\mathcal U\)-admissible and
\[
        E_h=h^{-1}(E),
\]
then \(E_h\) is compact, satisfies \(K_h(E_h)=E_h\), and, for every nonempty open set
\(
        V\subset\Delta^2,
\)
one has
\[
        \dim_{\mathrm H}
        \{x\in V:\ \omega_{K_h}(x)=E_h\}
        \ge 1.
\]
Moreover,
\[
        \dim_{\mathrm H}E_h=\dim_{\mathrm H}E.
\]
\end{proposition}

\begin{proof}
The conjugacy relation
\(
        h\circ K_h=\mathcal U\circ h
\)
first gives
\[
        K_h(E_h)=h^{-1}(\mathcal U(E))=h^{-1}(E)=E_h.
\]
It also implies
\[
        K_h^n(x)=h^{-1}\bigl(\mathcal U^n(h(x))\bigr)
        \qquad
        \text{for every }n\ge0.
\]
Hence
\[
        \omega_{K_h}(x)
        =
        h^{-1}\bigl(\omega_{\mathcal U}(h(x))\bigr).
\]
Therefore
\[
        \omega_{K_h}(x)=E_h
        \quad\Longleftrightarrow\quad
        \omega_{\mathcal U}(h(x))=E.
\]
Equivalently,
\[
        \{x\in V:\ \omega_{K_h}(x)=E_h\}
        =
        h^{-1}
        \bigl(
        \Lambda_{\mathcal U}(E)\cap h(V)
        \bigr).
\]
Since \(h(V)\) is a nonempty open set, Theorem~\ref{thm:BM-classical-realization}
gives
\[
        \dim_{\mathrm H}
        \bigl(
        \Lambda_{\mathcal U}(E)\cap h(V)
        \bigr)
        \ge 1.
\]
Because \(h^{-1}\) is bi-Lipschitz, Hausdorff dimension is preserved under
\(h^{-1}\).  Hence
\[
        \dim_{\mathrm H}
        \{x\in V:\ \omega_{K_h}(x)=E_h\}
        \ge 1.
\]
The equality
\[
        \dim_{\mathrm H}E_h=\dim_{\mathrm H}E
\]
also follows from the bi-Lipschitz invariance of Hausdorff dimension.
\end{proof}

\subsection{A non-Lotka--Volterra weak admissible example with BM-type realization}

We now combine the preceding proposition with the stability of weak
admissibility under small boundary-preserving conjugacy.

Parametrize the boundary \(\partial\Delta^2\) counterclockwise by
\[
        \Gamma:\mathbb R/3\mathbb Z\to\partial\Delta^2,
\]
where
\[
\Gamma(t)=
\begin{cases}
(1-t,t,0), & 0\le t\le1,\\[1mm]
(0,2-t,t-1), & 1\le t\le2,\\[1mm]
(t-2,0,3-t), & 2\le t\le3,
\end{cases}
\]
and \(\Gamma\) is extended periodically with period \(3\).

For \(0<\varepsilon<1\), define the boundary homeomorphism
\[
        g_\varepsilon(\Gamma(t))=\Gamma(t+\varepsilon).
\]
Every point
\(
        x\in\Delta^2\setminus\{\mathbf p\}
\)
has a unique radial representation
\[
        x=(1-r)\mathbf p+rq,
        \qquad
        0<r\le1,
        \qquad
        q\in\partial\Delta^2.
\]
Define
\[
        h_\varepsilon(x)=
   \begin{cases}
        \mathbf p,  & x=\mathbf p,\\[1mm]
(1-r)\mathbf p+r g_\varepsilon(q), & x=(1-r)\mathbf p+rq\ne\mathbf p.
\end{cases}
\]

The boundary map \(g_\varepsilon\) is piecewise affine with finitely many
breakpoints.  Triangulate \(\Delta^2\) by joining \(\mathbf p\) to those
breakpoints and to their images under \(g_\varepsilon\).  On each triangle of
this finite fan, the radial extension \(h_\varepsilon\) is an invertible affine
map onto another nondegenerate triangle.  Hence both \(h_\varepsilon\) and
\(h_\varepsilon^{-1}\) are Lipschitz on each member of a finite triangulation,
and the finite maximum of the corresponding Lipschitz constants gives a global
bi-Lipschitz constant.  Thus \(h_\varepsilon\) is a bi-Lipschitz
boundary-preserving homeomorphism of \(\Delta^2\) onto itself.  Moreover, the
boundary displacement is bounded by the boundary arclength
\(\sqrt2\varepsilon\), and the radial formula gives
\[
        \|h_\varepsilon-\operatorname{id}\|_\infty
        \le \sqrt2\varepsilon.
\]
Consequently,
\[
        h_\varepsilon\to\operatorname{id}
        \qquad
        \text{uniformly as }\varepsilon\to0.
\]
Let
\[
        K_\varepsilon
        =
        h_\varepsilon^{-1}\circ\mathcal U\circ h_\varepsilon.
\]

\begin{theorem}
\label{thm:non-LV-BM-example}
For all sufficiently small
\(
        \varepsilon>0,
\)
the map \(K_\varepsilon\) is a weak admissible Stein--Ulam map.  Moreover,
for every \(\mathcal U\)-admissible compact set
\(E\subset\partial\Delta^2\) with \(\dim_{\mathrm H}E>0\), the transported set
\[
        E_\varepsilon=h_\varepsilon^{-1}(E)
\]
has the Bara\'nski--Misiurewicz omega-limit realization conclusion for
\(K_\varepsilon\): in every nonempty open set there is a Hausdorff-dimension-at
least one family of points whose \(K_\varepsilon\)-omega-limit set is
\(E_\varepsilon\).

In particular, for the Cantor set \(E_C\) from
Proposition~\ref{prop:Cantor-boundary-set}, the set
\(
        E_{C,\varepsilon}
        =
        h_\varepsilon^{-1}(E_C)
\)
satisfies
\[
        \dim_{\mathrm H}E_{C,\varepsilon}
        =
        \frac{\log2}{\log3},
\]
and, for every nonempty open set
\(
        V\subset\Delta^2,
\)
one has
\[
        \dim_{\mathrm H}
        \{x\in V:\ \omega_{K_\varepsilon}(x)=E_{C,\varepsilon}\}
        \ge 1.
\]
Thus \(K_\varepsilon\) has positive-dimensional fractal omega-limit sets on a
Hausdorff-dimension-at-least-one set of initial conditions in every open set.

Moreover, if
\(
        0<\varepsilon<1,
\)
then \(K_\varepsilon\) is not a Lotka--Volterra stochastic map, and therefore is
not a \(\mathbf q\)-Stein--Ulam map for any
\(
        \mathbf q\in\operatorname{int}\Delta^2.
\)
\end{theorem}

\begin{proof}
Since
\(
        h_\varepsilon\to\operatorname{id}
\)
uniformly and \(h_\varepsilon\) preserves the boundary, the stability theorem
for weak admissibility under small boundary-preserving conjugacy implies that
\(K_\varepsilon\) is weak admissible for all sufficiently small
\(\varepsilon>0\).

Let \(E\subset\partial\Delta^2\) be \(\mathcal U\)-admissible with
\(\dim_{\mathrm H} E>0\), and put \(E_\varepsilon=h_\varepsilon^{-1}(E)\).
The omega-limit realization conclusion for \(K_\varepsilon\) over
\(E_\varepsilon\) follows from Proposition~\ref{prop:BM-bilip-conjugacy}, because
\(h_\varepsilon\) is bi-Lipschitz and Hausdorff dimension is preserved
under bi-Lipschitz maps. Applying that proposition to the Cantor admissible set \(E_C\)
gives
\[
        \dim_{\mathrm H}E_{C,\varepsilon}
        =
        \dim_{\mathrm H}E_C
        =
        \frac{\log2}{\log3}
\]
and
\[
        \dim_{\mathrm H}
        \{x\in V:\ \omega_{K_\varepsilon}(x)=E_{C,\varepsilon}\}
        \ge1
\]
for every nonempty open set \(V\subset\Delta^2\).

It remains to prove that \(K_\varepsilon\) is not Lotka--Volterra.  Every
stochastic Lotka--Volterra map has the coordinate form
\[
        (L(x))_i=x_i(1+F_i(x)),
        \qquad i=1,2,3.
\]
Hence every coordinate face is invariant.  In particular, every vertex is fixed:
\[
        L(\mathbf e_i)=\mathbf e_i,
        \qquad i=1,2,3.
\]

We show that \(K_\varepsilon\) does not fix \(\mathbf e_1=(1,0,0)\).  Since
\[
        h_\varepsilon(\mathbf e_1)=\Gamma(\varepsilon)
        =
        (1-\varepsilon,\varepsilon,0),
\]
we compute
\[
\begin{aligned}
        \mathcal U(h_\varepsilon(\mathbf e_1))
        &=
        \mathcal U(1-\varepsilon,\varepsilon,0)        \\
        &=
        (1-\varepsilon^2,\varepsilon^2,0)              \\
        &=
        \Gamma(\varepsilon^2).
\end{aligned}
\]
Since \(h_\varepsilon^{-1}\) acts on the boundary by
\[
        \Gamma(t)\mapsto \Gamma(t-\varepsilon),
\]
we obtain
\[
        K_\varepsilon(\mathbf e_1)
        =
        \Gamma(\varepsilon^2-\varepsilon)
        \qquad
        \text{in }\mathbb R/3\mathbb Z.
\]
For \(0<\varepsilon<1\), this is
\[
        K_\varepsilon(\mathbf e_1)
        =
        (1-\varepsilon+\varepsilon^2,\ 0,\ \varepsilon-\varepsilon^2),
\]
which is different from \(\mathbf e_1\).  Thus \(K_\varepsilon\) does not fix
the vertices and hence cannot be a stochastic Lotka--Volterra map.  Therefore
it cannot be a \(\mathbf q\)-Stein--Ulam map for any interior vector
\(\mathbf q\).
\end{proof}

\newpage
\subsection{A Baire-category caution}
\label{subsec:baire-caution}

It is natural to ask whether weak admissibility is a generic property.  The
answer depends strongly on the ambient space.  In the full space of simplex
self-maps, weak admissibility is not residual.  Indeed, it is not even dense.
This is because weak admissibility contains exact boundary invariance as part of
its definition.

Let
\[
        \mathcal M^r(\Delta^{m-1})
        =
        \{T\in C^r(\Delta^{m-1},\Delta^{m-1})\},
        \qquad r\ge0,
\]
with the \(C^r\) topology, where \(r=0\) denotes the \(C^0\) topology.  Let
\[
        \mathcal W^r(\Delta^{m-1})
\]
denote the set of weak admissible Stein--Ulam type maps belonging to
\(\mathcal M^r(\Delta^{m-1})\).

\begin{proposition}
\label{prop:weak-not-residual-full}
For every finite
\[
        r\ge 0,
\]
the set
\[
        \mathcal W^r(\Delta^{m-1})
\]
is not residual in
\[
        \mathcal M^r(\Delta^{m-1}).
\]
More precisely,
\[
        \mathcal W^r(\Delta^{m-1})
\]
is contained in a closed nowhere dense subset of
\[
        \mathcal M^r(\Delta^{m-1}).
\]
\end{proposition}

\begin{proof}
Let
\[
        \mathcal B^r(\Delta^{m-1})
        =
        \{T\in\mathcal M^r(\Delta^{m-1}):
        T(\partial\Delta^{m-1})\subset\partial\Delta^{m-1}\}.
\]
By condition \textup{(W1)}, every weak admissible Stein--Ulam type map belongs
to
\[
        \mathcal B^r(\Delta^{m-1}).
\]
Thus
\[
        \mathcal W^r(\Delta^{m-1})
        \subset
        \mathcal B^r(\Delta^{m-1}).
\]

We first show that
\[
        \mathcal B^r(\Delta^{m-1})
\]
is closed.  Suppose
\[
        T_n\to T
\]
in the \(C^r\) topology, and
\[
        T_n(\partial\Delta^{m-1})
        \subset
        \partial\Delta^{m-1}
        \qquad\text{for all }n.
\]
Then the convergence is in particular uniform.  Since
\[
        \partial\Delta^{m-1}
\]
is closed, for every
\[
        x\in\partial\Delta^{m-1}
\]
we have
\[
        T(x)=\lim_{n\to\infty}T_n(x)\in\partial\Delta^{m-1}.
\]
Therefore
\[
        T(\partial\Delta^{m-1})
        \subset
        \partial\Delta^{m-1},
\]
so \(T\in\mathcal B^r(\Delta^{m-1})\).  Hence
\(\mathcal B^r(\Delta^{m-1})\) is closed.

We now show that
\[
        \mathcal B^r(\Delta^{m-1})
\]
has empty interior.  Let
\[
        T\in\mathcal B^r(\Delta^{m-1})
\]
and let \(\varepsilon>0\).  Choose the barycenter
\[
        p_0=\left(\frac1m,\ldots,\frac1m\right)
        \in\operatorname{int}\Delta^{m-1}.
\]
For \(0<\delta<1\), define
\[
        T_\delta(x)
        =
        (1-\delta)T(x)+\delta p_0,
        \qquad x\in\Delta^{m-1}.
\]
Since \(\Delta^{m-1}\) is convex,
\[
        T_\delta(\Delta^{m-1})
        \subset
        \Delta^{m-1}.
\]
Moreover,
\[
        T_\delta\to T
\]
in the \(C^r\) topology as \(\delta\to0\).  Hence, for \(\delta>0\) sufficiently
small,
\[
        \|T_\delta-T\|_{C^r}<\varepsilon.
\]

However, \(T_\delta\) does not preserve the boundary.  Indeed, if
\[
        x\in\partial\Delta^{m-1},
\]
then
\[
        T(x)\in\partial\Delta^{m-1},
\]
but
\[
        T_\delta(x)
        =
        (1-\delta)T(x)+\delta p_0
        \in\operatorname{int}\Delta^{m-1},
\]
because it is a non-trivial convex combination of a point of the simplex with an
interior point.  Therefore
\[
        T_\delta(\partial\Delta^{m-1})
        \not\subset
        \partial\Delta^{m-1}.
\]
Thus every \(C^r\)-neighborhood of \(T\) contains maps outside
\[
        \mathcal B^r(\Delta^{m-1}).
\]
So \(\mathcal B^r(\Delta^{m-1})\) has empty interior.

Since \(\mathcal B^r(\Delta^{m-1})\) is closed and has empty interior, it is
nowhere dense.  Because
\[
        \mathcal W^r(\Delta^{m-1})
        \subset
        \mathcal B^r(\Delta^{m-1}),
\]
we obtain
\[
        \overline{\mathcal W^r(\Delta^{m-1})}
        \subset
        \mathcal B^r(\Delta^{m-1}),
\]
and hence
\[
        \operatorname{int}
        \overline{\mathcal W^r(\Delta^{m-1})}
        =
        \varnothing.
\]
Therefore
\[
        \mathcal W^r(\Delta^{m-1})
\]
is nowhere dense.  In particular, it cannot be residual in the Baire space
\[
        \mathcal M^r(\Delta^{m-1}).
\]
\end{proof}

\subsection{Relative Baire genericity in a cyclic-collar structural class}
\label{subsec:relative-baire-genericity}

We now formulate a relative Baire-category result.  As noted above, weak
admissibility should not be expected to be residual in the full space
\(C^r(\DeltaM,\DeltaM)\), because the full space does not preserve the boundary,
the Lyapunov geometry or the cyclic residence structure.  The correct generic
question is relative: once the boundary collar and the cyclic geometry are fixed,
how large is the set of speed profiles that produce non-statistical behavior?

The result below gives a precise answer in a fixed cyclic-collar class.  In that
class, non-statistical behavior is residual and dense, while regular behavior is
also dense.  Thus the two behaviors are intermingled inside the structural class.

Fix \(m\ge4\).  Let
\[
        \Xi:[s_0,\infty]\times\mathbb T\times
        \overline{\mathbb B}^{m-3}\longrightarrow \DeltaM,
        \qquad
        \mathbb T=\mathbb R/\mathbb Z,
\]
be a fixed cyclic boundary collar.  Here \(s=\infty\) corresponds to the
boundary, while finite \(s\ge s_0\) corresponds to interior points.  We assume
that
\[
        \Xi\bigl([s_0,\infty)\times\mathbb T\times
        \mathbb B^{m-3}\bigr)
        \subset \inte\DeltaM
\]
has positive \((m-1)\)-dimensional Lebesgue measure, and that the boundary limit
\[
        H(t,w)=\Xi(\infty,t,w)
\]
exists continuously.  We also assume that the angular windows
\[
        J_j=
        \left[\frac{j-1}{m},\frac j m\right],
        \qquad j=1,\ldots,m,
\]
define residence regions
\[
        R_j=
        \Xi\bigl([s_0,\infty]\times J_j\times
        \overline{\mathbb B}^{m-3}\bigr)
\]
satisfying the convex-hull separation
\[
        \bigcap_{j=1}^m \co(R_j)=\varnothing.
        \tag{16.1}\label{eq:collar-convhull-separation}
\]

Fix constants
\[
        c>0,\qquad 0<\beta<1,
\]
and choose \(A>0\) so small that
\[
        \frac{A}{s_0}\le \frac1N
\]
for some integer \(N\ge 2m\) divisible by \(m\).  For a function
\[
        a:[s_0,\infty)\to[0,A],
\]
define a map on the collar by
\[
        K_a\bigl(\Xi(s,t,w)\bigr)
        =
        \Xi\left(
        s+c,\,
        t+\frac{a(s)}{s},\,
        \beta w
        \right),
        \qquad s<\infty.
\]
On the boundary part \(s=\infty\), define
\[
        K_a\bigl(\Xi(\infty,t,w)\bigr)
        =
        \Xi(\infty,t,\beta w).
\]
The inequality \(a(s)/s\le 1/N\) ensures that one step of the angular coordinate
can cross at most one fine sector.  Thus the weak cyclic rule is built into the
structural class.

For \(\nu\in\mathbb N_0\), let
\[
        \mathcal A_A^\nu
        =
        \{a\in C^\nu_{\mathrm{loc}}([s_0,\infty)):\ 0\le a(s)\le A
        \text{ for all }s\ge s_0\}.
\]
We equip \(\mathcal A_A^\nu\) with the local \(C^\nu\) metric
\[
        d_\nu(a,b)
        =
        \sum_{q=1}^\infty
        2^{-q}
        \min\left\{
        1,\,
        \|a-b\|_{C^\nu([s_0,s_0+q])}
        \right\}.
\]
This is a Baire space.

For \(a\in\mathcal A_A^\nu\) and \(\theta\in[0,c]\), define
\[
        S_a(\theta)
        =
        \sum_{n=0}^{\infty}
        \frac{a(s_0+\theta+nc)}{s_0+\theta+nc}.
\]
We introduce two subclasses:
\[
        \mathcal N_A^\nu
        =
        \{a\in\mathcal A_A^\nu:\ S_a(\theta)=+\infty
        \text{ for every }\theta\in[0,c]\},
\]
and
\[
        \mathcal R_A^\nu
        =
        \{a\in\mathcal A_A^\nu:\ S_a(\theta)<+\infty
        \text{ for every }\theta\in[0,c]\}.
\]
The class \(\mathcal N_A^\nu\) is the non-statistical class, while
\(\mathcal R_A^\nu\) is the regular class.

\begin{theorem}
\label{thm:relative-generic-collar}
In the cyclic-collar structural class described above, the following assertions
hold.

\begin{enumerate}
\item[\textup{(i)}]
If
\[
        a\in\mathcal N_A^\nu,
\]
then every orbit starting in the collar interior
\[
        \Xi\bigl([s_0,\infty)\times\mathbb T\times
        \mathbb B^{m-3}\bigr)
\]
is non-statistical.  More precisely, its empirical measures do not converge, and
its vector Ces\`aro averages do not converge.

\item[\textup{(ii)}]
If
\[
        a\in\mathcal R_A^\nu,
\]
then every orbit starting in the collar interior is regular: the orbit converges
to a boundary point, and the empirical measures converge to the corresponding
Dirac measure.

\item[\textup{(iii)}]
The non-statistical class
\[
        \mathcal N_A^\nu
\]
is residual and dense in
\[
        \mathcal A_A^\nu.
\]

\item[\textup{(iv)}]
The regular class
\[
        \mathcal R_A^\nu
\]
is dense in
\[
        \mathcal A_A^\nu.
\]
Consequently, regular and non-statistical behaviors are intermingled in the
cyclic-collar structural class: every neighborhood of every profile contains
profiles of both types.
\end{enumerate}
\end{theorem}

\begin{proof}
Let
\[
        x_0=\Xi(s,t,w)
\]
be a point in the collar interior.  Its orbit is given explicitly by
\[
        s_n=s+nc,
\]
\[
        w_n=\beta^n w,
\]
and
\[
        t_n
        =
        t+
        \sum_{k=0}^{n-1}
        \frac{a(s+kc)}{s+kc}
        \quad \text{in }\mathbb T.
\]

We first prove \textup{(i)}.  Suppose
\[
        a\in\mathcal N_A^\nu.
\]
Writing \(s=s_0+\theta+q c\) for some \(\theta\in[0,c]\) and \(q\in\mathbb N_0\),
we get
\[
        \sum_{k=0}^{\infty}
        \frac{a(s+kc)}{s+kc}
        =
        \sum_{k=q}^{\infty}
        \frac{a(s_0+\theta+kc)}{s_0+\theta+kc}
        =
        +\infty.
\]
Thus the lifted angular coordinate is unbounded.  Since
\[
        0\le \frac{a(s)}{s}\le \frac1N,
\]
the orbit can only remain in the same fine sector or move to the next one.
Hence the weak cyclic rule and the non-trapping condition hold.

Moreover, if an orbit enters a residence window \(J_j\) at height \(s_{\rm in}\)
and we consider the corresponding maximal lifted residence block, then during
the whole block one has
\[
        \frac{a(s_n)}{s_n}\le \frac{A}{s_{\rm in}}.
\]
Since each \(J_j\) has angular length \(1/m\), the residence length \(\ell\)
satisfies, up to a uniform entry--exit error,
\[
        \ell
        \ge
        \frac{s_{\rm in}}{2mA}
\]
for all sufficiently large \(s_{\rm in}\).  If we take the Lyapunov coordinate
\[
        \varphi(\Xi(s,t,w))=e^{-s},
\]
then
\[
        s_{\rm in}=\log\frac1{\varphi(\Xi(s_{\rm in},t_{\rm in},w_{\rm in}))}.
\]
Therefore the logarithmic residence estimate holds:
\[
        \ell
        \ge
        A_j\log_{C_j}^{+}
        \frac{B_j}{\varphi(K_a^{b}(x_0))}
\]
with suitable constants \(A_j,B_j>0\) and \(C_j>1\), where \(b\) is the first
time of the block.  In fact the estimate is linear in the previous history: if
the block starts at time \(b\), then
\[
        s_{\rm in}=s+bc,
\]
and hence, for all large \(b\),
\[
        \frac{\ell}{b}
        \ge
        \frac{c}{4mA}>0.
\]

The convex-hull separation
\[
        \bigcap_{j=1}^m\co(R_j)=\varnothing
\]
is fixed by the structural class.  Hence the standard convex-hull block lemma
applies: if the vector Ces\`aro averages had a limit \(L\), then the long
residence blocks through every \(R_j\) would force
\[
        L\in\co(R_j),
        \qquad j=1,\ldots,m,
\]
contradicting \eqref{eq:collar-convhull-separation}.  Thus the vector Ces\`aro
averages do not converge.  Weak-\(*\) convergence of empirical measures would
imply convergence of vector Ces\`aro averages by testing against the identity
observable, so the empirical measures do not converge either.

We now prove \textup{(ii)}.  Suppose
\[
        a\in\mathcal R_A^\nu.
\]
Then for every initial value \(s\ge s_0\),
\[
        \sum_{k=0}^{\infty}
        \frac{a(s+kc)}{s+kc}<\infty.
\]
Hence \(t_n\) converges in \(\mathbb T\) to some limit \(t_\infty\).  Also
\[
        s_n\to\infty,
        \qquad
        w_n=\beta^n w\to0.
\]
Therefore
\[
        K_a^n(x_0)
        =
        \Xi(s_n,t_n,w_n)
        \longrightarrow
        \Xi(\infty,t_\infty,0).
\]
Thus the orbit is regular and the empirical measures converge to
\[
        \delta_{\Xi(\infty,t_\infty,0)}.
\]

It remains to prove the Baire-category statements.

For \(M\in\mathbb N\), define
\[
        \mathcal U_M
        =
        \bigcup_{N_0\ge1}
        \left\{
        a\in\mathcal A_A^\nu:
        \min_{\theta\in[0,c]}
        \sum_{n=0}^{N_0}
        \frac{a(s_0+\theta+nc)}{s_0+\theta+nc}
        >M
        \right\}.
\]
Each \(\mathcal U_M\) is open in the local \(C^\nu\) topology, because the
condition involves only values of \(a\) on a compact interval and is strict.

We show that \(\mathcal U_M\) is dense.  Let \(a\in\mathcal A_A^\nu\) and let
a basic neighborhood of \(a\) be prescribed by a compact interval
\[
        [s_0,s_0+Q]
\]
and a tolerance \(\varepsilon>0\).  Choose a smooth cut-off function
\[
        \chi:[s_0,\infty)\to[0,1]
\]
such that
\[
        \chi(s)=0
        \quad\text{on }[s_0,s_0+Q],
\]
and
\[
        \chi(s)=1
        \quad\text{for all sufficiently large }s.
\]
Choose \(0<\delta<A\), and define
\[
        \widetilde a(s)
        =
        a(s)+
        \delta\,\chi(s)\frac{A-a(s)}{A}.
\]
Then
\[
        0\le \widetilde a(s)\le A,
\]
and
\[
        \widetilde a=a
        \quad\text{on }[s_0,s_0+Q].
\]
Thus \(\widetilde a\) belongs to the prescribed neighborhood of \(a\).  For all
sufficiently large \(s\), we have \(\chi(s)=1\), and therefore
\[
        \widetilde a(s)
        =
        a(s)+\delta\frac{A-a(s)}{A}
        \ge
        \min\left\{\delta,\frac A2\right\}
\]
unless \(a(s)\ge A/2\), in which case the same lower bound holds with \(A/2\).
Hence there exists \(d>0\) such that
\[
        \widetilde a(s)\ge d
\]
for all sufficiently large \(s\).  Therefore, uniformly in \(\theta\in[0,c]\),
\[
        \sum_{n=0}^{N}
        \frac{\widetilde a(s_0+\theta+nc)}{s_0+\theta+nc}
        \longrightarrow+\infty
        \qquad
        \text{as }N\to\infty.
\]
Thus, for some \(N_0\),
\[
        \min_{\theta\in[0,c]}
        \sum_{n=0}^{N_0}
        \frac{\widetilde a(s_0+\theta+nc)}{s_0+\theta+nc}
        >M.
\]
Hence \(\widetilde a\in\mathcal U_M\).  This proves that
\[
        \mathcal U_M
\]
is dense.

Now set
\[
        \mathcal G
        =
        \bigcap_{M=1}^{\infty}\mathcal U_M.
\]
Then \(\mathcal G\) is residual in \(\mathcal A_A^\nu\).  By construction, every
\(a\in\mathcal G\) satisfies
\[
        \sum_{n=0}^{\infty}
        \frac{a(s_0+\theta+nc)}{s_0+\theta+nc}
        =
        +\infty
        \qquad
        \text{for every }\theta\in[0,c].
\]
Therefore
\[
        \mathcal G\subset\mathcal N_A^\nu.
\]
Since \(\mathcal G\) is residual and
\[
        \mathcal G\subset\mathcal N_A^\nu,
\]
the complement of \(\mathcal N_A^\nu\) is contained in the meagre set
\(\mathcal A_A^\nu\setminus\mathcal G\).  Hence
\[
        \mathcal N_A^\nu
\]
is residual and dense in \(\mathcal A_A^\nu\).  This proves \textup{(iii)}.

Finally, we show that the regular class \(\mathcal R_A^\nu\) is dense.  Let
\(a\in\mathcal A_A^\nu\), let \(Q>0\), and let \(\varepsilon>0\) be given.
Choose a smooth cut-off function
\[
        \zeta:[s_0,\infty)\to[0,1]
\]
such that
\[
        \zeta(s)=1
        \quad\text{on }[s_0,s_0+Q],
\]
and
\[
        \zeta(s)=0
        \quad\text{for all sufficiently large }s.
\]
Define
\[
        a_{\mathrm{reg}}(s)=\zeta(s)a(s).
\]
Then
\[
        a_{\mathrm{reg}}\in\mathcal A_A^\nu,
\]
and
\[
        a_{\mathrm{reg}}=a
        \quad\text{on }[s_0,s_0+Q].
\]
Thus \(a_{\mathrm{reg}}\) lies in the prescribed neighborhood of \(a\).  Since
\(a_{\mathrm{reg}}\) has compact support in \([s_0,\infty)\), for every
\(\theta\in[0,c]\) the series
\[
        \sum_{n=0}^{\infty}
        \frac{a_{\mathrm{reg}}(s_0+\theta+nc)}{s_0+\theta+nc}
\]
has only finitely many non-zero terms.  Hence
\[
        a_{\mathrm{reg}}\in\mathcal R_A^\nu.
\]
This proves that \(\mathcal R_A^\nu\) is dense.

Since both \(\mathcal N_A^\nu\) and \(\mathcal R_A^\nu\) meet every non-empty
open set in the structural topology, regular and non-statistical behaviors are
intermingled inside the cyclic-collar class.
\end{proof}

\begin{remark}
Theorem~\ref{thm:relative-generic-collar} should be compared with the
intermingling phenomenon for one-dimensional intermittent maps.  In the full
space of all simplex maps, weak admissibility is not a residual property because
the full space does not preserve the structural objects needed to formulate the
mechanism.  Inside the fixed cyclic-collar class, however, the boundary, the
Lyapunov coordinate, the residence regions and the cyclic order are built into
the ambient space.  In that relative topology, non-statistical behavior is
residual, while regular behavior is still dense.  Thus the two behaviors are
intermingled in precisely the structural sense relevant to the weak
Stein--Ulam mechanism.
\end{remark}

\section*{Declaration}
The author declares that he has no conflict of interest.  AI tools were used to
assist with language editing and manuscript preparation.  The mathematical
content, proofs and conclusions are the author's responsibility.

\section*{Data Availability Statement}
No data sets were generated or analyzed during the current study.

\end{document}